\documentclass[11pt]{amsart}

\usepackage{
    amsmath,
    amsfonts,
    amssymb,
    amsthm,
    amscd,
    comment,
    enumitem,
    etoolbox,
    gensymb,    % For \degree
    mathtools,
    mathdots,
    stmaryrd,
    booktabs,
    mathrsfs
}
\usepackage[usenames,dvipsnames]{xcolor}
\usepackage[all]{xy}

\usepackage[T1]{fontenc}
\usepackage{bbm}                     % For \mathbbm{1} (unit in monoidal category)
\usepackage[colorlinks=true, linkcolor=blue, citecolor=blue, urlcolor=blue, breaklinks=true]{hyperref}

\DeclareFontFamily{OT1}{pzc}{}
\DeclareFontShape{OT1}{pzc}{m}{it}{<-> s * [1.10] pzcmi7t}{}
\DeclareMathAlphabet{\mathpzc}{OT1}{pzc}{m}{it}

\usepackage[capitalize]{cleveref}   % Can add 'nameinlink' option to have the name included in hyperlink

\crefname{defin}{Definition}{Definitions}
\crefname{eg}{Example}{Examples}
\crefname{egs}{Examples}{Examples}
\crefname{convention}{Convention}{Convention}
\crefname{lem}{Lemma}{Lemmas}
\crefname{prop}{Proposition}{Propositions}
\crefname{theo}{Theorem}{Theorems}
\crefname{equation}{}{}
\crefname{enumi}{}{}
\crefrangeformat{equation}{(#3#1#4--#5#2#6)}
\newcommand\C{\mathbb{C}}
\newcommand\N{\mathbb{N}}
\newcommand\OO{\mathbb{O}}
\newcommand\lOO{\widetilde{\mathbf{O}}}
\newcommand\rOO{\mathbf{O}}

\newcommand\Z{\mathbb{Z}}
\newcommand\kk{\Bbbk}
\newcommand\one{\mathbbm{1}}

\newcommand\ba{\mathbf{a}}

\newcommand\op{\mathrm{op}}
\newcommand\rev{\mathrm{rev}}

\newcommand\Laurent[1]{(\! ( #1 )\! )}  % Laurent series
\newcommand\Taylor[1]{\llbracket #1 \rrbracket}

\newcommand\AB{\mathpzc{AB}}            % Affine Brauer category
\newcommand\AK{\mathpzc{AK}}            % Affine Kauffman skein category
\newcommand\cA{\mathcal{A}}
\newcommand\CB{\mathpzc{CB}}            % Cyclotomic Brauer category
\newcommand\CK{\mathpzc{CK}}            % Cyclotomic Kauffman category
\newcommand\cC{\mathcal{C}}
\newcommand\cEnd{\mathpzc{End}}
\newcommand\cR{\mathcal{R}}

\newcommand\cI{\mathcal{I}}             % Tensor ideal in affine Brauer category
\newcommand\cJ{\mathcal{J}}             % Tensor ideal in affine Kauffman category

\newcommand\go{{\mathsf{B}}}            % Self-dual generating object
\newcommand\gok{{\mathsf{K}}}           % Generating object for affine Kauffman skein category

\DeclareMathOperator{\End}{End}

\DeclareMathOperator{\id}{id}

\usepackage{tikz}
\usetikzlibrary{arrows.meta}
\usetikzlibrary{decorations.markings,decorations.pathreplacing}
\usetikzlibrary{calc}
\usepackage{tikz-cd}

\tikzset{anchorbase/.style={>=To,baseline={([yshift=-0.5ex]current bounding box.center)}}}
\tikzset{wipe/.style={white,line width=3pt}}
\tikzset{ % Syntax: \begin{tikzpicture}[centerzero={0,0.2}]
    centerzero/.style={>=To,baseline={([yshift=-0.5ex](#1))}},
    centerzero/.default={0,0}
}
\tikzset{module/.style={green!70!black, thick}}
\tikzset{alg/.style={very thick,blue}}

\newcommand\braidup{to[out=up,in=down]}
\newcommand\braiddown{to[out=down,in=up]}

\newcommand\dotlabel[1]{$\scriptstyle{#1}$}

\newcommand\strandlabel[1]{$\scriptstyle{#1}$}
\newcommand\botlabel[1]{node[anchor=north] {\strandlabel{#1}}}
\newcommand\toplabel[1]{node[anchor=south] {\strandlabel{#1}}}
\newcommand{\coupon}[3][0.15]{% \coupon[size]{position}{label}
    \filldraw[draw=black,fill=white] (#2) circle (#1);
    \node at (#2) {$\scriptscriptstyle{#3}$}
}

\newcommand\singdot[2][white]{% \singdot[colour]{position}
    \filldraw[fill=#1, draw=black] (#2) circle (1.5pt)
}
\newcommand\multdot[4][white]{% \multdot[colour]{position}{anchor}{label}
    \filldraw[fill=#1, draw=black] (#2) circle (1.5pt) node[anchor=#3] {\dotlabel{#4}}
}
\newcommand\uptriforce[2][white]{% \triforce[colour]{position}
    \filldraw[fill=#1, draw=black] (#2)++(0,0.075) to ++(0.06,-0.15) to ++(-0.12,0) -- cycle
}

\newcommand\downtriforce[2][white]{% \triforce[colour]{position}
    \filldraw[fill=#1, draw=black] (#2)++(0,-0.075) to ++(0.06,0.15) to ++(-0.12,0) -- cycle
}

\newcommand\bub[1]{% \bub{position}
    \draw (#1)++(0,0.2) arc(90:-270:0.2)
}
\newcommand\multbubr[3][white]{% \multbubr[dot colour]{position}{# dots}
    \draw (#2)++(0,0.2) arc(90:-270:0.2);
    \filldraw[fill=#1, draw=black] (#2)++(+0.2,0) circle (1.5pt) node[anchor=west] {\dotlabel{#3}}
}

\newcommand\bubgenr[3][white]{% \bubgenr[dot colour]{label}{position}
    \draw (#3)++(0,0.2) arc(90:-270:0.2);
    \node at (#3) {\dotlabel{#2}};
    \filldraw[fill=#1, draw=black] (#3)++(0.2,0) circle (1.5pt)
}
\newcommand\bubgenl[3][white]{% \bubgenl[dot colour]{label}{position}
    \draw (#3)++(0,0.2) arc(90:-270:0.2);
    \node at (#3) {\dotlabel{#2}};
    \filldraw[fill=#1, draw=black] (#3)++(-0.2,0) circle (1.5pt)
}

\newcommand\uptribubr[2][white]{% \uptribubr[colour]{position}
    \bub{#2};
    \filldraw[fill=#1, draw=black] (#2)++(0.2,0.075) to ++(0.06,-0.15) to ++(-0.12,0) -- cycle
}
\newcommand\uptribubl[2][white]{% \uptrbubl[colour]{position}
    \bub{#2};
    \filldraw[fill=#1, draw=black] (#2)++(-0.2,0.075) to ++(0.06,-0.15) to ++(-0.12,0) -- cycle
}
\newcommand\downtribubr[2][white]{% \downtribubr[colour]{position}
    \bub{#2};
    \filldraw[fill=#1, draw=black] (#2)++(0.2,-0.075) to ++(0.06,0.15) to ++(-0.12,0) -- cycle
}
\newcommand\downtribubl[2][white]{% \downtribubl{position}
    \bub{#2};
    \filldraw[fill=#1, draw=black] (#2)++(-0.2,-0.075) to ++(0.06,0.15) to ++(-0.12,0) -- cycle
}

\newcommand\bubble{%
    \begin{tikzpicture}[centerzero]
        \bub{0,0};
    \end{tikzpicture}
}
\newcommand\multbubbler[2][white]{% \multbubbler[dot colour]{token}
    \begin{tikzpicture}[centerzero]
        \multbubr[#1]{0,0}{#2};
    \end{tikzpicture}
}

\newcommand\bubblegenr[2][white]{% \cbubblegen[dot colour]{label}
    \begin{tikzpicture}[centerzero]
        \bubgenr[#1]{#2}{0,0};
    \end{tikzpicture}
}
\newcommand\bubblegenl[2][white]{% \cbubblegen[dot colour]{label}
    \begin{tikzpicture}[centerzero]
        \bubgenl[#1]{#2}{0,0};
    \end{tikzpicture}
}

\newcommand\uptribubbler[1][white]{%
    \begin{tikzpicture}[centerzero]
        \uptribubr[#1]{0,0};
    \end{tikzpicture}
}
\newcommand\uptribubblel[1][white]{%
    \begin{tikzpicture}[centerzero]
        \uptribubl[#1]{0,0};
    \end{tikzpicture}
}
\newcommand\downtribubbler[1][white]{%
    \begin{tikzpicture}[centerzero]
        \downtribubr[#1]{0,0};
    \end{tikzpicture}
}
\newcommand\downtribubblel[1][white]{%
    \begin{tikzpicture}[centerzero]
        \downtribubl[#1]{0,0};
    \end{tikzpicture}
}

\newcommand\idstrand{
    \begin{tikzpicture}[centerzero]
        \draw[-] (0,-0.2) -- (0,0.2);
    \end{tikzpicture}
}
\newcommand\dotstrand[1][white]{    % \dotstrand[colour]
    \begin{tikzpicture}[centerzero]
        \draw (0,-0.2) -- (0,0.2);
        \singdot[#1]{0,0};
    \end{tikzpicture}
}
\newcommand\multdotstrand[3][white]{% \multdotstrand[colour]{anchor}{label}
    \begin{tikzpicture}[centerzero]
        \draw (0,-0.2) -- (0,0.2);
        \multdot[#1]{0,0}{#2}{#3};
    \end{tikzpicture}
}
\newcommand\crossmor{
    \begin{tikzpicture}[centerzero]
        \draw[-] (0.2,-0.2) -- (-0.2,0.2);
        \draw[-] (-0.2,-0.2) -- (0.2,0.2);
    \end{tikzpicture}
}
\newcommand\poscross{
    \begin{tikzpicture}[centerzero]
        \draw[-] (0.2,-0.2) -- (-0.2,0.2);
        \draw[wipe] (-0.2,-0.2) -- (0.2,0.2);
        \draw[-] (-0.2,-0.2) -- (0.2,0.2);
    \end{tikzpicture}
}
\newcommand\negcross{
    \begin{tikzpicture}[centerzero]
        \draw[-] (-0.2,-0.2) -- (0.2,0.2);
        \draw[wipe] (0.2,-0.2) -- (-0.2,0.2);
        \draw[-] (0.2,-0.2) -- (-0.2,0.2);
    \end{tikzpicture}
}
\newcommand{\cupmor}{
    \begin{tikzpicture}[anchorbase]
        \draw[-] (-0.15,0.15) -- (-0.15,0) arc(180:360:0.15) -- (0.15,0.15);
    \end{tikzpicture}
}
\newcommand{\capmor}{
    \begin{tikzpicture}[anchorbase]
        \draw[-] (-0.15,-0.15) -- (-0.15,0) arc(180:0:0.15) -- (0.15,-0.15);
    \end{tikzpicture}
}

\newtheorem{theo}{Theorem}[section]

\newtheorem{prop}[theo]{Proposition}
\newtheorem{lem}[theo]{Lemma}
\newtheorem{cor}[theo]{Corollary}

\theoremstyle{definition}
\newtheorem{defin}[theo]{Definition}
\newtheorem{rem}[theo]{Remark}

\newtheorem{convention}[theo]{Convention}

\numberwithin{equation}{section}
\allowdisplaybreaks

\setenumerate[1]{label=(\alph*)}          % Only need to reference enumerate items with \ref{}

\newtoggle{comments}
\newtoggle{details}
\newtoggle{detailsnote}

\iftoggle{comments}{%
    \usepackage[notref,notcite]{showkeys}   % To show labels.
    \newcommand{\acomments}[1]{
        \ \\
        {\color{red}
            \textbf{AS:} #1
        }
        \ \\
    }
    \newcommand{\bcomments}[1]{
        \ \\
        {\color{purple}
            \textbf{BW:} #1
        }
        \ \\
    }
}{%
    \newcommand{\acomments}[1]{}
    \newcommand{\bcomments}[1]{}
}

\iftoggle{details}{%
    \newcommand{\details}[1]{
        \ \\
        {\color{OliveGreen}
            \textbf{Details:} #1
        }
        \\
    }
}{%
    \newcommand{\details}[1]{}
}

\begin{document}
%===============

\title{Bubbles in the affine Brauer and Kauffman categories}

\author{Alistair Savage}
\address[A.S.]{
  Department of Mathematics and Statistics \\
  University of Ottawa \\
  Ottawa, ON, Canada
}
\urladdr{\href{https://alistairsavage.ca}{alistairsavage.ca}, \textrm{\textit{ORCiD}:} \href{https://orcid.org/0000-0002-2859-0239}{orcid.org/0000-0002-2859-0239}}
\email{alistair.savage@uottawa.ca}

\author{Ben Webster}
\address[B.W.]{
    Department of Pure Mathematics, University of Waterloo \&
    Perimeter Institute for Theoretical Physics\\
    Waterloo, ON, Canada
}
\email{ben.webster@uwaterloo.ca}

\begin{abstract}
    We introduce a generating function approach to the affine Brauer and Kauffman categories and show how it allows one to efficiently recover important sets of relations in these categories.  We use this formalism to deduce restrictions on possible categorical actions and show how this recovers admissibility results that have appeared in the literature on cyclotomic Birman--Murakami--Wenzl (BMW) algebras and their degenerate versions, also known as cyclotomic Nazarov--Wenzl algebras or VW algebras.
\end{abstract}



\ifboolexpr{togl{comments} or togl{details}}{%
  {\color{magenta}DETAILS OR COMMENTS ON}
}{%
}

\maketitle
\thispagestyle{empty}

%\tableofcontents

%=====================
\section{Introduction}
%=====================

The \emph{affine Brauer category} $\AB$, introduced in \cite{RS19}, and the \emph{affine Kauffman category} $\AK$, introduced in \cite{GRS22}, are diagrammatic monoidal categories related to the invariant theory of the orthogonal and symplectic groups (or, more generally, the orthosymplectic supergroups) and their quantum analogues.  Both categories are generated by a single object, denoted $\go$ or $\gok$, with identity morphism denoted by an unoriented string.  The morphisms are generated by a cup, cap, dot, and crossing, subject to relations given in \cref{brauer,dotmoves} for $\AB$ and in \cref{braid,drops,skein,kauffdot} for $\AK$.
The resulting morphism spaces are spanned by diagrams like
\begin{equation}\label{example-diagram}
    \begin{tikzpicture}[centerzero]
        \draw (0,-0.6) \braidup (0.3,0.6);
        \draw (0.3,-0.6) -- (0.3,-0.4) to[out=up,in=up,looseness=2] (0.6,-0.4) -- (0.6,-0.6);
        \draw (0.9,-0.6) \braidup (1.5,0.6);
        \draw (1.2,-0.6) \braidup (0.6,0.4) -- (0.6,0.6);
        \draw (0,0.6) -- (0,0.4) to[out=down,in=down,looseness=1.5] (0.9,0.4) -- (0.9,0.6);
        \draw (1.2,0.6) -- (1.2,0.4) to[out=down,in=down,looseness=1.5] (1.8,0.4) -- (1.8,0.6);
        \singdot{0,0.4};
        \singdot{0.6,0.4};
        \singdot{1.2,0.4};
    \end{tikzpicture}
    \ .
\end{equation}
This particular diagram represents a morphism $\go^5 \to \go^7$, corresponding to the number of open ends on the bottom and top of the diagram.

The purpose of the current paper is to introduce a generating function formalism for these categories.  A similar technique was introduced in \cite{BSW-HKM} for the Heisenberg category and its quantum analogue, which include the affine \emph{oriented} Brauer category and the affinization of the HOMFLY-PT skein category as special cases.  The generating function approach was used there to develop a precise relationship between those categories and Kac--Moody 2-categories.  In the current paper, this approach allows one to simplify many arguments that have appeared in the study of cyclotomic Birman--Murakami--Wenzl (BMW) algebras and their degenerate analogues.  Cyclotomic BMW algebras were introduced in \cite{Har01} as quotients of BMW algebras.  The degenerate analogues, also known as cyclotomic Nazarov--Wenzl algebras, were introduced in \cite{AMR06} as quotients of the affine Brauer algebras defined in \cite{Naz96} (originally called \emph{affine Wenzl algebras} there).  In \cite{ES18}, the terminology \emph{affine/cyclotomic VW algebras} is used for these degenerate versions.

For the purpose of this introduction, we focus on the affine Brauer category, since our treatment of the affine Kauffman category is parallel.  As the name suggests, the affine oriented Brauer category is spanned by diagrams much like those in \cref{example-diagram}, but equipped with orientations.  This difference arises from the fact that the generating object in the affine Brauer category is self-dual, whereas the generating object in the affine oriented Brauer category is not.  Frequently, we can guess how results in the affine oriented and unoriented Brauer categories are related by considering an unoriented strand as a combination of both possible orientations---in this paper, we will not try to make this observation precise, but it is a useful guide to what to expect.

In the affine oriented Brauer category or, more generally, the Heisenberg category, the endomorphisms of the unit object form a polynomial ring in infinitely many variables---the clockwise dotted bubbles and the counterclockwise dotted bubbles are both sets of free generators for this polynomial ring.  In fact, it is combinatorially more natural  to identify this ring with symmetric functions, sending the clockwise and counterclockwise bubbles to the elementary and complete symmetric functions (up to sign).  These are related by the so-called \emph{infinite grassmannian relation}, which expresses the clockwise dotted bubbles in terms of the counterclockwise ones, or vice versa.

Following the principle articulated above, one expects that the endomorphism ring of the identity in the affine Brauer category is generated by dotted bubbles, but that these are not algebraically independent, since they are related to both clockwise and counterclockwise bubbles.  Indeed, in the affine Brauer category, the bubbles with an even number of dots freely generate the endomorphism algebra as a polynomial ring and the bubbles with odd numbers of dots can be written in terms of them by \cite[Th. ~B \& Lem.~3.4]{RS19}.

For those not used to working with degenerate cyclotomic BMW algebras, the algebraic dependencies can appear quite strange.  Thus, in the affine Brauer category, we propose a different generating set for the endomorphism ring of the identity functor, which we think of as the coefficients of a power series
\[
    \bubblegenr{u}
    = 1 + u^{-1} \bubble{} + u^{-2} \left( \multbubbler{}\!\!\!+ \tfrac{1}{2}\bubble{} \right) + u^{-3} \left( \multbubbler{2}\! + \tfrac{1}{2}\multbubbler{}\!\!\! + \tfrac{1}{4}\bubble{} \right) + \dotsb.
\]
(See \cref{eatz} for the precise definition.)  Using this set of generators simplifies calculations in the affine Brauer category considerably.  In particular, the dependence of odd-degree bubbles on even ones can be expressed simply by
\[
    \begin{tikzpicture}[centerzero]
        \draw (0.3,0) arc(0:360:0.3);
        \node at (0,0) {\dotlabel{-u}};
        \singdot{0.3,0};
    \end{tikzpicture}
    \ \bubblegenr{u}
    = 1.
\]
Thus, in any representation where all bubbles act by scalars, there are algebraic restrictions on the scalars by which they can act.  In particular, if $L$ is any object in a module category and $\End(L)=\kk$ is a field, then $\bubblegenr{u}$ must act by a power series $\OO_L(u) \in \kk \llbracket u^{-1} \rrbracket$.  One of our main results is \cref{hemlock}, which states that
\[
    \OO_L(u) = \frac{\left( (-1)^{\deg m} u - \frac{1}{2} \right) m(-u)}{(u-\frac{1}{2}) m(u)} \in 1 + u^{-1} \kk \llbracket u^{-1} \rrbracket,
\]
where $m$ is the minimal polynomial of the dot.  This is equivalent to the ``admissibility'' conditions that have appeared in the literature on degenerate cyclotomic BMW algebras \cite{AMR06,Goo11}; see \cref{bench}.  In particular, it is clear from existing nondegeneracy results in the literature (in particular, \cite[Th. C]{RS19}) that this is the only possible expression for $\OO_L(u)$; however, we are able to give a simple direct proof using power series, without using any explicit calculation of bases or construction of representations.  \Cref{backhome} gives an analogous result for the affine Kauffman category with a similar proof.    We feel that \cref{hemlock,backhome} give a particularly concise and elegant formulation of the concept of admissibility.

Closely related to the notion of a categorical action is that of a cyclotomic quotient, which is the quotient of the affine Brauer category by a left tensor ideal that specializes the dotted bubbles to scalars and imposes a polynomial relation on the dot.  In \cite[Th.~C]{RS19}, these quotients are studied and a nondegeneracy theorem is proved in the case of $\mathbf{u}$-admissible parameters;  see \cref{sec:admissibleBrauer} for a more detailed discussion of admissibility conditions for parameters.  However, we consider these quotients from a slightly different perspective than the existing literature---rather than fixing the polynomial relation on the dot first and then analyzing compatible choices of the bubble parameters, we fix the bubble parameters first and then study the cyclotomic quotient for different choices of the polynomial relation.  In particular, in \cref{UNAM}, we give an explicit description of the minimal polynomial of the dot in the quotient categories without any assumption of admissibility, and thus identify each cyclotomic Brauer category with a standard one, where the bubble scalars are determined by the minimal polynomial; see \cref{tacos}.  This extends \cite[Th.~C]{RS19} to give a basis for the morphism spaces of the cyclotomic quotient in all cases.  In \cref{sec:cyclotomicKauffman}, we perform a similar analysis of the cyclotomic Kauffman categories of \cite{GRS22}.

In addition to providing a useful formalism for studying the affine Brauer and Kauffman categories, we also expect the generating function approach developed in the current paper to be a key ingredient in relating these categories to categorified iquantum groups, including those defined in \cite{BSWW18,BWW24}, in a manner analogous to the connection between the Heisenberg category and Kac--Moody 2-categories developed in \cite{BSW-HKM}.  This is one of the main motivations for developing such a generating function approach.

%-----------------------------
\subsection*{Acknowledgements}
%-----------------------------

This research of A.S.\ was supported by NSERC Discovery Grant RGPIN-2023-03842 and B.W.\ was supported by NSERC Discovery Grants RGPIN-2018-03974 and RGPIN-2024-03760.  This research was supported in part by Perimeter Institute for Theoretical Physics. Research at Perimeter Institute is supported by the Government of Canada through the Department of Innovation, Science and Economic Development Canada and by the Province of Ontario through the Ministry of Research, Innovation and Science.
The authors are grateful for the support and hospitality of the Sydney Mathematical Research Institute (SMRI).  They would also like to thank Jon Brundan, Mengmeng Gao, Hebing Rui, and Linliang Song for helpful comments on an earlier draft of this paper.

%=================================================
\section{The affine Brauer category\label{sec:AB}}
%=================================================

In this section, we recall the definition of the affine Brauer category and introduce the generating function formalism.  We use the generating function approach to recover several important results about the category. We let $\kk$ be an arbitrary commutative ring in which $2$ is invertible.  Throughout this paper, all categories are $\kk$-linear.  We also adopt the convention that $0 \in \N$.

\begin{defin}[{\cite[Def.~1.2]{RS19}}] \label{ABdef}
    The \emph{affine Brauer} category $\AB$ is the strict $\kk$-linear monoidal category generated by an object $\go$ and morphisms
    \[
        \capmor \colon \go \otimes \go \to \one,\quad
        \cupmor \colon \one \to \go \otimes \go,\quad
        \crossmor \colon \go \otimes \go \to \go \otimes \go,\quad
        \dotstrand \colon \go \to \go,
    \]
    subject to the relations
    \begin{gather} \label{brauer}
        \begin{tikzpicture}[centerzero]
            \draw (0.2,-0.4) to[out=135,in=down] (-0.15,0) to[out=up,in=225] (0.2,0.4);
            \draw (-0.2,-0.4) to[out=45,in=down] (0.15,0) to[out=up,in=-45] (-0.2,0.4);
        \end{tikzpicture}
        \ =\
        \begin{tikzpicture}[centerzero]
            \draw (-0.2,-0.4) -- (-0.2,0.4);
            \draw (0.2,-0.4) -- (0.2,0.4);
        \end{tikzpicture}
        \ ,\quad
        \begin{tikzpicture}[centerzero]
            \draw (0.4,-0.4) -- (-0.4,0.4);
            \draw (0,-0.4) to[out=135,in=down] (-0.32,0) to[out=up,in=225] (0,0.4);
            \draw (-0.4,-0.4) -- (0.4,0.4);
        \end{tikzpicture}
        \ =\
        \begin{tikzpicture}[centerzero]
            \draw (0.4,-0.4) -- (-0.4,0.4);
            \draw (0,-0.4) to[out=45,in=down] (0.32,0) to[out=up,in=-45] (0,0.4);
            \draw (-0.4,-0.4) -- (0.4,0.4);
        \end{tikzpicture}
        \ ,\quad
        \begin{tikzpicture}[centerzero]
            \draw (-0.3,0.4) -- (-0.3,0) arc(180:360:0.15) arc(180:0:0.15) -- (0.3,-0.4);
        \end{tikzpicture}
        =
        \begin{tikzpicture}[centerzero]
            \draw (0,-0.4) -- (0,0.4);
        \end{tikzpicture}
        = \
        \begin{tikzpicture}[centerzero]
            \draw (-0.3,-0.4) -- (-0.3,0) arc(180:0:0.15) arc(180:360:0.15) -- (0.3,0.4);
        \end{tikzpicture}
        \ ,\quad
        \begin{tikzpicture}[anchorbase]
            \draw (-0.15,-0.4) to[out=60,in=-90] (0.15,0) arc(0:180:0.15) to[out=-90,in=120] (0.15,-0.4);
        \end{tikzpicture}
        = \,
        \capmor
        \ ,\quad
        \begin{tikzpicture}[centerzero]
            \draw (-0.2,-0.3) -- (-0.2,-0.1) arc(180:0:0.2) -- (0.2,-0.3);
            \draw (-0.3,0.3) \braiddown (0,-0.3);
        \end{tikzpicture}
        =
        \begin{tikzpicture}[centerzero]
            \draw (-0.2,-0.3) -- (-0.2,-0.1) arc(180:0:0.2) -- (0.2,-0.3);
            \draw (0.3,0.3) \braiddown (0,-0.3);
        \end{tikzpicture}
        \ ,
        \\ \label{dotmoves}
        \begin{tikzpicture}[centerzero]
            \draw (-0.35,-0.35) -- (0.35,0.35);
            \draw (0.35,-0.35) -- (-0.35,0.35);
            \singdot{-0.17,0.17};
        \end{tikzpicture}
        -
        \begin{tikzpicture}[centerzero]
            \draw (-0.35,-0.35) -- (0.35,0.35);
            \draw (0.35,-0.35) -- (-0.35,0.35);
            \singdot{0.17,-0.17};
        \end{tikzpicture}
        \ = \
        \begin{tikzpicture}[centerzero]
            \draw (-0.2,-0.35) -- (-0.2,0.35);
            \draw (0.2,-0.35) -- (0.2,0.35);
        \end{tikzpicture}
        \, -\,
        \begin{tikzpicture}[centerzero]
            \draw (-0.2,-0.35) -- (-0.2,-0.3) arc(180:0:0.2) -- (0.2,-0.35);
            \draw (-0.2,0.35) -- (-0.2,0.3) arc(180:360:0.2) -- (0.2,0.35);
        \end{tikzpicture}
        \ ,\qquad
        \begin{tikzpicture}[anchorbase]
            \draw (-0.2,-0.3) -- (-0.2,-0.1) arc(180:0:0.2) -- (0.2,-0.3);
            \singdot{-0.2,-0.1};
        \end{tikzpicture}
        \ = -\
        \begin{tikzpicture}[anchorbase]
            \draw (-0.2,-0.3) -- (-0.2,-0.1) arc(180:0:0.2) -- (0.2,-0.3);
            \singdot{0.2,-0.1};
        \end{tikzpicture}
        \ .
    \end{gather}
    We call the morphism $\dotstrand$ a \emph{dot}.
\end{defin}
One can easily verify from these relations that there is no nontrivial grading by any abelian group on the morphisms in this category in which the dot, crossing, cup and cap morphisms are homogeneous.
\begin{rem} \label{RSflip}
    The affine Brauer category of \cite[Def.~1.2]{RS19} is actually the reverse of ours.  Precisely, let $\AB'$ be the category of \cite[Def.~1.2]{RS19}.  Then, comparing presentations, one sees that we have an isomorphism of $\kk$-linear monoidal categories $\AB \xrightarrow{\cong} (\AB')^\rev$ defined on generating morphisms by
    \[
        \capmor \mapsto \capmor\, ,\quad
        \cupmor \mapsto \cupmor\, ,\quad
        \dotstrand \mapsto \dotstrand\, ,\quad
        \crossmor \mapsto \crossmor\, ,
    \]
    where $(\AB')^\rev$ denotes the reversed category of $\AB'$, with tensor product reversed.  Intuitively, this isomorphism flips diagrams in a vertical line.  Combined with the isomorphism \cref{yak} below, this implies that the category of \cite[Def.~1.2]{RS19} is also isomorphic to the category defined in \cref{ABdef}.  However, it is more useful to think of the two categories as the reverses of each other, since we will work with \emph{left} tensor ideals, whereas \cite{RS19} uses \emph{right} tensor ideals.

    The definition \cite[Def.~1.2]{RS19} includes some additional defining relations: the first and second relations in \cref{mirror} below.  However,  these follow from \cref{brauer,dotmoves} using adjunction.
\end{rem}

\begin{prop} \label{mirrorprop}
    The following relations hold in $\AB$:
    \begin{gather} \label{mirror}
        \begin{tikzpicture}[anchorbase]
            \draw (-0.15,0.4) to[out=-60,in=90] (0.15,0) arc(360:180:0.15) to[out=90,in=240] (0.15,0.4);
        \end{tikzpicture}
        =\ \cupmor
        \ ,\qquad
        \begin{tikzpicture}[centerzero]
            \draw (-0.2,0.3) -- (-0.2,0.1) arc(180:360:0.2) -- (0.2,0.3);
            \draw (-0.3,-0.3) \braidup (0,0.3);
        \end{tikzpicture}
        =
        \begin{tikzpicture}[centerzero]
            \draw (-0.2,0.3) -- (-0.2,0.1) arc(180:360:0.2) -- (0.2,0.3);
            \draw (0.3,-0.3) \braidup (0,0.3);
        \end{tikzpicture}
        \ ,\qquad
        \begin{tikzpicture}[centerzero]
            \draw (-0.35,-0.35) -- (0.35,0.35);
            \draw (0.35,-0.35) -- (-0.35,0.35);
            \singdot{-0.17,-0.17};
        \end{tikzpicture}
        -
        \begin{tikzpicture}[centerzero]
            \draw (-0.35,-0.35) -- (0.35,0.35);
            \draw (0.35,-0.35) -- (-0.35,0.35);
            \singdot{0.17,0.17};
        \end{tikzpicture}
        \ = \
        \begin{tikzpicture}[centerzero]
            \draw (-0.2,-0.35) -- (-0.2,0.35);
            \draw (0.2,-0.35) -- (0.2,0.35);
        \end{tikzpicture}
        \, -\,
        \begin{tikzpicture}[centerzero]
            \draw (-0.2,-0.35) -- (-0.2,-0.3) arc(180:0:0.2) -- (0.2,-0.35);
            \draw (-0.2,0.35) -- (-0.2,0.3) arc(180:360:0.2) -- (0.2,0.35);
        \end{tikzpicture}
        \ ,\qquad
        \begin{tikzpicture}[anchorbase]
            \draw (-0.2,0.3) -- (-0.2,0.1) arc(180:360:0.2) -- (0.2,0.3);
            \singdot{-0.2,0.1};
        \end{tikzpicture}
        \ = -\
        \begin{tikzpicture}[anchorbase]
            \draw (-0.2,0.3) -- (-0.2,0.1) arc(180:360:0.2) -- (0.2,0.3);
            \singdot{0.2,0.1};
        \end{tikzpicture}
        \ .
    \end{gather}
\end{prop}

\begin{proof}
    The proofs all involve rotating the relations \cref{brauer,dotmoves} using cups and caps.  We give the proof of the third relation in \cref{mirror}, since the others are analogous.  First note that, using the fourth and sixth equalities in \cref{brauer}, we have
    \[
        \begin{tikzpicture}[centerzero]
            \draw (0.4,0.3) -- (0.4,0.1) to[out=down,in=right] (0.2,-0.2) to[out=left,in=right] (-0.2,0.2) to[out=left,in=up] (-0.4,-0.1) -- (-0.4,-0.3);
            \draw (-0.2,-0.3) \braidup (0.2,0.3);
        \end{tikzpicture}
        =
        \begin{tikzpicture}[centerzero]
            \draw (0.4,0.3) -- (0.4,0.1) to[out=down,in=right] (0.2,-0.2) to[out=left,in=right] (-0.2,0.2) to[out=left,in=up] (-0.4,-0.1) -- (-0.4,-0.3);
            \draw (-0.2,-0.3) \braidup (-0.6,0.3);
        \end{tikzpicture}
        =
        \begin{tikzpicture}[centerzero]
            \draw (-0.3,-0.3) -- (0.3,0.3);
            \draw (0.3,-0.3) -- (-0.3,0.3);
        \end{tikzpicture}
        \ .
    \]
    Thus,
    \begin{multline*}
        \begin{tikzpicture}[centerzero]
            \draw (-0.4,-0.4) -- (0.4,0.4);
            \draw (0.4,-0.4) -- (-0.4,0.4);
            \singdot{-0.17,-0.17};
        \end{tikzpicture}
        =
        \begin{tikzpicture}[centerzero]
            \draw (0.5,0.4) -- (0.5,0.1) to[out=down,in=right] (0.25,-0.25) to[out=left,in=right] (-0.25,0.25) to[out=left,in=up] (-0.5,-0.1) -- (-0.5,-0.4);
            \draw (-0.2,-0.4) \braidup (0.2,0.4);
            \singdot{-0.5,-0.1};
        \end{tikzpicture}
        \overset{\cref{dotmoves}}{=} -\
        \begin{tikzpicture}[centerzero]
            \draw (0.5,0.4) -- (0.5,0.1) to[out=down,in=right] (0.25,-0.25) to[out=left,in=right] (-0.25,0.25) to[out=left,in=up] (-0.5,-0.1) -- (-0.5,-0.4);
            \draw (-0.2,-0.4) \braidup (0.2,0.4);
            \singdot{-0.08,0.14};
        \end{tikzpicture}
        \overset{\cref{dotmoves}}{=}\
        -\,
        \begin{tikzpicture}[centerzero]
            \draw (0.5,0.4) -- (0.5,0.1) to[out=down,in=right] (0.25,-0.25) to[out=left,in=right] (-0.25,0.25) to[out=left,in=up] (-0.5,-0.1) -- (-0.5,-0.4);
            \draw (-0.2,-0.4) \braidup (0.2,0.4);
            \singdot{0.08,-0.14};
        \end{tikzpicture}
        \, -\,
        \begin{tikzpicture}[centerzero]
            \draw (-0.2,-0.35) -- (-0.2,-0.3) arc(180:0:0.2) -- (0.2,-0.35);
            \draw (-0.2,0.35) -- (-0.2,0.3) arc(180:360:0.2) -- (0.2,0.35);
        \end{tikzpicture}
        \, +\,
        \begin{tikzpicture}[centerzero]
            \draw (-0.3,0.4) -- (-0.3,0) arc(180:360:0.15) arc(180:0:0.15) -- (0.3,-0.4);
            \draw (0.6,-0.4) -- (0.6,0.4);
        \end{tikzpicture}
        \\
        \overset{\cref{brauer}}{\underset{\cref{dotmoves}}{=}}\
        \begin{tikzpicture}[centerzero]
            \draw (0.5,0.4) -- (0.5,0.1) to[out=down,in=right] (0.25,-0.25) to[out=left,in=right] (-0.25,0.25) to[out=left,in=up] (-0.5,-0.1) -- (-0.5,-0.4);
            \draw (-0.2,-0.4) \braidup (0.2,0.4);
            \singdot{0.5,0.1};
        \end{tikzpicture}
        \, -\,
        \begin{tikzpicture}[centerzero]
            \draw (-0.2,-0.35) -- (-0.2,-0.3) arc(180:0:0.2) -- (0.2,-0.35);
            \draw (-0.2,0.35) -- (-0.2,0.3) arc(180:360:0.2) -- (0.2,0.35);
        \end{tikzpicture}
        \, +\,
        \begin{tikzpicture}[centerzero]
            \draw (-0.2,-0.35) -- (-0.2,0.35);
            \draw (0.2,-0.35) -- (0.2,0.35);
        \end{tikzpicture}
        \, =\,
        \begin{tikzpicture}[centerzero]
            \draw (-0.35,-0.35) -- (0.35,0.35);
            \draw (0.35,-0.35) -- (-0.35,0.35);
            \singdot{0.17,0.17};
        \end{tikzpicture}
        \, -\,
        \begin{tikzpicture}[centerzero]
            \draw (-0.2,-0.35) -- (-0.2,-0.3) arc(180:0:0.2) -- (0.2,-0.35);
            \draw (-0.2,0.35) -- (-0.2,0.3) arc(180:360:0.2) -- (0.2,0.35);
        \end{tikzpicture}
        \, + \,
        \begin{tikzpicture}[centerzero]
            \draw (-0.2,-0.35) -- (-0.2,0.35);
            \draw (0.2,-0.35) -- (0.2,0.35);
        \end{tikzpicture}
        \ . \qedhere
    \end{multline*}
\end{proof}

All rotations of the relations \cref{brauer,dotmoves} by $180\degree$ also hold in $\AB$. This shows that there is a functor $\mathbf{D} \colon \cC \to \cC^\op$ given by rotating diagrams through $180\degree$.  In terms of the theory of monoidal categories, this is the duality functor, and the fact that $\mathbf{D}$ is strictly monoidal and that $\mathbf{D}^2 = \id_\cC$ shows that the category $\AB$ is strict pivotal---this is precisely the definition of a strictly pivotal category.
Note, this does not mean that we can freely isotope diagrams, since the final relation of \cref{dotmoves} shows that we will pick up signs when dots move past minima or maxima.

Although we do not use them in the current paper, we point out two useful symmetries of the affine Brauer category.  It follows from \cref{mirrorprop} that reflecting in a horizontal line gives an isomorphism of monoidal categories
\[
    \AB \to \AB^\op,
\]
where $\AB^\op$ denotes the opposite category of $\AB$.  On the other hand, reflecting diagrams in a vertical line and then multiplying all dots by $-1$ gives an isomorphism of monoidal categories
\begin{equation} \label{yak}
    \AB \to \AB^\rev.
\end{equation}

We define
\[
    \begin{tikzpicture}[centerzero]
        \draw (0,-0.2) -- (0,0.2);
        \multdot{0,0}{east}{n};
    \end{tikzpicture}
    :=
    \left( \dotstrand \right)^{\circ n}.
\]
to denote the $n$-fold vertical composition of $\dotstrand$ with itself.  To simplify algebraic manipulations with sums of powers, we will also use the notation
\[
    \begin{tikzpicture}[centerzero]
        \draw (0,-0.2) -- (0,0.2);
        \multdot{0,0}{east}{x^n};
    \end{tikzpicture}
    :=
    \begin{tikzpicture}[centerzero]
        \draw (0,-0.2) -- (0,0.2);
        \multdot{0,0}{west}{n};
    \end{tikzpicture}
    \ ,\quad n \in \N.
\]
We similarly interpret dots labelled by polynomials in $x$ by extending this notation linearly.  We view the power series
\[
    (u-x)^{-1} = u^{-1} + u^{-2} x + u^{-3} x^2 + \dotsb \in \kk[x] \llbracket u^{-1} \rrbracket
\]
as a generating function for the different possible numbers of dots on a string.  Since this power series will appear so frequently, we introduce the notation
\[
    \begin{tikzpicture}[centerzero]
        \draw (0,-0.3) -- (0,0.3);
        \uptriforce{0,0};
    \end{tikzpicture}
    :=
    \begin{tikzpicture}[centerzero]
        \draw (0,-0.3) -- (0,0.3);
        \multdot{0,0}{west}{(u-x)^{-1}};
    \end{tikzpicture}
    \qquad \text{and} \qquad
    \begin{tikzpicture}[centerzero]
        \draw (0,-0.3) -- (0,0.3);
        \downtriforce{0,0};
    \end{tikzpicture}
    :=
    \begin{tikzpicture}[centerzero]
        \draw (0,-0.3) -- (0,0.3);
        \multdot{0,0}{west}{(u+x)^{-1}};
    \end{tikzpicture}
    .
\]
Then, for example, we have
\begin{equation} \label{trirot}
    \begin{tikzpicture}[anchorbase]
        \draw (-0.2,-0.3) -- (-0.2,-0.1) arc(180:0:0.2) -- (0.2,-0.3);
        \uptriforce{-0.2,-0.1};
    \end{tikzpicture}
    \ = \
    \begin{tikzpicture}[anchorbase]
        \draw (-0.2,-0.3) -- (-0.2,-0.1) arc(180:0:0.2) -- (0.2,-0.3);
        \downtriforce{0.2,-0.1};
    \end{tikzpicture}
    \ ,\qquad
    \begin{tikzpicture}[anchorbase]
        \draw (-0.2,-0.3) -- (-0.2,-0.1) arc(180:0:0.2) -- (0.2,-0.3);
        \downtriforce{-0.2,-0.1};
    \end{tikzpicture}
    \ = \
    \begin{tikzpicture}[anchorbase]
        \draw (-0.2,-0.3) -- (-0.2,-0.1) arc(180:0:0.2) -- (0.2,-0.3);
        \uptriforce{0.2,-0.1};
    \end{tikzpicture}
    \ ,\qquad
    \begin{tikzpicture}[anchorbase]
        \draw (-0.2,0.3) -- (-0.2,0.1) arc(180:360:0.2) -- (0.2,0.3);
        \uptriforce{-0.2,0.1};
    \end{tikzpicture}
    \ = \
    \begin{tikzpicture}[anchorbase]
        \draw (-0.2,0.3) -- (-0.2,0.1) arc(180:360:0.2) -- (0.2,0.3);
        \downtriforce{0.2,0.1};
    \end{tikzpicture}
    \ ,\qquad
    \begin{tikzpicture}[anchorbase]
        \draw (-0.2,0.3) -- (-0.2,0.1) arc(180:360:0.2) -- (0.2,0.3);
        \downtriforce{-0.2,0.1};
    \end{tikzpicture}
    \ = \
    \begin{tikzpicture}[anchorbase]
        \draw (-0.2,0.3) -- (-0.2,0.1) arc(180:360:0.2) -- (0.2,0.3);
        \uptriforce{0.2,0.1};
    \end{tikzpicture}
    \ .
\end{equation}
It is also useful to note that
\begin{equation} \label{diamond}
    \begin{tikzpicture}[centerzero]
        \draw (0,-0.35) -- (0,0.35);
        \uptriforce{0,0.15};
        \downtriforce{0,-0.15};
    \end{tikzpicture}
    =
    \begin{tikzpicture}[centerzero]
        \draw (0,-0.35) -- (0,0.35);
        \uptriforce{0,-0.15};
        \downtriforce{0,0.15};
    \end{tikzpicture}
    = \tfrac{1}{2u}
    \left(
        \begin{tikzpicture}[centerzero]
            \draw (0,-0.2) -- (0,0.2);
            \uptriforce{0,0};
        \end{tikzpicture}
        +
        \begin{tikzpicture}[centerzero]
           \draw (0,-0.2) -- (0,0.2);
            \downtriforce{0,0};
        \end{tikzpicture}
    \right).
\end{equation}

For a Laurent series $p(u) \in \kk\Laurent{u^{-1}}$, we let $[p(u)]_{u^r}$ denote its $u^r$-coefficient, and we define
\begin{equation} \label{bmx}
    [p(u)]_{u^{\ge 0}} = [p]_{\ge 0}(u) := \sum_{r \ge 0} [p(u)]_{u^r} u^r \in \kk[u].
\end{equation}
For any polynomial $p(u) \in \kk[u]$, we have
\begin{equation} \label{trick}
    \begin{tikzpicture}[centerzero]
        \draw (0,-0.3) -- (0,0.3);
        \multdot{0,0}{east}{p(x)};
    \end{tikzpicture}
    =
    \left[
        \begin{tikzpicture}[centerzero]
            \draw (0,-0.3) -- (0,0.3);
            \uptriforce{0,0};
        \end{tikzpicture}
        \ p(u)
    \right]_{u^{-1}},
    \qquad
    \begin{tikzpicture}[centerzero]
        \draw (0,-0.3) -- (0,0.3);
        \multdot{0,0}{east}{p(-x)};
    \end{tikzpicture}
    =
    \left[
        \begin{tikzpicture}[centerzero]
            \draw (0,-0.3) -- (0,0.3);
            \downtriforce{0,0};
        \end{tikzpicture}
        \ p(u)
    \right]_{u^{-1}}.
\end{equation}
More generally, if $p(u) \in \kk\Laurent{u^{-1}}$, then
\begin{equation} \label{trick+}
    \begin{tikzpicture}[centerzero]
        \draw (0,-0.3) -- (0,0.3);
        \multdot{0,0}{east}{[p]_{\ge 0}(x)};
    \end{tikzpicture}
    =
    \left[
        \begin{tikzpicture}[centerzero]
            \draw (0,-0.3) -- (0,0.3);
            \uptriforce{0,0};
        \end{tikzpicture}
        \ p(u)
    \right]_{u^{-1}}
    , \qquad
    \begin{tikzpicture}[centerzero]
        \draw (0,-0.3) -- (0,0.3);
        \multdot{0,0}{east}{[p]_{\ge 0}(-x)};
    \end{tikzpicture}
    =
    \left[
        \begin{tikzpicture}[centerzero]
            \draw (0,-0.3) -- (0,0.3);
            \downtriforce{0,0};
        \end{tikzpicture}
        \ p(u)
    \right]_{u^{-1}}
    .
\end{equation}

\begin{lem}
    The following relations hold in $\AB$:
    \begin{align} \label{uptricross}
        \begin{tikzpicture}[anchorbase]
            \draw (-0.3,-0.5) -- (0.3,0.5);
            \draw (0.3,-0.5) -- (-0.3,0.5);
            \uptriforce{-0.15,0.25};
        \end{tikzpicture}
        -
        \begin{tikzpicture}[anchorbase]
            \draw (-0.3,-0.5) -- (0.3,0.5);
            \draw (0.3,-0.5) -- (-0.3,0.5);
            \uptriforce{0.15,-0.25};
        \end{tikzpicture}
        &=
        \begin{tikzpicture}[anchorbase]
            \draw (-0.2,-0.5) -- (-0.2,0.5);
            \draw (0.2,-0.5) -- (0.2,0.5);
            \uptriforce{-0.2,0};
            \uptriforce{0.2,0};
        \end{tikzpicture}
        -
        \begin{tikzpicture}[anchorbase]
            \draw (-0.2,-0.5) -- (-0.2,-0.3) arc(180:0:0.2) -- (0.2,-0.5);
            \draw (-0.2,0.5) -- (-0.2,0.3) arc(180:360:0.2) -- (0.2,0.5);
            \uptriforce{-0.2,0.3};
            \uptriforce{0.2,-0.3};
        \end{tikzpicture}
        ,&
        \begin{tikzpicture}[anchorbase]
            \draw (-0.3,-0.5) -- (0.3,0.5);
            \draw (0.3,-0.5) -- (-0.3,0.5);
            \uptriforce{-0.15,-0.25};
        \end{tikzpicture}
        -
        \begin{tikzpicture}[anchorbase]
            \draw (-0.3,-0.5) -- (0.3,0.5);
            \draw (0.3,-0.5) -- (-0.3,0.5);
            \uptriforce{0.15,0.25};
        \end{tikzpicture}
        &=
        \begin{tikzpicture}[anchorbase]
            \draw (-0.2,-0.5) -- (-0.2,0.5);
            \draw (0.2,-0.5) -- (0.2,0.5);
            \uptriforce{-0.2,0};
            \uptriforce{0.2,0};
        \end{tikzpicture}
        -
        \begin{tikzpicture}[anchorbase]
            \draw (-0.2,-0.5) -- (-0.2,-0.3) arc(180:0:0.2) -- (0.2,-0.5);
            \draw (-0.2,0.5) -- (-0.2,0.3) arc(180:360:0.2) -- (0.2,0.5);
            \uptriforce{0.2,0.3};
            \uptriforce{-0.2,-0.3};
        \end{tikzpicture}
        \ ,
        \\ \label{downtricross}
        \begin{tikzpicture}[anchorbase]
            \draw (-0.3,-0.5) -- (0.3,0.5);
            \draw (0.3,-0.5) -- (-0.3,0.5);
            \downtriforce{-0.15,0.25};
        \end{tikzpicture}
        -
        \begin{tikzpicture}[anchorbase]
            \draw (-0.3,-0.5) -- (0.3,0.5);
            \draw (0.3,-0.5) -- (-0.3,0.5);
            \downtriforce{0.15,-0.25};
        \end{tikzpicture}
        &=
        \begin{tikzpicture}[anchorbase]
            \draw (-0.2,-0.5) -- (-0.2,-0.3) arc(180:0:0.2) -- (0.2,-0.5);
            \draw (-0.2,0.5) -- (-0.2,0.3) arc(180:360:0.2) -- (0.2,0.5);
            \downtriforce{-0.2,0.3};
            \downtriforce{0.2,-0.3};
        \end{tikzpicture}
        -
        \begin{tikzpicture}[anchorbase]
            \draw (-0.2,-0.5) -- (-0.2,0.5);
            \draw (0.2,-0.5) -- (0.2,0.5);
            \downtriforce{-0.2,0};
            \downtriforce{0.2,0};
        \end{tikzpicture}
        ,&
        \begin{tikzpicture}[anchorbase]
            \draw (-0.3,-0.5) -- (0.3,0.5);
            \draw (0.3,-0.5) -- (-0.3,0.5);
            \downtriforce{-0.15,-0.25};
        \end{tikzpicture}
        -
        \begin{tikzpicture}[anchorbase]
            \draw (-0.3,-0.5) -- (0.3,0.5);
            \draw (0.3,-0.5) -- (-0.3,0.5);
            \downtriforce{0.15,0.25};
        \end{tikzpicture}
        &=
        \begin{tikzpicture}[anchorbase]
            \draw (-0.2,-0.5) -- (-0.2,-0.3) arc(180:0:0.2) -- (0.2,-0.5);
            \draw (-0.2,0.5) -- (-0.2,0.3) arc(180:360:0.2) -- (0.2,0.5);
            \downtriforce{0.2,0.3};
            \downtriforce{-0.2,-0.3};
        \end{tikzpicture}
        -
        \begin{tikzpicture}[anchorbase]
            \draw (-0.2,-0.5) -- (-0.2,0.5);
            \draw (0.2,-0.5) -- (0.2,0.5);
            \downtriforce{-0.2,0};
            \downtriforce{0.2,0};
        \end{tikzpicture}
        \ .
    \end{align}
\end{lem}

\begin{proof}
By \cref{dotmoves}, we have
    \[
        \begin{tikzpicture}[anchorbase]
            \draw (-0.3,-0.5) -- (0.3,0.5);
            \draw (0.3,-0.5) -- (-0.3,0.5);
            \multdot{0.15,-0.25}{west}{u-x};
        \end{tikzpicture}
        -
        \begin{tikzpicture}[anchorbase]
            \draw (-0.3,-0.5) -- (0.3,0.5);
            \draw (0.3,-0.5) -- (-0.3,0.5);
            \multdot{-0.15,0.25}{east}{u-x};
        \end{tikzpicture}
        \ =\
        \begin{tikzpicture}[anchorbase]
            \draw (-0.2,-0.5) -- (-0.2,0.5);
            \draw (0.2,-0.5) -- (0.2,0.5);
        \end{tikzpicture}
        -
        \begin{tikzpicture}[anchorbase]
            \draw (-0.2,-0.5) -- (-0.2,-0.3) arc(180:0:0.2) -- (0.2,-0.5);
            \draw (-0.2,0.5) -- (-0.2,0.3) arc(180:360:0.2) -- (0.2,0.5);
        \end{tikzpicture}
        \ .
    \]
    Composing with $(u-x)^{-1}$ on the top-left string and bottom-right string then yields the first relation in \cref{uptricross}.  The second equality in \cref{uptricross} is proved similarly.  The equalities in \cref{downtricross} are obtained from those in \cref{uptricross} by replacing $u$ by $-u$ and then multiplying all diagrams by $-1$.  Note that we \emph{cannot} replace $x$ by $-x$ in general, since the defining relations of $\AB$ are not homogeneous in the number of dots.  Alternatively, the relations \cref{downtricross} can be obtained by rotating the relations \cref{uptricross} using cups and caps.
\end{proof}

We define
\begin{equation} \label{eatz}
    \bubblegenr{u} := \tfrac{2u}{2u-1} \left( \uptribubbler - \tfrac{1}{2u} + 1 \right)
    = \tfrac{2u}{2u-1} \uptribubbler + 1
    , \qquad
    \bubblegenl{u} := \tfrac{2u}{2u+1} \left( \uptribubblel - \tfrac{1}{2u} - 1 \right)
    = \tfrac{2u}{2u+1} \uptribubblel - 1.
\end{equation}

We interpret these equations by Laurent expansion in $u^{-1}$ at $u=\infty$;  that is,
\[
    \bubblegenr{u}
    = 1 + \left( 1+\tfrac{1}{2u} + \dotsb + \tfrac{1}{(2u)^n} + \dotsb \right) \uptribubbler
    = 1 + u^{-1} \bubble{} + u^{-2} \left( \multbubbler{}\!\!\!+ \tfrac{1}{2}\bubble{} \right) + \dotsb.	
\]

\begin{prop} \label{lanote}
    The following relations hold in $\AB$:
    \begin{gather} \label{infgrass}
        \bubblegenl{u} = -\,
        \begin{tikzpicture}[centerzero]
            \draw (0.3,0) arc(0:360:0.3);
            \node at (0,0) {\dotlabel{-u}};
            \singdot{0.3,0};
        \end{tikzpicture}
        \ ,\qquad\qquad
        \bubblegenl{u} \ \bubblegenr{u} = -1,
        \\ \label{curlsup}
        \begin{tikzpicture}[anchorbase]
            \draw (0,-0.5) to[out=up,in=180] (0.3,0.2) to[out=0,in=up] (0.45,0) to[out=down,in=0] (0.3,-0.2) to[out=180,in=down] (0,0.5);
            \uptriforce{0.45,0};
        \end{tikzpicture}
        = \left( 1 - \tfrac{1}{2u} \right)
        \begin{tikzpicture}[centerzero]
            \draw (0,-0.5) -- (0,0.5);
            \downtriforce{0,0};
            \bubgenr{u}{0.4,0};
        \end{tikzpicture}
        - \tfrac{1}{2u}
        \begin{tikzpicture}[centerzero]
            \draw (0,-0.5) -- (0,0.5);
            \uptriforce{0,0};
        \end{tikzpicture}
        \ , \qquad
        \begin{tikzpicture}[anchorbase]
            \draw (0,-0.5) to[out=up,in=0] (-0.3,0.2) to[out=180,in=up] (-0.45,0) to[out=down,in=180] (-0.3,-0.2) to[out=0,in=down] (0,0.5);
            \uptriforce{-0.45,0};
        \end{tikzpicture}
        = - \left( 1 + \tfrac{1}{2u} \right)\
        \begin{tikzpicture}[centerzero]
            \draw (0,-0.5) -- (0,0.5);
            \downtriforce{0,0};
            \bubgenl{u}{-0.4,0};
        \end{tikzpicture}
        + \tfrac{1}{2u}
        \begin{tikzpicture}[centerzero]
            \draw (0,-0.5) -- (0,0.5);
            \uptriforce{0,0};
        \end{tikzpicture}
        \ ,
        \\ \label{curlsdown}
        \begin{tikzpicture}[anchorbase]
            \draw (0,-0.5) to[out=up,in=180] (0.3,0.2) to[out=0,in=up] (0.45,0) to[out=down,in=0] (0.3,-0.2) to[out=180,in=down] (0,0.5);
            \downtriforce{0.45,0};
        \end{tikzpicture}
        = - \left( 1 + \tfrac{1}{2u} \right)
        \begin{tikzpicture}[centerzero]
            \draw (0,-0.5) -- (0,0.5);
            \uptriforce{0,0};
            \bubgenl{u}{0.5,0};
        \end{tikzpicture}
        + \tfrac{1}{2u}
        \begin{tikzpicture}[centerzero]
            \draw (0,-0.5) -- (0,0.5);
            \downtriforce{0,0};
        \end{tikzpicture}
        \ , \qquad
        \begin{tikzpicture}[anchorbase]
            \draw (0,-0.5) to[out=up,in=0] (-0.3,0.2) to[out=180,in=up] (-0.45,0) to[out=down,in=180] (-0.3,-0.2) to[out=0,in=down] (0,0.5);
            \downtriforce{-0.45,0};
        \end{tikzpicture}
        = \left( 1 - \tfrac{1}{2u} \right)\
        \begin{tikzpicture}[centerzero]
            \draw (0,-0.5) -- (0,0.5);
            \uptriforce{0,0};
            \bubgenr{u}{-0.5,0};
        \end{tikzpicture}
        - \tfrac{1}{2u}
        \begin{tikzpicture}[centerzero]
            \draw (0,-0.5) -- (0,0.5);
            \downtriforce{0,0};
        \end{tikzpicture}
        \ ,
        \\ \label{bubslide}
        \begin{tikzpicture}[anchorbase]
            \draw (0,-0.5) -- (0,0.5);
            \bubgenr{u}{-0.4,0};
        \end{tikzpicture}
        \ =
        \begin{tikzpicture}[anchorbase]
            \draw (0,-0.5) -- (0,0.5);
            \multdot{0,0}{east}{\frac{1-(u+x)^{-2}}{1-(u-x)^{-2}}};
            \bubgenr{u}{0.4,0};
        \end{tikzpicture}
        \ , \qquad
        \begin{tikzpicture}[anchorbase]
            \draw (0,-0.5) -- (0,0.5);
            \bubgenl{u}{-0.4,0};
        \end{tikzpicture}
        \ =
        \begin{tikzpicture}[anchorbase]
            \draw (0,-0.5) -- (0,0.5);
            \multdot{0,0}{east}{\frac{1-(u-x)^{-2}}{1-(u+x)^{-2}}};
            \bubgenl{u}{0.5,0};
        \end{tikzpicture}
        \ .
    \end{gather}
\end{prop}

\begin{proof}
    To see the first equality of \cref{infgrass}, we note that
    \[
        \begin{tikzpicture}[centerzero]
            \draw (0.3,0) arc(0:360:0.3);
            \node at (0,0) {\dotlabel{-u}};
            \singdot{0.3,0};
        \end{tikzpicture}
        = \tfrac{-2u}{-2u-1} \left( -\, \downtribubbler + \tfrac{1}{2u} + 1 \right)
        \overset{\cref{trirot}}{=} -\, \bubblegenl{u}.
    \]
    Next,
    \[
        \uptribubblel
        \overset{\cref{brauer}}{=}
        \begin{tikzpicture}[centerzero]
            \draw (-0.2,-0.25) \braidup (0.2,0.25) arc(0:180:0.2) \braiddown (0.2,-0.25) arc(360:180:0.2);
            \uptriforce{-0.2,-0.25};
        \end{tikzpicture}
        \overset{\cref{uptricross}}{=}
        \begin{tikzpicture}[centerzero]
            \draw (-0.2,-0.25) \braidup (0.2,0.25) arc(0:180:0.2) \braiddown (0.2,-0.25) arc(360:180:0.2);
            \uptriforce{0.2,0.25};
        \end{tikzpicture}
        +
        \begin{tikzpicture}[centerzero]
            \draw (0.2,-0.15) -- (0.2,0.15) arc(0:180:0.2) -- (-0.2,-0.15) arc(180:360:0.2);
            \uptriforce{0.2,0};
            \uptriforce{-0.2,0};
        \end{tikzpicture}
        - \uptribubblel \ \uptribubbler
        \overset{\cref{mirror}}{\underset{\cref{trirot}}{=}}
        \uptribubbler
        +
        \begin{tikzpicture}[centerzero]
            \draw (0.2,-0.15) -- (0.2,0.15) arc(0:180:0.2) -- (-0.2,-0.15) arc(180:360:0.2);
            \downtriforce{0.2,0.15};
            \uptriforce{0.2,-0.15};
        \end{tikzpicture}
        - \uptribubblel \ \uptribubbler
        \overset{\cref{diamond}}{=}
        \uptribubbler
        + \tfrac{1}{2u} \left( \uptribubbler + \downtribubbler \right)
        - \uptribubblel \ \uptribubbler\, .
    \]
    Thus
    \begin{equation} \label{peets}
        \uptribubblel \ \uptribubbler
        = \left( \tfrac{1}{2u} + 1 \right) \uptribubbler + \left( \tfrac{1}{2u} - 1 \right) \uptribubblel\, ,
    \end{equation}
    which implies that
    \begin{equation}\label{euclid}
        \left( \uptribubblel - \tfrac{1}{2u} - 1 \right)
        \left( \uptribubbler - \tfrac{1}{2u} + 1 \right)
        = \left( \tfrac{1}{2u} + 1 \right) \left( \tfrac{1}{2u} - 1 \right).    	
    \end{equation}
    Multiplying both sides by $\frac{4u^2}{(1+2u)(1-2u)}$ then gives the second equality of \cref{infgrass}.

    To prove the first equality in \cref{curlsup}, we compute
    \begin{multline*}
        \begin{tikzpicture}[centerzero]
            \draw (0,-0.5) to[out=up,in=180] (0.3,0.2) to[out=0,in=up] (0.45,0) to[out=down,in=0] (0.3,-0.2) to[out=180,in=down] (0,0.5);
            \uptriforce{0.45,0};
        \end{tikzpicture}
        \overset{\cref{trirot}}{=}
        \begin{tikzpicture}[centerzero]
            \draw (0,-0.5) to[out=up,in=180] (0.3,0.2) to[out=0,in=up] (0.45,0) to[out=down,in=0] (0.3,-0.2) to[out=180,in=down] (0,0.5);
            \downtriforce{0.15,-0.14};
        \end{tikzpicture}
        \overset{\cref{downtricross}}{=}
        \begin{tikzpicture}[centerzero]
            \draw (0,-0.5) to[out=up,in=180] (0.3,0.2) to[out=0,in=up] (0.45,0) to[out=down,in=0] (0.3,-0.2) to[out=180,in=down] (0,0.5);
            \downtriforce{0.01,0.27};
        \end{tikzpicture}
        +
        \begin{tikzpicture}[centerzero]
            \draw (0,-0.5) -- (0,0.5);
            \downtriforce{0,0};
            \downtribubl{0.5,0};
        \end{tikzpicture}
        -
        \begin{tikzpicture}[centerzero]
            \draw (0,-0.5) -- (0,-0.2) arc(180:0:0.15) arc(180:360:0.15) -- (0.6,0.2) arc(0:180:0.15) arc(360:180:0.15) -- (0,0.5);
            \downtriforce{0.3,-0.2};
            \downtriforce{0,0.3};
        \end{tikzpicture}
        \overset{\cref{trirot}}{\underset{\cref{brauer}}{=}}
        \begin{tikzpicture}[centerzero]
            \draw (0,-0.5) -- (0,0.5);
            \downtriforce{0,0};
        \end{tikzpicture}
        +
        \begin{tikzpicture}[centerzero]
            \draw (0,-0.5) -- (0,0.5);
            \downtriforce{0,0};
            \downtribubl{0.5,0};
        \end{tikzpicture}
        -
        \begin{tikzpicture}[centerzero]
            \draw (0,-0.5) -- (0,0.5);
            \downtriforce{0,0.2};
            \uptriforce{0,-0.2};
        \end{tikzpicture}
        \overset{\cref{diamond}}{=}
        \begin{tikzpicture}[centerzero]
            \draw (0,-0.5) -- (0,0.5);
            \downtriforce{0,0};
        \end{tikzpicture}
        +
        \begin{tikzpicture}[centerzero]
            \draw (0,-0.5) -- (0,0.5);
            \downtriforce{0,0};
            \downtribubl{0.5,0};
        \end{tikzpicture}
        - \tfrac{1}{2u}
        \left(
            \begin{tikzpicture}[centerzero]
                \draw (0,-0.5) -- (0,0.5);
                \uptriforce{0,0};
            \end{tikzpicture}
            +
            \begin{tikzpicture}[centerzero]
                \draw (0,-0.5) -- (0,0.5);
                \downtriforce{0,0};
            \end{tikzpicture}
        \right)
        \\
        \overset{\cref{trirot}}{=}
        \begin{tikzpicture}[centerzero]
            \draw (0,-0.5) -- (0,0.5);
            \downtriforce{0,0};
        \end{tikzpicture}
        \left( \uptribubbler - \tfrac{1}{2u} + 1 \right)
        - \tfrac{1}{2u}\,
        \begin{tikzpicture}[centerzero]
            \draw (0,-0.5) -- (0,0.5);
            \uptriforce{0,0};
        \end{tikzpicture}
        \overset{\cref{eatz}}{=} \left( 1 - \tfrac{1}{2u} \right)
        \begin{tikzpicture}[centerzero]
            \draw (0,-0.5) -- (0,0.5);
            \downtriforce{0,0};
            \bubgenr{u}{0.4,0};
        \end{tikzpicture}
        - \tfrac{1}{2u}\,
        \begin{tikzpicture}[centerzero]
            \draw (0,-0.5) -- (0,0.5);
            \uptriforce{0,0};
        \end{tikzpicture}
        \, .
    \end{multline*}
    The proof of the second equality in \cref{curlsup} and both equalities in \cref{curlsdown} are analogous.

    To prove the first equality in \cref{bubslide}, we compute
    \begin{multline*}
        \begin{tikzpicture}[centerzero]
            \draw (0,-0.6) -- (0,0.6);
            \uptribubr{-0.5,0};
        \end{tikzpicture}
        \ \overset{\cref{brauer}}{=}\
        \begin{tikzpicture}[centerzero]
            \draw (0,-0.6) -- (0,0.6);
            \draw (-0.5,0) to[out=90,in=180] (-0.3,0.4) to[out=0,in=90,looseness=0.5] (0.2,-0.2) to[out=down,in=down] (-0.5,0);
            \uptriforce{-0.15,0.33};
        \end{tikzpicture}
        \ \overset{\cref{uptricross}}{=}\
        \begin{tikzpicture}[centerzero]
            \draw (0,-0.6) -- (0,0.6);
            \draw (-0.5,0) to[out=90,in=180] (-0.3,0.4) to[out=0,in=90,looseness=0.5] (0.25,-0.2) to[out=down,in=down] (-0.5,0);
            \uptriforce{0.15,0.03};
        \end{tikzpicture}
        +
        \begin{tikzpicture}[centerzero]
            \draw (0,-0.6) -- (0,-0.45) \braidup (-0.2,0) -- (-0.2,0.2) arc(0:180:0.2) -- (-0.6,-0.2) arc(180:270:0.2) to[out=right,in=down] (0,0.4) -- (0,0.6);
            \uptriforce{-0.2,0};
            \uptriforce{0,0.4};
        \end{tikzpicture}
        -
        \begin{tikzpicture}[centerzero]
            \draw (0,-0.6) to[out=up,in=down] (-0.2,-0.15) arc (180:0:0.15) to[out=down,in=down] (-0.5,-0.25) -- (-0.5,0.25) arc(180:0:0.15) arc(180:360:0.15) -- (0.1,0.6);
            \uptriforce{0.1,-0.15};
            \uptriforce{-0.2,0.25};
        \end{tikzpicture}
        \ \overset{\cref{brauer}}{=}\
        \begin{tikzpicture}[centerzero]
            \draw (0,-0.6) -- (0,0.6);
            \uptribubr{0.4,0};
        \end{tikzpicture}
        +
        \begin{tikzpicture}[centerzero]
            \draw (0,-0.6) -- (0,-0.5) to[out=up,in=0] (-0.3,0.2) to[out=left,in=up] (-0.45,0) to[out=down,in=left] (-0.3,-0.2) to[out=right,in=down] (0,0.5) -- (0,0.6);
            \downtriforce{-0.45,0};
            \uptriforce{0,0.4};
        \end{tikzpicture}
        -
        \begin{tikzpicture}[centerzero]
            \draw (0,-0.6) -- (0,-0.5) to[out=up,in=180] (0.3,0.2) to[out=0,in=up] (0.45,0) to[out=down,in=0] (0.3,-0.2) to[out=180,in=down] (0,0.5) -- (0,0.6);
            \uptriforce{0.45,0};
            \downtriforce{0,0.4};
        \end{tikzpicture}
        \\
        \overset{\cref{curlsup}}{\underset{\cref{curlsdown}}{=}}\
        \begin{tikzpicture}[centerzero]
            \draw (0,-0.6) -- (0,0.6);
            \uptribubr{0.4,0};
        \end{tikzpicture}
        + \left( 1 - \tfrac{1}{2u} \right) \bubblegenr{u}\
        \begin{tikzpicture}[centerzero]
            \draw (0,-0.6) -- (0,0.6);
            \uptriforce{0,-0.2};
            \uptriforce{0,0.2};
        \end{tikzpicture}
        - \left( 1 - \tfrac{1}{2u} \right)
        \begin{tikzpicture}[centerzero]
            \draw (0,-0.6) -- (0,0.6);
            \downtriforce{0,0.2};
            \downtriforce{0,-0.2};
            \bubgenr{u}{0.4,0};
        \end{tikzpicture}
        \ .
    \end{multline*}
    Thus
    \[
        \begin{tikzpicture}[centerzero]
            \draw (0,-0.4) -- (0,0.4);
            \uptribubr{-0.5,0};
        \end{tikzpicture}
        \ - \left( 1 - \tfrac{1}{2u} \right)\
        \begin{tikzpicture}[centerzero]
            \draw (0,-0.4) -- (0,0.4);
            \bubgenr{u}{-0.5,0};
            \uptriforce{0,0.15};
            \uptriforce{0,-0.15};
        \end{tikzpicture}
        \ =\
        \begin{tikzpicture}[centerzero]
            \draw (0,-0.4) -- (0,0.4);
            \uptribubr{0.4,0};
        \end{tikzpicture}
        - \left( 1 - \tfrac{1}{2u} \right)
        \begin{tikzpicture}[centerzero]
            \draw (0,-0.4) -- (0,0.4);
            \downtriforce{0,0.2};
            \downtriforce{0,-0.2};
            \bubgenr{u}{0.4,0};
        \end{tikzpicture}
        \ .
    \]
    Adding $\left( 1-\frac{1}{2u} \right) \idstrand$ to both sides, then multiplying both sides by $\frac{2u}{2u-1}$, gives
    \[
        \bubblegenr{u}
        \left(\
            \begin{tikzpicture}[centerzero]
                \draw (0,-0.4) -- (0,0.4);
            \end{tikzpicture}
            -
            \begin{tikzpicture}[centerzero]
                \draw (0,-0.4) -- (0,0.4);
                \uptriforce{0,0.2};
                \uptriforce{0,-0.2};
            \end{tikzpicture}
        \right)
        \ =\
        \left(\
            \begin{tikzpicture}[centerzero]
                \draw (0,-0.4) -- (0,0.4);
            \end{tikzpicture}
            -
            \begin{tikzpicture}[centerzero]
                \draw (0,-0.4) -- (0,0.4);
                \downtriforce{0,0.2};
                \downtriforce{0,-0.2};
            \end{tikzpicture}
        \right)
        \bubblegenr{u}
        \ .
    \]
    Since
    \[
        \begin{tikzpicture}[centerzero]
            \draw (0,-0.4) -- (0,0.4);
        \end{tikzpicture}
        -
        \begin{tikzpicture}[centerzero]
            \draw (0,-0.4) -- (0,0.4);
            \uptriforce{0,0.2};
            \uptriforce{0,-0.2};
        \end{tikzpicture}
        =
        \begin{tikzpicture}[centerzero]
            \draw (0,-0.4) -- (0,0.4);
            \multdot{0,0}{west}{1-(u-x)^{-2}};
        \end{tikzpicture}
        \qquad \text{and} \qquad
        \begin{tikzpicture}[centerzero]
            \draw (0,-0.4) -- (0,0.4);
        \end{tikzpicture}
        -
        \begin{tikzpicture}[centerzero]
            \draw (0,-0.4) -- (0,0.4);
            \downtriforce{0,0.2};
            \downtriforce{0,-0.2};
        \end{tikzpicture}
        =
        \begin{tikzpicture}[centerzero]
            \draw (0,-0.4) -- (0,0.4);
            \multdot{0,0}{west}{1-(u+x)^{-2}};
        \end{tikzpicture}
        ,
    \]
    the result follows.  Finally, to obtain the second equality in \cref{bubslide}, we multiply both sides of the first equality in \cref{bubslide} on the left and right by $\bubblegenl{u}$, use \cref{infgrass}, compose both sides with
    \[
        \begin{tikzpicture}[centerzero]
            \draw (0,-0.4) -- (0,0.4);
            \multdot{0,0}{west}{\frac{1-(u-x)^{-2}}{1-(u+x)^{-2}}};
        \end{tikzpicture}
        ,
    \]
    and then multiply both sides by $-1$.
\end{proof}

\begin{rem}
    The power series $(u-\frac{1}{2}) \bubblegenr{u}$ was considered in \cite[\S 4]{GRS23}, and \cref{infgrass,bubslide} correspond to Lemmas 4.1 and 4.2 there.  When the bubbles $\multbubbler{r}$, $r \in \N$, are evaluated at scalars, the second relation in \cref{infgrass} corresponds to \cite[Rem.~2.11]{AMR06}, particularly when written in the form \cref{euclid}.  Similarly, \cref{bubslide} corresponds to \cite[Lem. 4.15]{AMR06}.  We will discuss these connections in more detail in \cref{sec:admissibleBrauer}.
\end{rem}

The relations \cref{euclid} are equivalent to the following lemma, which shows how the bubbles with an odd number of dots can be written in terms of bubbles with fewer dots.

\begin{lem}[{\cite[Lem.~3.4]{RS19}}] \label{scratch}
    In $\AB$, we have
    \begin{equation} \label{pow}
        \multbubbler{2r+1}
        = \frac{1}{2} \left( -\, \multbubbler{2r} + \sum_{n=0}^{2r} (-1)^n\, \multbubbler{n} \multbubbler{2r-n} \right),\qquad
        r \in \N.
    \end{equation}
\end{lem}

\begin{proof}
    Although \cref{pow} is proved in \cite[Lem.~3.4]{RS19}, we give here a proof showing how it follows from \cref{infgrass}.  As explained in the proof of \cref{lanote}, the second equation in \cref{infgrass} is equivalent to \cref{peets}.   From \cref{peets} and the second relation in \cref{dotmoves}, we have
    \[
        \left( \sum_{n=0}^\infty (-1)^n\, \multbubbler{n} u^{-n-1} \right)
        \left( \sum_{n=0}^\infty \multbubbler{n} u^{-n-1} \right)
        = \sum_{n=0}^\infty \left( \frac{1+(-1)^n}{2u} + 1 - (-1)^n\right) \multbubbler{n} u^{-n-1}.
    \]
    Equating coefficients of $u^{-m-1}$, $m \in \N$, gives
    \begin{equation} \label{hubie}
        \sum_{n=0}^{m-1} (-1)^n\, \multbubbler{n} \multbubbler{m-n-1}
        = \frac{1 + (-1)^{m-1}}{2}\, \multbubbler{m-1} + \big( 1 - (-1)^m \big)\, \multbubbler{m}.
    \end{equation}
    When $m$ is even, both sides of \cref{hubie} are equal to zero.  On the other hand, when $m = 2r+1$, $r \in \N$, \cref{hubie} yields \cref{pow}.
\end{proof}

%===================================================
\section{The affine Kauffman category\label{sec:AK}}
%===================================================

In this section we recall the definition of the affine Kauffman category of \cite{GRS22}.  As for the affine Brauer category, we use a generating function approach to deduce some important relations that hold in this category.  In this section, $\kk$ is an arbitrary commutative ring.

\begin{defin}[{\cite[Def.~1.3]{GRS22}}]
    Suppose $z,t \in \kk^\times$.  The \emph{affine Kauffman category} $\AK= \AK(z,t)$ is the strict $\kk$-linear monoidal category generated by an object $\gok$ and morphisms
    \[
        \capmor \colon \gok \otimes \gok \to \one,\quad
        \cupmor \colon \one \to \gok \otimes \gok,\quad
        \poscross,\, \negcross \colon \gok \otimes \gok \to \gok \otimes \gok,\quad
        \dotstrand[black], \multdotstrand[black]{east}{-1} \colon \gok \to \gok,
    \]
    subject to the relations
    \begin{gather} \label{braid}
        \begin{tikzpicture}[centerzero]
            \draw (0.2,-0.4) to[out=135,in=down] (-0.15,0) to[out=up,in=225] (0.2,0.4);
            \draw[wipe] (-0.2,-0.4) to[out=45,in=down] (0.15,0) to[out=up,in=-45] (-0.2,0.4);
            \draw (-0.2,-0.4) to[out=45,in=down] (0.15,0) to[out=up,in=-45] (-0.2,0.4);
        \end{tikzpicture}
        \ =\
        \begin{tikzpicture}[centerzero]
            \draw (-0.2,-0.4) -- (-0.2,0.4);
            \draw (0.2,-0.4) -- (0.2,0.4);
        \end{tikzpicture}
        \ =\
        \begin{tikzpicture}[centerzero]
            \draw (-0.2,-0.4) to[out=45,in=down] (0.15,0) to[out=up,in=-45] (-0.2,0.4);
            \draw[wipe] (0.2,-0.4) to[out=135,in=down] (-0.15,0) to[out=up,in=225] (0.2,0.4);
            \draw (0.2,-0.4) to[out=135,in=down] (-0.15,0) to[out=up,in=225] (0.2,0.4);
        \end{tikzpicture}
        \ ,\qquad
        \begin{tikzpicture}[centerzero]
            \draw (0.4,-0.4) -- (-0.4,0.4);
            \draw[wipe] (0,-0.4) to[out=135,in=down] (-0.32,0) to[out=up,in=225] (0,0.4);
            \draw (0,-0.4) to[out=135,in=down] (-0.32,0) to[out=up,in=225] (0,0.4);
            \draw[wipe] (-0.4,-0.4) -- (0.4,0.4);
            \draw (-0.4,-0.4) -- (0.4,0.4);
        \end{tikzpicture}
        \ =\
        \begin{tikzpicture}[centerzero]
            \draw (0.4,-0.4) -- (-0.4,0.4);
            \draw[wipe] (0,-0.4) to[out=45,in=down] (0.32,0) to[out=up,in=-45] (0,0.4);
            \draw (0,-0.4) to[out=45,in=down] (0.32,0) to[out=up,in=-45] (0,0.4);
            \draw[wipe] (-0.4,-0.4) -- (0.4,0.4);
            \draw (-0.4,-0.4) -- (0.4,0.4);
        \end{tikzpicture}
        \ ,
        \\ \label{drops}
        \begin{tikzpicture}[centerzero]
            \draw (-0.3,0.4) -- (-0.3,0) arc(180:360:0.15) arc(180:0:0.15) -- (0.3,-0.4);
        \end{tikzpicture}
        =
        \begin{tikzpicture}[centerzero]
            \draw (0,-0.4) -- (0,0.4);
        \end{tikzpicture}
        = \
        \begin{tikzpicture}[centerzero]
            \draw (-0.3,-0.4) -- (-0.3,0) arc(180:0:0.15) arc(180:360:0.15) -- (0.3,0.4);
        \end{tikzpicture}
        \ ,\qquad
        \begin{tikzpicture}[anchorbase]
            \draw (-0.15,-0.4) to[out=60,in=-90] (0.15,0) arc(0:180:0.15);
            \draw[wipe] (-0.15,0) to[out=-90,in=120] (0.15,-0.4);
            \draw (-0.15,0) to[out=-90,in=120] (0.15,-0.4);
        \end{tikzpicture}
        = t \
        \capmor
        \ ,\qquad
        \begin{tikzpicture}[centerzero]
            \draw (-0.2,-0.3) -- (-0.2,-0.1) arc(180:0:0.2) -- (0.2,-0.3);
            \draw[wipe] (-0.3,0.3) \braiddown (0,-0.3);
            \draw (-0.3,0.3) \braiddown (0,-0.3);
        \end{tikzpicture}
        =
        \begin{tikzpicture}[centerzero]
            \draw (-0.2,-0.3) -- (-0.2,-0.1) arc(180:0:0.2) -- (0.2,-0.3);
            \draw[wipe] (0.3,0.3) \braiddown (0,-0.3);
            \draw (0.3,0.3) \braiddown (0,-0.3);
        \end{tikzpicture}
        \ ,
        \\ \label{skein}
        \begin{tikzpicture}[centerzero]
            \draw (0.35,-0.35) -- (-0.35,0.35);
            \draw[wipe] (-0.35,-0.35) -- (0.35,0.35);
            \draw (-0.35,-0.35) -- (0.35,0.35);
        \end{tikzpicture}
        \ -\
        \begin{tikzpicture}[centerzero]
            \draw (-0.35,-0.35) -- (0.35,0.35);
            \draw[wipe] (0.35,-0.35) -- (-0.35,0.35);
            \draw (0.35,-0.35) -- (-0.35,0.35);
        \end{tikzpicture}
        = z
        \left(\
            \begin{tikzpicture}[centerzero]
                \draw (-0.2,-0.35) -- (-0.2,0.35);
                \draw (0.2,-0.35) -- (0.2,0.35);
            \end{tikzpicture}
            \ -\
            \begin{tikzpicture}[centerzero]
                \draw (-0.2,-0.35) -- (-0.2,-0.3) arc(180:0:0.2) -- (0.2,-0.35);
                \draw (-0.2,0.35) -- (-0.2,0.3) arc(180:360:0.2) -- (0.2,0.35);
            \end{tikzpicture}\
        \right)
        ,\qquad
        \bubble\ = \frac{t-t^{-1}}{z}+1,
        \\ \label{kauffdot}
        \begin{tikzpicture}[centerzero]
            \draw (-0.35,-0.35) -- (0.35,0.35);
            \draw[wipe] (0.35,-0.35) -- (-0.35,0.35);
            \draw (0.35,-0.35) -- (-0.35,0.35);
            \singdot[black]{-0.2,-0.2};
        \end{tikzpicture}
        \ =\
        \begin{tikzpicture}[centerzero]
            \draw (0.35,-0.35) -- (-0.35,0.35);
            \draw[wipe] (-0.35,-0.35) -- (0.35,0.35);
            \draw (-0.35,-0.35) -- (0.35,0.35);
            \singdot[black]{0.17,0.17};
        \end{tikzpicture}
        \ ,\qquad
        \begin{tikzpicture}[anchorbase]
            \draw (-0.2,-0.3) -- (-0.2,-0.1) arc(180:0:0.2) -- (0.2,-0.3);
            \singdot[black]{-0.2,-0.1};
        \end{tikzpicture}
        \ =\
        \begin{tikzpicture}[anchorbase]
            \draw (-0.2,-0.3) -- (-0.2,-0.1) arc(180:0:0.2) -- (0.2,-0.3);
            \multdot[black]{0.2,-0.1}{west}{-1};
        \end{tikzpicture}
        \ ,\qquad
        \begin{tikzpicture}[centerzero]
            \draw (0,-0.35) -- (0,0.35);
            \singdot[black]{0,-0.15};
            \multdot[black]{0,0.15}{east}{-1};
        \end{tikzpicture}
        \ =\
        \begin{tikzpicture}[centerzero]
            \draw (0,-0.35) -- (0,0.35);
        \end{tikzpicture}
        \ =
        \begin{tikzpicture}[centerzero]
            \draw (0,-0.35) -- (0,0.35);
            \singdot[black]{0,0.15};
            \multdot[black]{0,-0.15}{east}{-1};
        \end{tikzpicture}
        \ .
    \end{gather}
    The first relation in \cref{skein} is called the \emph{Kauffman skein relation}.  We call the morphism $\dotstrand[black]$ a \emph{dot}, or a \emph{Kauffman dot}, when we want to emphasize that we are working in $\AK$, as opposed to $\AB$.
\end{defin}

\begin{rem} \label{glass}
    The category of \cite[Def.~1.3]{GRS22} is the actually the reverse of ours.  To be precise, let $\AK'(z,\delta)$ be the category defined in \cite[Def.~1.3]{GRS22}.  Then, comparing the presentations, one sees that we have an isomorphism of $\kk$-linear monoidal categories $\AK(z,t) \xrightarrow{\cong} \AK'(z,t)^\rev$ defined on generating morphisms by
    \[
        \poscross \mapsto \poscross, \quad
        \negcross \mapsto \negcross, \quad
        \cupmor \mapsto \cupmor\, , \quad
        \capmor \mapsto \capmor\, , \quad
        \dotstrand[black] \mapsto \dotstrand[black]\, , \quad
        \multdotstrand[black]{east}{-1} \mapsto \dotstrand.
    \]
\end{rem}

\begin{lem}
    The following relations hold in $\AK$:
    \begin{gather} \label{chicken}
        \begin{tikzpicture}[anchorbase]
            \draw (0.15,0) arc(0:180:0.15) to[out=-90,in=120] (0.15,-0.4);
            \draw[wipe] (-0.15,-0.4) to[out=60,in=-90] (0.15,0);
            \draw (-0.15,-0.4) to[out=60,in=-90] (0.15,0);
        \end{tikzpicture}
        = t^{-1} \
        \capmor
        \ ,\qquad
        \begin{tikzpicture}[anchorbase]
            \draw (-0.15,0) to[out=90,in=-120] (0.15,0.4);
            \draw[wipe] (-0.15,0.4) to[out=-60,in=90] (0.15,0);
            \draw (-0.15,0.4) to[out=-60,in=90] (0.15,0) arc(0:-180:0.15);
        \end{tikzpicture}
        = t \
        \cupmor
        \ ,\qquad
        \begin{tikzpicture}[anchorbase]
            \draw (-0.15,0.4) to[out=-60,in=90] (0.15,0) arc(0:-180:0.15);
            \draw[wipe] (-0.15,0) to[out=90,in=-120] (0.15,0.4);
            \draw (-0.15,0) to[out=90,in=-120] (0.15,0.4);
        \end{tikzpicture}
        = t^{-1} \
        \cupmor
        \ ,\qquad
        \begin{tikzpicture}[anchorbase]
            \draw (-0.2,0.3) -- (-0.2,0.1) arc(180:360:0.2) -- (0.2,0.3);
            \singdot[black]{-0.2,0.1};
        \end{tikzpicture}
        \ =\
        \begin{tikzpicture}[anchorbase]
            \draw (-0.2,0.3) -- (-0.2,0.1) arc(180:360:0.2) -- (0.2,0.3);
            \multdot[black]{0.2,0.1}{west}{-1};
        \end{tikzpicture}
        ,
        \\ \label{nuggets}
        \begin{tikzpicture}[centerzero]
            \draw (-0.3,0.3) \braiddown (0,-0.3);
            \draw[wipe] (-0.2,-0.3) -- (-0.2,-0.1) arc(180:0:0.2) -- (0.2,-0.3);
            \draw (-0.2,-0.3) -- (-0.2,-0.1) arc(180:0:0.2) -- (0.2,-0.3);
        \end{tikzpicture}
        =
        \begin{tikzpicture}[centerzero]
            \draw (0.3,0.3) \braiddown (0,-0.3);
            \draw[wipe] (-0.2,-0.3) -- (-0.2,-0.1) arc(180:0:0.2) -- (0.2,-0.3);
            \draw (-0.2,-0.3) -- (-0.2,-0.1) arc(180:0:0.2) -- (0.2,-0.3);
        \end{tikzpicture}
        \ ,\qquad
        \begin{tikzpicture}[centerzero]
            \draw (-0.2,0.3) -- (-0.2,0.1) arc(180:360:0.2) -- (0.2,0.3);
            \draw[wipe] (-0.3,-0.3) \braidup (0,0.3);
            \draw (-0.3,-0.3) \braidup (0,0.3);
        \end{tikzpicture}
        =
        \begin{tikzpicture}[centerzero]
            \draw (-0.2,0.3) -- (-0.2,0.1) arc(180:360:0.2) -- (0.2,0.3);
            \draw[wipe] (0.3,-0.3) \braidup (0,0.3);
            \draw (0.3,-0.3) \braidup (0,0.3);
        \end{tikzpicture}
        \ ,\qquad
        \begin{tikzpicture}[centerzero]
            \draw (-0.3,-0.3) \braidup (0,0.3);
            \draw[wipe] (-0.2,0.3) -- (-0.2,0.1) arc(180:360:0.2) -- (0.2,0.3);
            \draw (-0.2,0.3) -- (-0.2,0.1) arc(180:360:0.2) -- (0.2,0.3);
        \end{tikzpicture}
        =
        \begin{tikzpicture}[centerzero]
            \draw (0.3,-0.3) \braidup (0,0.3);
            \draw[wipe] (-0.2,0.3) -- (-0.2,0.1) arc(180:360:0.2) -- (0.2,0.3);
            \draw (-0.2,0.3) -- (-0.2,0.1) arc(180:360:0.2) -- (0.2,0.3);
        \end{tikzpicture}
        \ .
    \end{gather}
\end{lem}

\begin{proof}
    Since these relations follow from standard arguments, we give only a sketch of the proof.  The first relation in \cref{chicken} follows from the third equality in \cref{drops} by attaching $\poscross$ to the bottom of both diagrams, then using the first equality in \cref{braid}.  To prove the first equality in \cref{nuggets}, we compute
    \[
        \begin{tikzpicture}[centerzero]
            \draw (-0.3,0.3) \braiddown (0,-0.3);
            \draw[wipe] (-0.2,-0.3) -- (-0.2,-0.1) arc(180:0:0.2) -- (0.2,-0.3);
            \draw (-0.2,-0.3) -- (-0.2,-0.1) arc(180:0:0.2) -- (0.2,-0.3);
        \end{tikzpicture}
        \overset{\cref{skein}}{=}
        \begin{tikzpicture}[centerzero]
            \draw (-0.2,-0.3) -- (-0.2,-0.1) arc(180:0:0.2) -- (0.2,-0.3);
            \draw[wipe] (-0.3,0.3) \braiddown (0,-0.3);
            \draw (-0.3,0.3) \braiddown (0,-0.3);
        \end{tikzpicture}
        + z
        \left(\,
            \begin{tikzpicture}[centerzero]
                \draw (-0.3,-0.3) -- (-0.3,0.3);
                \draw (-0.1,-0.3) -- (-0.1,-0.1) arc(180:0:0.2) -- (0.3,-0.3);
            \end{tikzpicture}
            -
            \begin{tikzpicture}[centerzero]
                \draw (-0.3,-0.3) -- (-0.3,-0.1) arc(180:0:0.2) -- (0.1,-0.3);
                \draw (0.3,-0.3) -- (0.3,0.3);
            \end{tikzpicture}
        \, \right)
        \overset{\cref{drops}}{=}
        \begin{tikzpicture}[centerzero]
            \draw (-0.2,-0.3) -- (-0.2,-0.1) arc(180:0:0.2) -- (0.2,-0.3);
            \draw[wipe] (0.3,0.3) \braiddown (0,-0.3);
            \draw (0.3,0.3) \braiddown (0,-0.3);
        \end{tikzpicture}
        + z
        \left(\,
            \begin{tikzpicture}[centerzero]
                \draw (-0.3,-0.3) -- (-0.3,0.3);
                \draw (-0.1,-0.3) -- (-0.1,-0.1) arc(180:0:0.2) -- (0.3,-0.3);
            \end{tikzpicture}
            -
            \begin{tikzpicture}[centerzero]
                \draw (-0.3,-0.3) -- (-0.3,-0.1) arc(180:0:0.2) -- (0.1,-0.3);
                \draw (0.3,-0.3) -- (0.3,0.3);
            \end{tikzpicture}
        \, \right)
        \overset{\cref{skein}}{=}
        \begin{tikzpicture}[centerzero]
            \draw (0.3,0.3) \braiddown (0,-0.3);
            \draw[wipe] (-0.2,-0.3) -- (-0.2,-0.1) arc(180:0:0.2) -- (0.2,-0.3);
            \draw (-0.2,-0.3) -- (-0.2,-0.1) arc(180:0:0.2) -- (0.2,-0.3);
        \end{tikzpicture}
        \ .
    \]
    The remaining relations following from rotation using cups and caps.
\end{proof}

We define
\[
    \begin{tikzpicture}[centerzero]
        \draw (0,-0.2) -- (0,0.2);
        \multdot[black]{0,0}{east}{n};
    \end{tikzpicture}
    :=
    \left( \dotstrand[black] \right)^{\circ n},
    \qquad
    \begin{tikzpicture}[centerzero]
        \draw (0,-0.2) -- (0,0.2);
        \multdot[black]{0,0}{east}{-n};
    \end{tikzpicture}
    :=
    \left( \multdotstrand[black]{east}{-1} \right)^{\circ n},
    \qquad n \in \N.
\]
We then adopt the same generating function conventions as in \cref{sec:AB}, using $x$ to denote the dot.  The difference here is that the dot is now invertible.  Noting that
\[
    u(u-x)^{-1} = 1 + u^{-1}x + u^{-2}x^2 + \dotsb \in \kk[x,x^{-1}] \llbracket u^{-1} \rrbracket,
\]
it is convenient to define
\[
    \begin{tikzpicture}[centerzero]
        \draw (0,-0.3) -- (0,0.3);
        \uptriforce[black]{0,0};
    \end{tikzpicture}
    :=
    \begin{tikzpicture}[centerzero]
        \draw (0,-0.3) -- (0,0.3);
        \multdot[black]{0,0}{west}{u(u-x)^{-1}};
    \end{tikzpicture}
    \qquad \text{and} \qquad
    \begin{tikzpicture}[centerzero]
        \draw (0,-0.3) -- (0,0.3);
        \downtriforce[black]{0,0};
    \end{tikzpicture}
    :=
    \begin{tikzpicture}[centerzero]
        \draw (0,-0.3) -- (0,0.3);
        \multdot[black]{0,0}{west}{u(u-x^{-1})^{-1}};
    \end{tikzpicture}
    .
\]
It follows that
\begin{equation} \label{blackrot}
    \begin{tikzpicture}[anchorbase]
        \draw (-0.2,-0.3) -- (-0.2,-0.1) arc(180:0:0.2) -- (0.2,-0.3);
        \uptriforce[black]{-0.2,-0.1};
    \end{tikzpicture}
    \ = \
    \begin{tikzpicture}[anchorbase]
        \draw (-0.2,-0.3) -- (-0.2,-0.1) arc(180:0:0.2) -- (0.2,-0.3);
        \downtriforce[black]{0.2,-0.1};
    \end{tikzpicture}
    \ ,\qquad
    \begin{tikzpicture}[anchorbase]
        \draw (-0.2,-0.3) -- (-0.2,-0.1) arc(180:0:0.2) -- (0.2,-0.3);
        \downtriforce[black]{-0.2,-0.1};
    \end{tikzpicture}
    \ = \
    \begin{tikzpicture}[anchorbase]
        \draw (-0.2,-0.3) -- (-0.2,-0.1) arc(180:0:0.2) -- (0.2,-0.3);
        \uptriforce[black]{0.2,-0.1};
    \end{tikzpicture}
    \ ,\qquad
    \begin{tikzpicture}[anchorbase]
        \draw (-0.2,0.3) -- (-0.2,0.1) arc(180:360:0.2) -- (0.2,0.3);
        \uptriforce[black]{-0.2,0.1};
    \end{tikzpicture}
    \ = \
    \begin{tikzpicture}[anchorbase]
        \draw (-0.2,0.3) -- (-0.2,0.1) arc(180:360:0.2) -- (0.2,0.3);
        \downtriforce[black]{0.2,0.1};
    \end{tikzpicture}
    \ ,\qquad
    \begin{tikzpicture}[anchorbase]
        \draw (-0.2,0.3) -- (-0.2,0.1) arc(180:360:0.2) -- (0.2,0.3);
        \downtriforce[black]{-0.2,0.1};
    \end{tikzpicture}
    \ = \
    \begin{tikzpicture}[anchorbase]
        \draw (-0.2,0.3) -- (-0.2,0.1) arc(180:360:0.2) -- (0.2,0.3);
        \uptriforce[black]{0.2,0.1};
    \end{tikzpicture}
    \ .
\end{equation}
Using the identity
\[
    \frac{u^2}{(u-x)(u-x^{-1})}
    = \frac{u^2}{u^2-1} \left( \frac{u}{u-x} + \frac{u}{u-x^{-1}} - 1 \right),
\]
we see that
\begin{equation} \label{blackdub}
    \begin{tikzpicture}[centerzero]
        \draw (0,-0.35) -- (0,0.35);
        \uptriforce[black]{0,0.15};
        \downtriforce[black]{0,-0.15};
    \end{tikzpicture}
    =
    \begin{tikzpicture}[centerzero]
        \draw (0,-0.35) -- (0,0.35);
        \uptriforce[black]{0,-0.15};
        \downtriforce[black]{0,0.15};
    \end{tikzpicture}
    = \tfrac{u^2}{u^2-1}
    \left(
        \begin{tikzpicture}[centerzero]
            \draw (0,-0.2) -- (0,0.2);
            \uptriforce[black]{0,0};
        \end{tikzpicture}
        +
        \begin{tikzpicture}[centerzero]
           \draw (0,-0.2) -- (0,0.2);
            \downtriforce[black]{0,0};
        \end{tikzpicture}
        - \idstrand\
    \right).
\end{equation}

For any polynomial $p(u) \in \kk[u]$, we have
\begin{equation} \label{blacktrick}
    \begin{tikzpicture}[centerzero]
        \draw (0,-0.3) -- (0,0.3);
        \multdot[black]{0,0}{east}{p(x)};
    \end{tikzpicture}
    =
    \left[
        \begin{tikzpicture}[centerzero]
            \draw (0,-0.3) -- (0,0.3);
            \uptriforce[black]{0,0};
        \end{tikzpicture}
        \ p(u)
    \right]_{u^0},
    \qquad
    \begin{tikzpicture}[centerzero]
        \draw (0,-0.3) -- (0,0.3);
        \multdot[black]{0,0}{east}{p(x^{-1})};
    \end{tikzpicture}
    =
    \left[
        \begin{tikzpicture}[centerzero]
            \draw (0,-0.3) -- (0,0.3);
            \downtriforce[black]{0,0};
        \end{tikzpicture}
        \ p(u)
    \right]_{u^0}.
\end{equation}
Just as in \cref{trick+}, we can generalize to the case $p(u) \in \kk\Laurent{u^{-1}}$ using the truncation defined in \cref{bmx};  we find that
\begin{equation} \label{blacktrick+}
    \begin{tikzpicture}[centerzero]
        \draw (0,-0.3) -- (0,0.3);
        \multdot[black]{0,0}{east}{[p]_{\ge 0}(x)};
    \end{tikzpicture}
    =
    \left[
        \begin{tikzpicture}[centerzero]
            \draw (0,-0.3) -- (0,0.3);
            \uptriforce[black]{0,0};
        \end{tikzpicture}
        \ p(u)
    \right]_{u^0}
    , \qquad
    \begin{tikzpicture}[centerzero]
        \draw (0,-0.3) -- (0,0.3);
        \multdot[black]{0,0}{east}{[p]_{\ge 0}(x^{-1})};
    \end{tikzpicture}
    =
    \left[
        \begin{tikzpicture}[centerzero]
            \draw (0,-0.3) -- (0,0.3);
            \downtriforce[black]{0,0};
        \end{tikzpicture}
        \ p(u)
    \right]_{u^0}
    .
\end{equation}

\begin{lem}
    The following relations hold in $\AK$:
    \begin{equation} \label{blackslide}
        \begin{tikzpicture}[anchorbase]
            \draw (0.3,-0.5) -- (-0.3,0.5);
            \draw[wipe] (-0.3,-0.5) -- (0.3,0.5);
            \draw (-0.3,-0.5) -- (0.3,0.5);
            \uptriforce[black]{-0.16,-0.25};
        \end{tikzpicture}
        -
        \begin{tikzpicture}[anchorbase]
            \draw (-0.3,-0.5) -- (0.3,0.5);
            \draw[wipe] (0.3,-0.5) -- (-0.3,0.5);
            \draw (0.3,-0.5) -- (-0.3,0.5);
            \uptriforce[black]{0.19,0.33};
        \end{tikzpicture}
        = z
        \left(
            \begin{tikzpicture}[anchorbase]
                \draw (-0.2,-0.5) -- (-0.2,0.5);
                \draw (0.2,-0.5) -- (0.2,0.5);
                \uptriforce[black]{-0.2,0};
                \uptriforce[black]{0.2,0};
            \end{tikzpicture}
            -
            \begin{tikzpicture}[anchorbase]
                \draw (-0.2,-0.5) -- (-0.2,-0.3) arc(180:0:0.2) -- (0.2,-0.5);
                \draw (-0.2,0.5) -- (-0.2,0.3) arc(180:360:0.2) -- (0.2,0.5);
                \uptriforce[black]{0.2,0.3};
                \uptriforce[black]{-0.2,-0.3};
            \end{tikzpicture}
        \right),
        \qquad
        \begin{tikzpicture}[anchorbase]
            \draw (0.3,-0.5) -- (-0.3,0.5);
            \draw[wipe] (-0.3,-0.5) -- (0.3,0.5);
            \draw (-0.3,-0.5) -- (0.3,0.5);
            \downtriforce[black]{0.19,-0.33};
        \end{tikzpicture}
        -
        \begin{tikzpicture}[anchorbase]
            \draw (-0.3,-0.5) -- (0.3,0.5);
            \draw[wipe] (0.3,-0.5) -- (-0.3,0.5);
            \draw (0.3,-0.5) -- (-0.3,0.5);
            \downtriforce[black]{-0.16,0.25};
        \end{tikzpicture}
        = z
        \left(
            \begin{tikzpicture}[anchorbase]
                \draw (-0.2,-0.5) -- (-0.2,0.5);
                \draw (0.2,-0.5) -- (0.2,0.5);
                \downtriforce[black]{-0.2,0};
                \downtriforce[black]{0.2,0};
            \end{tikzpicture}
            -
            \begin{tikzpicture}[anchorbase]
                \draw (-0.2,-0.5) -- (-0.2,-0.3) arc(180:0:0.2) -- (0.2,-0.5);
                \draw (-0.2,0.5) -- (-0.2,0.3) arc(180:360:0.2) -- (0.2,0.5);
                \downtriforce[black]{-0.2,0.3};
                \downtriforce[black]{0.2,-0.3};
            \end{tikzpicture}
        \right).
    \end{equation}
\end{lem}

\begin{proof}
By the first equalities of \cref{skein} and \cref{kauffdot}, we have
    \[
    \begin{tikzpicture}[anchorbase]
            \draw (0.3,-0.5) -- (-0.3,0.5);
            \draw[wipe] (-0.3,-0.5) -- (0.3,0.5);
            \draw (-0.3,-0.5) -- (0.3,0.5);
            \multdot[black]{0.15,0.25}{west}{\frac{u-x}{u}};
        \end{tikzpicture}-    \begin{tikzpicture}[anchorbase]
            \draw (-0.3,-0.5) -- (0.3,0.5);
            \draw[wipe] (0.3,-0.5) -- (-0.3,0.5);
            \draw (0.3,-0.5) -- (-0.3,0.5);
            \multdot[black]{-0.15,-0.25}{east}{\frac{u-x}{u}};
        \end{tikzpicture}
        = z
        \left(
            \begin{tikzpicture}[anchorbase]
                \draw (-0.2,-0.5) -- (-0.2,0.5);
                \draw (0.2,-0.5) -- (0.2,0.5);
            \end{tikzpicture}
            -
            \begin{tikzpicture}[anchorbase]
                \draw (-0.2,-0.5) -- (-0.2,-0.3) arc(180:0:0.2) -- (0.2,-0.5);
                \draw (-0.2,0.5) -- (-0.2,0.3) arc(180:360:0.2) -- (0.2,0.5);
            \end{tikzpicture}
        \right).
    \]
    The first equality in \cref{blackslide} follows from multiplication by $\begin{tikzpicture}[centerzero]
            \draw (0,-0.2) -- (0,0.2);
            \uptriforce[black]{0,0};
        \end{tikzpicture}$ at the bottom left and top right on both sides of the equation.
    % Repeated use of the first relation in \cref{kauffdot} and the skein relation \cref{skein} shows that
    % \[
    %     \begin{tikzpicture}[anchorbase]
    %         \draw (0.3,-0.5) -- (-0.3,0.5);
    %         \draw[wipe] (-0.3,-0.5) -- (0.3,0.5);
    %         \draw (-0.3,-0.5) -- (0.3,0.5);
    %         \multdot[black]{-0.15,-0.25}{east}{n};
    %     \end{tikzpicture}
    %     -
    %     \begin{tikzpicture}[anchorbase]
    %         \draw (-0.3,-0.5) -- (0.3,0.5);
    %         \draw[wipe] (0.3,-0.5) -- (-0.3,0.5);
    %         \draw (0.3,-0.5) -- (-0.3,0.5);
    %         \multdot[black]{0.15,0.25}{west}{n};
    %     \end{tikzpicture}
    %     = \sum_{r=0}^n
    %     \begin{tikzpicture}[anchorbase]
    %         \draw (-0.2,-0.5) -- (-0.2,0.5);
    %         \draw (0.2,-0.5) -- (0.2,0.5);
    %         \multdot[black]{-0.2,0}{east}{r};
    %         \multdot[black]{0.2,0}{west}{n-r};
    %     \end{tikzpicture}
    %     - \sum_{r=0}^n
    %     \begin{tikzpicture}[anchorbase]
    %         \draw (-0.2,-0.5) -- (-0.2,-0.3) arc(180:0:0.2) -- (0.2,-0.5);
    %         \draw (-0.2,0.5) -- (-0.2,0.3) arc(180:360:0.2) -- (0.2,0.5);
    %         \multdot[black]{0.2,0.3}{west}{n-r};
    %         \multdot[black]{-0.2,-0.3}{east}{r};
    %     \end{tikzpicture}
    %     \ ,\qquad n \in \N.
    % \]
    % Multiplying by $u^{-n}$ and summing over $n \in \N$ yields the first equality in \cref{blackslide}.

    To obtain the second equality in \cref{blackslide}, we then attach a cap to the top of the rightmost strand, attach a cup to the bottom of the leftmost strand, and then use \cref{blackrot}.
\end{proof}

Define
\begin{equation} \label{eatzk}
    \begin{aligned}
        \bubblegenr[black]{u} &:= \frac{(t^{-1}-z)u^2 - t^{-1}}{u^2-zu-1} 1_\one + \frac{z(u^2-1)}{u^2-zu-1}\ \uptribubbler[black]
        \in t 1_\one + u^{-1} \End_{\AK(z,t)}(\one) \llbracket u^{-1} \rrbracket,
        \\
        \bubblegenl[black]{u} &:= \frac{(t+z)u^2-t}{u^2+zu-1} 1_\one  - \frac{z(u^2-1)}{u^2+zu-1}\ \uptribubblel[black]
        \in t^{-1} 1_\one + u^{-1} \End_{\AK(z,t)}(\one) \llbracket u^{-1} \rrbracket.
    \end{aligned}
\end{equation}

\begin{lem}\label{swap}
    There is an isomorphism of $\kk$-linear monoidal categories
    \[
        \beta \colon \AK(z,t) \xrightarrow{\cong} \AK(-z,t^{-1}),
    \]
    given by flipping crossings and taking inverses of dots.  More precisely, $\beta$ sends $\gok$ to $\gok$ and acts on the generating morphisms by
    \[
        \capmor \mapsto \capmor\, ,\qquad
        \cupmor \mapsto \cupmor\, ,\qquad
        \poscross \mapsto \negcross,\qquad
        \negcross \mapsto \poscross,\qquad
        \dotstrand[black] \mapsto \multdotstrand[black]{east}{-1}\, ,\qquad
        \multdotstrand[black]{east}{-1} \mapsto \dotstrand[black]\, .
    \]
    We also have
    \[
        \begin{tikzpicture}[centerzero]
            \draw (0,-0.3) -- (0,0.3);
            \uptriforce[black]{0,0};
        \end{tikzpicture}
        \mapsto
        \begin{tikzpicture}[centerzero]
            \draw (0,-0.3) -- (0,0.3);
            \downtriforce[black]{0,0};
        \end{tikzpicture}
        \ ,\qquad
        \begin{tikzpicture}[centerzero]
            \draw (0,-0.3) -- (0,0.3);
            \downtriforce[black]{0,0};
        \end{tikzpicture}
        \mapsto
        \begin{tikzpicture}[centerzero]
            \draw (0,-0.3) -- (0,0.3);
            \uptriforce[black]{0,0};
        \end{tikzpicture}
        \ ,\qquad
        \bubblegenl[black]{u} \mapsto \bubblegenr[black]{u}
        \ ,\qquad
        \bubblegenr[black]{u} \mapsto \bubblegenl[black]{u}
        \ .
    \]
\end{lem}

\begin{proof}
    This is a straightforward verification of the defining relations.
\end{proof}

We call $\beta$ the \emph{bar involution}.

\begin{prop}
    The following relation holds in $\AK$:
    \begin{equation} \label{infgrassk}
        \bubblegenl[black]{u} \ \bubblegenr[black]{u} = 1_\one.
    \end{equation}
\end{prop}

\begin{proof}
    To simplify notation, we will drop $1_\one$, identifying a scalar $c \in \kk$ with $c 1_\one$ in what follows.  We have
    \begin{multline*}
        t^{-1}\, \uptribubblel[black]
        \overset{\cref{chicken}}{=}
        \begin{tikzpicture}[centerzero]
            \draw (0.2,0.25) arc(0:180:0.2) \braiddown (0.2,-0.25) arc(360:180:0.2);
            \draw[wipe] (-0.2,-0.25) \braidup (0.2,0.25);
            \draw (-0.2,-0.25) \braidup (0.2,0.25);
            \uptriforce[black]{-0.2,-0.25};
        \end{tikzpicture}
        \overset{\cref{blackslide}}{=}
        \begin{tikzpicture}[centerzero]
            \draw (-0.2,-0.25) \braidup (0.2,0.25) arc(0:180:0.2);
            \draw[wipe] (-0.2,0.25) \braiddown (0.2,-0.25) arc(360:180:0.2);
            \draw (-0.2,0.25) \braiddown (0.2,-0.25) arc(360:180:0.2);
            \uptriforce[black]{0.2,0.25};
        \end{tikzpicture}
        + z
        \left(
            \begin{tikzpicture}[centerzero]
                \draw (0.2,-0.15) -- (0.2,0.15) arc(0:180:0.2) -- (-0.2,-0.15) arc(180:360:0.2);
                \uptriforce[black]{0.2,0};
                \uptriforce[black]{-0.2,0};
            \end{tikzpicture}
            - \uptribubblel[black] \ \uptribubbler[black]
        \right)
        \overset{\cref{chicken}}{\underset{\cref{blackrot}}{=}}
        t\ \uptribubbler[black]
        + z
        \left(
            \begin{tikzpicture}[centerzero]
                \draw (0.2,-0.15) -- (0.2,0.15) arc(0:180:0.2) -- (-0.2,-0.15) arc(180:360:0.2);
                \downtriforce[black]{0.2,0.15};
                \uptriforce[black]{0.2,-0.15};
            \end{tikzpicture}
            - \uptribubblel[black] \ \uptribubbler[black]
        \right)
        \\
        \overset{\cref{blackdub}}{=}
        t\ \uptribubbler[black]
        + \tfrac{zu^2}{u^2-1} \left( \uptribubbler[black] + \downtribubbler[black] - \bubble \right)
        - z \left( \uptribubblel[black] \ \uptribubbler[black] \right)
        \\
        \overset{\cref{blackrot}}{\underset{\cref{skein}}{=}}
        t\ \uptribubbler[black]
        + \tfrac{zu^2}{u^2-1} \left( \uptribubbler[black] + \uptribubblel[black] - z^{-1}t + z^{-1}t^{-1} - 1 \right)
        - z \left( \uptribubblel[black] \ \uptribubbler[black] \right).
    \end{multline*}
    Thus
    \[
        \uptribubblel[black] \ \uptribubbler[black]
        + \left( z^{-1}t^{-1} - \tfrac{u^2}{u^2-1} \right)\, \uptribubblel[black]
        + \left( - z^{-1}t - \tfrac{u^2}{u^2-1} \right)\, \uptribubbler[black]
        + \tfrac{u^2}{u^2-1} \left( z^{-1}t - z^{-1}t^{-1} \right)
        = - \tfrac{u^2}{u^2-1},
    \]
    which implies that
    \[
        \left( \uptribubblel[black] - z^{-1}t - \tfrac{u^2}{u^2-1} \right)
        \left( \uptribubbler[black] + z^{-1}t^{-1} - \tfrac{u^2}{u^2-1} \right)
        = \tfrac{u^4}{(u^2-1)^2} - \tfrac{u^2}{u^2-1} - z^{-2}
        = \tfrac{u^2}{(u^2-1)^2} - z^{-2}.
    \]
    Multiplying both sides by $\frac{z^2(u^2-1)^2}{z^2u^2-(u^2-1)^2}$ and using \cref{eatzk} then gives \cref{infgrassk}.
\end{proof}

\begin{rem}
    When the bubbles $\multbubbler[black]{r}$, $r \in \Z$, are evaluated at scalars, the relation \cref{infgrassk} corresponds to \cite[(2.30)]{RX09}, noting that $\delta$ and $\varrho$ of \cite{RX09} correspond to our $z$ and $t^{-1}$, respectively.  We will discuss these connections in more detail in \cref{sec:admissibleKauffman}.
\end{rem}

\begin{prop}
    The following relations hold in $\AK$:
    \begin{equation} \label{blackcurl}
        \begin{tikzpicture}[anchorbase]
            \draw (0.3,0.2) to[out=0,in=up] (0.45,0) to[out=down,in=0] (0.3,-0.2) to[out=180,in=down] (0,0.5);
            \draw[wipe] (0,-0.5) to[out=up,in=180] (0.3,0.2);
            \draw (0,-0.5) to[out=up,in=180] (0.3,0.2);
            \uptriforce[black]{0.45,0};
        \end{tikzpicture}
        = \frac{u^2-zu-1}{u^2-1}\
        \begin{tikzpicture}[centerzero]
            \draw (0,-0.5) -- (0,0.5);
            \downtriforce[black]{0,0};
            \bubgenr[black]{u}{0.4,0};
        \end{tikzpicture}
        + \frac{zu^2}{u^2-1}
        \left(\,
            \begin{tikzpicture}[centerzero]
                \draw (0,-0.5) -- (0,0.5);
            \end{tikzpicture}
            -
            \begin{tikzpicture}[centerzero]
                \draw (0,-0.5) -- (0,0.5);
                \uptriforce[black]{0,0};
            \end{tikzpicture}
        \right),
        \qquad
        \begin{tikzpicture}[anchorbase]
            \draw (0,-0.5) to[out=up,in=180] (0.3,0.2) to[out=0,in=up] (0.45,0) to[out=down,in=0] (0.3,-0.2);
            \draw[wipe] (0.3,-0.2) to[out=180,in=down] (0,0.5);
            \draw (0.3,-0.2) to[out=180,in=down] (0,0.5);
            \downtriforce[black]{0.45,0};
        \end{tikzpicture}
        = \frac{u^2+zu-1}{u^2-1}
        \begin{tikzpicture}[centerzero]
            \draw (0,-0.5) -- (0,0.5);
            \uptriforce[black]{0,0};
            \bubgenl[black]{u}{0.45,0};
        \end{tikzpicture}
        - \frac{zu^2}{u^2-1}
        \left(\,
            \begin{tikzpicture}[centerzero]
                \draw (0,-0.5) -- (0,0.5);
            \end{tikzpicture}
            -
            \begin{tikzpicture}[centerzero]
                \draw (0,-0.5) -- (0,0.5);
                \downtriforce[black]{0,0};
            \end{tikzpicture}
        \right).
    \end{equation}
\end{prop}

\begin{proof}
    We have
    \begin{multline*}
        \begin{tikzpicture}[anchorbase]
            \draw (0.3,0.2) to[out=0,in=up] (0.45,0) to[out=down,in=0] (0.3,-0.2) to[out=180,in=down] (0,0.5);
            \draw[wipe] (0,-0.5) to[out=up,in=180] (0.3,0.2);
            \draw (0,-0.5) to[out=up,in=180] (0.3,0.2);
            \uptriforce[black]{0.45,0};
        \end{tikzpicture}
        \overset{\cref{blackrot}}{\underset{\cref{blackslide}}{=}}
        \begin{tikzpicture}[anchorbase]
            \draw (0,-0.5) to[out=up,in=180] (0.3,0.2) to[out=0,in=up] (0.45,0) to[out=down,in=0] (0.3,-0.2);
            \draw[wipe] (0.3,-0.2) to[out=180,in=down] (0,0.5);
            \draw (0.3,-0.2) to[out=180,in=down] (0,0.5);
            \downtriforce[black]{0,0.3};
        \end{tikzpicture}
        + z
        \left(
            \begin{tikzpicture}[centerzero]
                \draw (0,-0.5) -- (0,0.5);
                \downtriforce[black]{0,0};
                \downtribubl[black]{0.5,0};
            \end{tikzpicture}
            -
            \begin{tikzpicture}[centerzero]
                \draw (0,-0.5) -- (0,-0.2) arc(180:0:0.15) arc(180:360:0.15) -- (0.6,0.2) arc(0:180:0.15) arc(360:180:0.15) -- (0,0.5);
                \downtriforce[black]{0.3,-0.2};
                \downtriforce[black]{0,0.3};
            \end{tikzpicture}
        \right)
        \overset{\cref{blackrot}}{\underset{\cref{chicken}}{=}} t^{-1}\,
        \begin{tikzpicture}[centerzero]
            \draw (0,-0.5) -- (0,0.5);
            \downtriforce[black]{0,0};
        \end{tikzpicture}
        + z
        \left(
            \begin{tikzpicture}[centerzero]
                \draw (0,-0.5) -- (0,0.5);
                \downtriforce[black]{0,0};
                \downtribubl[black]{0.5,0};
            \end{tikzpicture}
            -
            \begin{tikzpicture}[centerzero]
                \draw (0,-0.5) -- (0,0.5);
                \downtriforce[black]{0,0.2};
                \uptriforce[black]{0,-0.2};
            \end{tikzpicture}
        \right)
        \\
        \overset{\cref{blackdub}}{=}
        \begin{tikzpicture}[centerzero]
            \draw (0,-0.5) -- (0,0.5);
            \downtriforce[black]{0,0};
        \end{tikzpicture}
        \left(
            t^{-1} + z\, \downtribubblel[black] - \frac{zu^2}{u^2-1}
        \right)
        + \frac{zu^2}{u^2-1}
        \left(\,
            \begin{tikzpicture}[centerzero]
                \draw (0,-0.5) -- (0,0.5);
            \end{tikzpicture}
            -
            \begin{tikzpicture}[centerzero]
                \draw (0,-0.5) -- (0,0.5);
                \uptriforce[black]{0,0};
            \end{tikzpicture}
        \right)
        \overset{\cref{eatzk}}{=} \frac{u^2-zu-1}{u^2-1}\
        \begin{tikzpicture}[centerzero]
            \draw (0,-0.5) -- (0,0.5);
            \downtriforce[black]{0,0};
            \bubgenr[black]{u}{0.4,0};
        \end{tikzpicture}
        + \frac{zu^2}{u^2-1}
        \left(\,
            \begin{tikzpicture}[centerzero]
                \draw (0,-0.5) -- (0,0.5);
            \end{tikzpicture}
            -
            \begin{tikzpicture}[centerzero]
                \draw (0,-0.5) -- (0,0.5);
                \uptriforce[black]{0,0};
            \end{tikzpicture}
        \right).
    \end{multline*}
    The second relation then follows after applying the bar involution $\beta$.
\end{proof}

\begin{prop}
    The following relations hold in $\AK$:
    \begin{equation} \label{bubslidek}
        \begin{tikzpicture}[anchorbase]
            \draw (0,-0.5) -- (0,0.5);
            \bubgenr[black]{u}{-0.4,0};
        \end{tikzpicture}
        \ =
        \begin{tikzpicture}[anchorbase]
            \draw (0,-0.5) -- (0,0.5);
            \multdot[black]{0,0}{east}{\frac{1-z^2ux^{-1}(u-x^{-1})^{-2}}{1-z^2ux(u-x)^{-2}}};
            \bubgenr[black]{u}{0.4,0};
        \end{tikzpicture}
        \ , \qquad
        \begin{tikzpicture}[anchorbase]
            \draw (0,-0.5) -- (0,0.5);
            \bubgenl[black]{u}{-0.4,0};
        \end{tikzpicture}
        \ =
        \begin{tikzpicture}[anchorbase]
            \draw (0,-0.5) -- (0,0.5);
            \multdot[black]{0,0}{east}{\frac{1-z^2ux(u-x)^{-2}}{1-z^2ux^{-1}(u-x^{-1})^{-2}}};
            \bubgenl[black]{u}{0.5,0};
        \end{tikzpicture}
        \ .
    \end{equation}
\end{prop}

\begin{proof}
    To prove the first equality in \cref{bubslidek}, we compute
    \begin{multline*}
        \begin{tikzpicture}[centerzero]
            \draw (0,-0.6) -- (0,0.6);
            \uptribubr[black]{-0.5,0};
        \end{tikzpicture}
        \ \overset{\cref{braid}}{=}\
        \begin{tikzpicture}[centerzero]
            \draw (0,-0.6) -- (0,0.6);
            \draw[wipe] (-0.5,0) to[out=-90,in=180] (-0.3,-0.4) to[out=0,in=-90,looseness=0.5] (0.2,0.2) to[out=up,in=up] (-0.5,0);
            \draw (-0.5,0) to[out=-90,in=180] (-0.3,-0.4) to[out=0,in=-90,looseness=0.5] (0.2,0.2) to[out=up,in=up] (-0.5,0);
            \uptriforce[black]{-0.15,-0.3};
        \end{tikzpicture}
        \ \overset{\cref{blackslide}}{=}\
        \begin{tikzpicture}[centerzero]
            \draw (-0.5,0) to[out=-90,in=180] (-0.3,-0.4) to[out=0,in=-90,looseness=0.5] (0.25,0.2);
            \draw[wipe] (0,-0.6) -- (0,0.6);
            \draw (0,-0.6) -- (0,0.6);
            \draw[wipe] (0.25,0.2) to[out=up,in=up] (-0.5,0);
            \draw (0.25,0.2) to[out=up,in=up] (-0.5,0);
            \uptriforce[black]{0.17,0.05};
        \end{tikzpicture}
        + z\
        \begin{tikzpicture}[centerzero]
            \draw (0,0.6) -- (0,0.45) \braiddown (-0.2,0) -- (-0.2,-0.2) arc(0:-180:0.2) -- (-0.6,0.2) arc(180:90:0.2);
            \draw[wipe] (-0.4,0.4) to[out=right,in=up] (0,-0.4) -- (0,-0.6);
            \draw (-0.4,0.4) to[out=right,in=up] (0,-0.4) -- (0,-0.6);
            \uptriforce[black]{-0.2,-0.1};
            \uptriforce[black]{0,-0.4};
        \end{tikzpicture}
        - z\
        \begin{tikzpicture}[centerzero]
            \draw (0,0.6) to[out=down,in=up] (-0.2,0.15);
            \draw[wipe] (0.1,0.15) to[out=up,in=up] (-0.5,0.25);
            \draw (-0.2,0.15) arc (-180:0:0.15) to[out=up,in=up] (-0.5,0.25) -- (-0.5,-0.25) arc(-180:0:0.15) arc(-180:-360:0.15) -- (0.1,-0.6);
            \uptriforce[black]{0.1,0.15};
            \uptriforce[black]{-0.2,-0.25};
        \end{tikzpicture}
        \overset{\cref{skein}}{\underset{\cref{blackrot}}{=}}
        \begin{tikzpicture}[centerzero]
            \draw (0,-0.6) -- (0,0.6);
            \draw[wipe] (-0.5,0) to[out=-90,in=180] (-0.3,-0.4) to[out=0,in=-90,looseness=0.5] (0.25,0.2);
            \draw (-0.5,0) to[out=-90,in=180] (-0.3,-0.4) to[out=0,in=-90,looseness=0.5] (0.25,0.2);
            \draw[wipe] (0.25,0.2) to[out=up,in=up] (-0.5,0);
            \draw (0.25,0.2) to[out=up,in=up] (-0.5,0);
            \uptriforce[black]{0.17,0.05};
        \end{tikzpicture}
        - z\
        \begin{tikzpicture}[centerzero]
            \draw (-0.45,0) to[out=down,in=left] (-0.3,-0.2) to[out=right,in=down] (0,0.5) -- (0,0.6);
            \draw[wipe] (0,-0.5) to[out=up,in=0] (-0.3,0.2);
            \draw (0,-0.6) -- (0,-0.5) to[out=up,in=0] (-0.3,0.2) to[out=left,in=up] (-0.45,0);
            \uptriforce[black]{0,-0.4};
        \end{tikzpicture}
        + z\
        \begin{tikzpicture}[anchorbase]
            \draw (0.3,0.2) to[out=0,in=up] (0.45,0) to[out=down,in=0] (0.3,-0.2) to[out=180,in=down] (0,0.5);
            \draw[wipe] (0,-0.5) to[out=up,in=180] (0.3,0.2);
            \draw (0,-0.5) to[out=up,in=180] (0.3,0.2);
            \uptriforce[black]{0.45,0};
        \end{tikzpicture}
        + z\
        \begin{tikzpicture}[centerzero]
            \draw (0,0.6) -- (0,0.45) \braiddown (-0.2,0) -- (-0.2,-0.2) arc(0:-180:0.2) -- (-0.6,0.2) arc(180:90:0.2);
            \draw[wipe] (-0.4,0.4) to[out=right,in=up] (0,-0.4) -- (0,-0.6);
            \draw (-0.4,0.4) to[out=right,in=up] (0,-0.4) -- (0,-0.6);
            \uptriforce[black]{-0.2,-0.1};
            \uptriforce[black]{0,-0.4};
        \end{tikzpicture}
        - z\
        \begin{tikzpicture}[anchorbase]
            \draw (0.3,0.2) to[out=0,in=up] (0.45,0) to[out=down,in=0] (0.3,-0.2) to[out=180,in=down] (0,0.5);
            \draw[wipe] (0,-0.5) to[out=up,in=180] (0.3,0.2);
            \draw (0,-0.5) to[out=up,in=180] (0.3,0.2);
            \uptriforce[black]{0.45,0};
            \downtriforce[black]{0,-0.35};
        \end{tikzpicture}
        \\
        \overset{\substack{\cref{braid} \\ \cref{chicken}}}{\underset{\substack{\cref{blackcurl} \\ \cref{blackrot}}}{=}}
        \begin{tikzpicture}[centerzero]
            \draw (0,-0.6) -- (0,0.6);
            \uptribubr[black]{0.4,0};
        \end{tikzpicture}
        - z t^{-1}\
        \begin{tikzpicture}[centerzero]
            \draw (0,-0.6) -- (0,0.6);
            \uptriforce[black]{0,0};
        \end{tikzpicture}
        + \frac{zu^2-z^2u-z}{u^2-1}
        \left(
            \begin{tikzpicture}[centerzero]
                \draw (0,-0.5) -- (0,0.5);
                \downtriforce[black]{0,0};
            \end{tikzpicture}
            - \
            \begin{tikzpicture}[centerzero]
                \draw (0,-0.5) -- (0,0.5);
                \downtriforce[black]{0,0.2};
                \downtriforce[black]{0,-0.2};
            \end{tikzpicture}
        \right)
        \bubblegenr[black]{u}
        + \frac{z^2u^2}{u^2-1}
        \left(\,
            \begin{tikzpicture}[centerzero]
                \draw (0,-0.5) -- (0,0.5);
            \end{tikzpicture}
            -
            \begin{tikzpicture}[centerzero]
                \draw (0,-0.5) -- (0,0.5);
                \uptriforce[black]{0,0};
            \end{tikzpicture}
            -
            \begin{tikzpicture}[centerzero]
                \draw (0,-0.5) -- (0,0.5);
                \downtriforce[black]{0,0};
            \end{tikzpicture}
            +
            \begin{tikzpicture}[centerzero]
                \draw (0,-0.5) -- (0,0.5);
                \uptriforce[black]{0,0.2};
                \downtriforce[black]{0,-0.2};
            \end{tikzpicture}
        \right)
        + z\
        \begin{tikzpicture}[centerzero]
            \draw (-0.45,0) to[out=down,in=left] (-0.3,-0.2) to[out=right,in=down] (0,0.5) -- (0,0.6);
            \draw[wipe] (0,-0.5) to[out=up,in=0] (-0.3,0.2);
            \draw (0,-0.6) -- (0,-0.5) to[out=up,in=0] (-0.3,0.2) to[out=left,in=up] (-0.45,0);
            \uptriforce[black]{0,-0.4};
            \downtriforce[black]{-0.45,0};
        \end{tikzpicture}
        \\
        \overset{\cref{blackdub}}{=}
        \begin{tikzpicture}[centerzero]
            \draw (0,-0.6) -- (0,0.6);
            \uptribubr[black]{0.4,0};
        \end{tikzpicture}
        - z t^{-1}\
        \begin{tikzpicture}[centerzero]
            \draw (0,-0.6) -- (0,0.6);
            \uptriforce[black]{0,0};
        \end{tikzpicture}
        + \frac{zu^2-z^2u-z}{u^2-1}
        \left(
            \begin{tikzpicture}[centerzero]
                \draw (0,-0.5) -- (0,0.5);
                \downtriforce[black]{0,0};
            \end{tikzpicture}
            - \
            \begin{tikzpicture}[centerzero]
                \draw (0,-0.5) -- (0,0.5);
                \downtriforce[black]{0,0.2};
                \downtriforce[black]{0,-0.2};
            \end{tikzpicture}
        \right)
        \bubblegenr[black]{u}
        + \frac{z^2u^2}{u^2-1}
        \left(\,
            \begin{tikzpicture}[centerzero]
                \draw (0,-0.5) -- (0,0.5);
            \end{tikzpicture}
            -
            \begin{tikzpicture}[centerzero]
                \draw (0,-0.5) -- (0,0.5);
                \uptriforce[black]{0,0};
            \end{tikzpicture}
            -
            \begin{tikzpicture}[centerzero]
                \draw (0,-0.5) -- (0,0.5);
                \downtriforce[black]{0,0};
            \end{tikzpicture}
            +
            \begin{tikzpicture}[centerzero]
                \draw (0,-0.5) -- (0,0.5);
                \uptriforce[black]{0,0.2};
                \downtriforce[black]{0,-0.2};
            \end{tikzpicture}
        \right)
        + z\
        \begin{tikzpicture}[centerzero]
            \draw (0,-0.6) -- (0,-0.5) to[out=up,in=0] (-0.3,0.2) to[out=left,in=up] (-0.45,0);
            \draw[wipe] (-0.3,-0.2) to[out=right,in=down] (0,0.5);
            \draw (-0.45,0) to[out=down,in=left] (-0.3,-0.2) to[out=right,in=down] (0,0.5) -- (0,0.6);
            \uptriforce[black]{0,-0.4};
            \downtriforce[black]{-0.45,0};
        \end{tikzpicture}
        - z^2\,
        \begin{tikzpicture}[centerzero]
            \draw (0,-0.6) -- (0,0.6);
            \downtribubl[black]{-0.4,0};
            \uptriforce[black]{0,0};
        \end{tikzpicture}
        + z^2\
        \begin{tikzpicture}[centerzero]
            \draw (0,-0.6) -- (0,-0.2) arc(0:180:0.15) arc(360:180:0.15) -- (-0.6,0.2) arc(180:0:0.15) arc(180:360:0.15) -- (0,0.6);
            \downtriforce[black]{-0.6,0};
            \uptriforce[black]{0,-0.3};
        \end{tikzpicture}
        \\
        \overset{\cref{blackcurl}}{\underset{\cref{blackrot}}{=}}
        \begin{tikzpicture}[centerzero]
            \draw (0,-0.6) -- (0,0.6);
            \uptribubr[black]{0.4,0};
        \end{tikzpicture}
        - z t^{-1}\
        \begin{tikzpicture}[centerzero]
            \draw (0,-0.6) -- (0,0.6);
            \uptriforce[black]{0,0};
        \end{tikzpicture}
        + \frac{zu^2-z^2u-z}{u^2-1}
        \left(
            \begin{tikzpicture}[centerzero]
                \draw (0,-0.5) -- (0,0.5);
                \downtriforce[black]{0,0};
            \end{tikzpicture}
            - \
            \begin{tikzpicture}[centerzero]
                \draw (0,-0.5) -- (0,0.5);
                \downtriforce[black]{0,0.2};
                \downtriforce[black]{0,-0.2};
            \end{tikzpicture}
        \right)
        \bubblegenr[black]{u}
        + \frac{z^2u^2}{u^2-1}
        \left(\,
            \begin{tikzpicture}[centerzero]
                \draw (0,-0.5) -- (0,0.5);
            \end{tikzpicture}
            -
            \begin{tikzpicture}[centerzero]
                \draw (0,-0.5) -- (0,0.5);
                \downtriforce[black]{0,0};
            \end{tikzpicture}
        \right)
        + \frac{zu^2-z^2u-z}{u^2-1}\
        \begin{tikzpicture}[centerzero]
            \draw (0,-0.6) -- (0,0.6);
            \bubgenr[black]{u}{-0.5,0};
            \uptriforce[black]{0,0.2};
            \uptriforce[black]{0,-0.2};
        \end{tikzpicture}
        - z^2\,
        \begin{tikzpicture}[centerzero]
            \draw (0,-0.6) -- (0,0.6);
            \downtribubl[black]{-0.4,0};
            \uptriforce[black]{0,0};
        \end{tikzpicture}
        + z^2\
        \begin{tikzpicture}[centerzero]
            \draw (0,-0.6) -- (0,0.6);
            \downtriforce[black]{0,0.2};
            \uptriforce[black]{0,-0.2};
        \end{tikzpicture}
        \\
        \overset{\cref{blackdub}}{\underset{\cref{blackrot}}{=}}
        \begin{tikzpicture}[centerzero]
            \draw (0,-0.6) -- (0,0.6);
            \uptribubr[black]{0.4,0};
        \end{tikzpicture}
        - z t^{-1}\
        \begin{tikzpicture}[centerzero]
            \draw (0,-0.6) -- (0,0.6);
            \uptriforce[black]{0,0};
        \end{tikzpicture}
        + \frac{zu^2-z^2u-z}{u^2-1}
        \left(
            \begin{tikzpicture}[centerzero]
                \draw (0,-0.5) -- (0,0.5);
                \downtriforce[black]{0,0};
            \end{tikzpicture}
            - \
            \begin{tikzpicture}[centerzero]
                \draw (0,-0.5) -- (0,0.5);
                \downtriforce[black]{0,0.2};
                \downtriforce[black]{0,-0.2};
            \end{tikzpicture}
        \right)
        \bubblegenr[black]{u}
        + \frac{zu^2 - z^2u - z}{u^2-1}\
        \begin{tikzpicture}[centerzero]
            \draw (0,-0.6) -- (0,0.6);
            \bubgenr[black]{u}{-0.5,0};
            \uptriforce[black]{0,0.2};
            \uptriforce[black]{0,-0.2};
        \end{tikzpicture}
        - z^2\,
        \begin{tikzpicture}[centerzero]
            \draw (0,-0.6) -- (0,0.6);
            \uptribubr[black]{-0.5,0};
            \uptriforce[black]{0,0};
        \end{tikzpicture}
        + \frac{z^2u^2}{u^2-1}\
        \begin{tikzpicture}[centerzero]
            \draw (0,-0.5) -- (0,0.5);
            \uptriforce[black]{0,0};
        \end{tikzpicture}
        \\
        \overset{\cref{eatzk}}{=}
        \begin{tikzpicture}[centerzero]
            \draw (0,-0.6) -- (0,0.6);
            \uptribubr[black]{0.4,0};
        \end{tikzpicture}
        + \frac{zu^2-z^2u-z}{u^2-1}
        \left(
            \begin{tikzpicture}[centerzero]
                \draw (0,-0.5) -- (0,0.5);
                \downtriforce[black]{0,0};
            \end{tikzpicture}
            - \
            \begin{tikzpicture}[centerzero]
                \draw (0,-0.5) -- (0,0.5);
                \downtriforce[black]{0,0.2};
                \downtriforce[black]{0,-0.2};
            \end{tikzpicture}
        \right)
        \bubblegenr[black]{u}
        + \frac{zu^2 - z^2u - z}{u^2-1}\ \bubblegenr[black]{u}
        \left(
            \begin{tikzpicture}[centerzero]
                \draw (0,-0.6) -- (0,0.6);
                \uptriforce[black]{0,0.2};
                \uptriforce[black]{0,-0.2};
            \end{tikzpicture}
            - \,
            \begin{tikzpicture}[centerzero]
                \draw (0,-0.6) -- (0,0.6);
                \uptriforce[black]{0,0};
            \end{tikzpicture}
        \right)
        .
    \end{multline*}
    Thus
    \[
        \begin{tikzpicture}[centerzero]
            \draw (0,-0.5) -- (0,0.5);
            \uptribubr[black]{-0.5,0};
        \end{tikzpicture}
        +
        \frac{zu^2 - z^2u - z}{u^2-1}\ \bubblegenr[black]{u}
        \left(
            \begin{tikzpicture}[centerzero]
                \draw (0,-0.5) -- (0,0.5);
                \uptriforce[black]{0,0};
            \end{tikzpicture}
            \ -\
            \begin{tikzpicture}[centerzero]
                \draw (0,-0.5) -- (0,0.5);
                \uptriforce[black]{0,0.2};
                \uptriforce[black]{0,-0.2};
            \end{tikzpicture}
        \right)
        =
        \begin{tikzpicture}[centerzero]
            \draw (0,-0.5) -- (0,0.5);
            \uptribubr[black]{0.4,0};
        \end{tikzpicture}
        + \frac{zu^2 - z^2u - z}{u^2-1}
        \left(
            \begin{tikzpicture}[centerzero]
                \draw (0,-0.5) -- (0,0.5);
                \downtriforce[black]{0,0};
            \end{tikzpicture}
            \ -\
            \begin{tikzpicture}[centerzero]
                \draw (0,-0.5) -- (0,0.5);
                \downtriforce[black]{0,0.2};
                \downtriforce[black]{0,-0.2};
            \end{tikzpicture}
        \right)
        \ \bubblegenr[black]{u}.
    \]
    Multiplying both sides by $\frac{z(u^2-1)}{u^2-zu-1}$ and adding $\frac{(t^{-1}-z)u^2 - t^{-1}}{u^2-zu-1}\, \idstrand$ to both sides gives
    \[
        \bubblegenr[black]{u}\
        \left(\
            \begin{tikzpicture}[centerzero]
                \draw (0,-0.5) -- (0,0.5);
            \end{tikzpicture}
            + z^2
            \left(
                \begin{tikzpicture}[centerzero]
                    \draw (0,-0.5) -- (0,0.5);
                    \uptriforce[black]{0,0};
                \end{tikzpicture}
                \ -\
                \begin{tikzpicture}[centerzero]
                    \draw (0,-0.5) -- (0,0.5);
                    \uptriforce[black]{0,0.2};
                    \uptriforce[black]{0,-0.2};
                \end{tikzpicture}
            \right)
        \right)
        =
        \left(\
            \begin{tikzpicture}[centerzero]
                \draw (0,-0.5) -- (0,0.5);
            \end{tikzpicture}
            + z^2
            \left(
                \begin{tikzpicture}[centerzero]
                    \draw (0,-0.5) -- (0,0.5);
                    \downtriforce[black]{0,0};
                \end{tikzpicture}
                \ -\
                \begin{tikzpicture}[centerzero]
                    \draw (0,-0.5) -- (0,0.5);
                    \downtriforce[black]{0,0.2};
                    \downtriforce[black]{0,-0.2};
                \end{tikzpicture}
            \right)
        \right)
        \bubblegenr[black]{u}\, .
    \]
    The first equation in \cref{bubslidek} then follows from the fact that
    \[
        1 + z^2 \left( \frac{u}{u-x} - \frac{u^2}{(u-x)^2} \right)
        = 1 - \frac{z^2ux}{(u-x)^2}.
    \]
    To obtain the second equality in \cref{bubslidek}, we multiply both sides of the first equality in \cref{bubslidek} on the left and right by $\bubblegenl[black]{u}$, then use \cref{infgrassk}.  Alternatively, we can apply the bar involution $\beta$.
\end{proof}

Note that, when $z=q-q^{-1}$ for some $q \in \kk^\times$, the rational function appearing in \cref{bubslidek} can be factored:
\[
    \frac{1-z^2ux^{-1}(u-x^{-1})^{-2}}{1-z^2ux(u-x)^{-2}}
    = \frac{(u-x)^2 (u-q^2x^{-1}) (u-q^{-2}x^{-1})}{(u-x^{-1})^2 (u-q^2x) (u-q^{-2}x)}.
\]

%===================================================================================
\section{Module categories over the affine Brauer category\label{sec:BrauerModules}}
%===================================================================================

A \emph{module category} over a strict $\kk$-linear monoidal category $\cA$ is a $\kk$-linear category $\cR$ together with a strict $\kk$-linear monoidal functor $\mathbf{R} \colon \cA \to \cEnd_\kk(\cR)$, where $\cEnd_\kk(\cR)$ denotes the strict $\kk$-linear monoidal category whose objects are $\kk$-linear endofunctors of $\cR$ and morphisms are natural transformations.  We usually suppress the functor $\mathbf{R}$, using the same notation $f \colon X \to Y$ both for a morphism in $\cA$ and for the natural transformation between endofunctors of $\cR$ that is its image under $\mathbf{R}$.  For a morphism $f \colon X \to Y$ in $\cA$, we represent the evaluation $f_V \colon XV \to YV$ of this natural transformation on an object $V \in \cR$ diagrammatically by drawing a line labelled by $V$ on the right-hand side of the usual string diagram for $f$:
\[
    \begin{tikzpicture}[centerzero]
        \draw (0,-0.4) \botlabel{X} -- (0,0.4) \toplabel{Y};
        \coupon{0,0}{f};
        \draw[module] (0.4,-0.4) \botlabel{V} -- (0.4,0.4);
    \end{tikzpicture}
\]

We wish to study module categories over the affine Brauer category $\AB$ in the case where $\kk$ is a field of characteristic different from $2$, which we assume through the end of \cref{sec:analysis}.  We will also want to assume some finiteness conditions on the module category.  As in \cite{BSW-HKM}, we will consider an abelian category $\cR$ that is either \emph{locally finite abelian} (that is, all objects are finite length and all morphism spaces are finite dimensional as $\kk$-vector spaces) or \emph{schurian} (that is, equivalent to the category of locally finite-dimensional modules over a locally finite-dimensional locally unital algebra)---see \cite[\S2.2]{BSW-HKM} for a more detailed discussion of these notions.

%---------------------------------------------------------------
\subsection{Analysis of minimal polynomials\label{sec:analysis}}
%---------------------------------------------------------------

Let $L\in \cR$ be an object which is a \emph{brick}, that is, an object such that $\End_\cR(L)=\kk$.  In this case, the coefficients of the series $\bubblegenr{u}$ must act by scalars in $\kk$.  Define
\[
    \OO_L(u)
    :=
    \begin{tikzpicture}[centerzero]
        \bubgenr{u}{0,0};
        \draw[module] (0.4,-0.4) \botlabel{L} -- (0.4,0.4);
    \end{tikzpicture}
    \in 1 + u^{-1} \kk \llbracket u^{-1} \rrbracket.
\]

On the other hand, we can consider the image $\go L$ of the generating functor under the action.  This carries an action of $\dotstrand$, which we can think of as defining a ring homomorphism $\kk[u]\to \End_\cR(\go L)$.  By our finiteness assumptions on $\cR$, together with the fact that $\go$ is self-adjoint, we know that $\End_\cR(\go L)$ is finite dimensional over $\kk$.  Thus, the algebra homomorphism
\[
    \kk[u] \to \End_\cR(\go L),\qquad
    f(u) \mapsto
    \begin{tikzpicture}[centerzero]
        \draw (-0.3,-0.4) -- (-0.3,0.4);
        \multdot{-0.3,0}{east}{f(x)};
        \draw[module] (0,-0.4) \botlabel{L} -- (0,0.4);
    \end{tikzpicture}
    \ ,
\]
has nontrivial kernel $I_L$ generated by a unique monic polynomial $m_L \in \kk[u]$.  Therefore, $m_L$ is the unique monic polynomial such that, for $f \in \kk[u]$,
\begin{equation} \label{strada}
    \begin{tikzpicture}[centerzero]
        \draw (-0.3,-0.4) -- (-0.3,0.4);
        \multdot{-0.3,0}{east}{f(x)};
        \draw[module] (0,-0.4) \botlabel{L} -- (0,0.4);
    \end{tikzpicture}
    = 0
    \iff m_L \text{ divides } f.
\end{equation}
Let
\[
    d_L = \deg m_L.
\]

\begin{lem} \label{cdmx}
    If $g \in I_L$, then
    \begin{equation} \label{star}
        \hat{g}(u) := \left( -u-\tfrac{1}{2} \right) g(-u) \OO_L(-u) \in I_L.
    \end{equation}
\end{lem}

\begin{proof}
    Suppose $g \in I_L$.  A priori, $\hat{g}(u) \in \kk \Laurent{u^{-1}}$.  However, since $g(\dotstrand)=0$, we have, for $r> 0$,
    \begin{multline} \label{cerveza}
        \left[ \hat{g}(-u) \right]_{u^{-r}}
        =
        \Big[ \left( u - \tfrac{1}{2} \right) g(u) \OO_L(u) \Big]_{u^{-r}}
        =
        \left[
            (u-\tfrac{1}{2})g(u)\
            \begin{tikzpicture}[centerzero={0,-0.1}]
                \bubgenr{u}{-0.4,0};
                \draw[module] (0,-0.3) \botlabel{L} -- (0,0.3);
            \end{tikzpicture}
        \right]_{u^{-r}}
        \\
        \overset{\cref{eatz}}{=}
        \left[
            u g(u) \left( \uptribubbler - \tfrac{1}{2u} + 1 \right)
            \begin{tikzpicture}[centerzero={0,-0.1}]
                \draw[module] (0,-0.3) \botlabel{L} -- (0,0.3);
            \end{tikzpicture}
        \right]_{u^{-r}}
        \overset{\dagger}{=}
        \left[
            u^{r} g(u)\ \uptribubbler
            \begin{tikzpicture}[centerzero={0,-0.1}]
                \draw[module] (0,-0.3) \botlabel{L} -- (0,0.3);
            \end{tikzpicture}
        \right]_{u^{-1}}
        \overset{\cref{trick}}{=}
        \begin{tikzpicture}[centerzero={0,-0.1}]
            \multbubr{-1.4,0}{x^{r}g(x)};
            \draw[module] (0,-0.3) \botlabel{L} -- (0,0.3);
        \end{tikzpicture}
        =0,
    \end{multline}
    where the equality $\dagger$ holds because $u^{r}(1-\frac{1}{2u})$ is a polynomial for $r > 0$.  From this, we can conclude that $\hat{g}(u) \in \kk[u]$.

    Next, note that
    \begin{multline} \label{burrito}
        0
        =
        \begin{tikzpicture}[centerzero]
            \draw (0,-0.5) to[out=up,in=180] (0.3,0.2) to[out=0,in=up] (0.45,0) to[out=down,in=0] (0.3,-0.2) to[out=180,in=down] (0,0.5);
            \multdot{0.45,0}{west}{x g(x)};
            \draw[module] (1.8,-0.5) \botlabel{L} -- (1.8,0.5);
        \end{tikzpicture}
        \overset{\cref{trick}}{=}
        \left[
            u g(u)\
            \begin{tikzpicture}[centerzero]
                \draw (0,-0.5) to[out=up,in=180] (0.3,0.2) to[out=0,in=up] (0.45,0) to[out=down,in=0] (0.3,-0.2) to[out=180,in=down] (0,0.5);
                \uptriforce{0.45,0};
                \draw[module] (0.8,-0.5) \botlabel{L} -- (0.8,0.5);
            \end{tikzpicture}
        \right]_{u^{-1}}
        \overset{\cref{curlsup}}{=}
        \left[
            g(u) \left( u-\frac{1}{2} \right)
            \begin{tikzpicture}[centerzero]
                \draw (0,-0.5) -- (0,0.5);
                \bubgenr{u}{0.4,0};
                \downtriforce{0,0};
                \draw[module] (0.8,-0.5) \botlabel{L} -- (0.8,0.5);
            \end{tikzpicture}
            \ - \frac{g(u)}{2}\
            \begin{tikzpicture}[centerzero]
                \draw (0,-0.5) -- (0,0.5);
                \uptriforce{0,0};
                \draw[module] (0.3,-0.5) \botlabel{L} -- (0.3,0.5);
            \end{tikzpicture}
        \right]_{u^{-1}}
        \\
        \overset{\cref{star}}{=}
        \left[
            \hat{g}(-u)\
            \begin{tikzpicture}[centerzero]
                \draw (0,-0.5) -- (0,0.5);
                \downtriforce{0,0};
                \draw[module] (0.3,-0.5) \botlabel{L} -- (0.3,0.5);
            \end{tikzpicture}
            \ - \frac{g(u)}{2}\
            \begin{tikzpicture}[centerzero]
                \draw (0,-0.5) -- (0,0.5);
                \uptriforce{0,0};
                \draw[module] (0.3,-0.5) \botlabel{L} -- (0.3,0.5);
            \end{tikzpicture}
        \right]_{u^{-1}}
        \overset{\cref{trick}}{=}
        \begin{tikzpicture}[anchorbase]
            \draw (0,-0.5) -- (0,0.5);
            \multdot{0,0}{west}{\hat{g}(x) - \frac{1}{2} g(x)};
        \end{tikzpicture}.
	 \end{multline}
    Therefore, $\hat{g} - \frac{1}{2}g \in I_L$, and so $\hat{g} \in I_L$, as desired.
\end{proof}

For $f \in \kk[u]$, define
\begin{equation} \label{metrobus}
    \OO_f(u)
    := \frac{\big( (-1)^{\deg f} u - \frac{1}{2} \big) f(-u)}{(u-\frac{1}{2})f(u)} \in 1 + u^{-1} \kk \llbracket u^{-1} \rrbracket.
\end{equation}

\begin{theo} \label{hemlock}
	We have $\OO_L = \OO_{m_L}$.
\end{theo}

\begin{proof}
    By \cref{infgrass}, we have $\OO_L(-u) = \OO_L(u)^{-1}$.  Thus, by \cref{star},
    \begin{equation} \label{meteor}
        \frac{\hat{m}_L(-u)}{\left( u-\frac{1}{2} \right) m_L(u)}
        = \OO_L
        = \frac{\left( -u-\frac{1}{2} \right) m_L(-u)}{\hat{m}_L(u)},
    \end{equation}
    and so
    \begin{equation} \label{sun}
        \hat{m}_L(u) \hat{m}_L(-u) = \left( u - \tfrac{1}{2} \right) \left( -u - \tfrac{1}{2} \right) m_L(u) m_L(-u).
    \end{equation}
    By \cref{cdmx}, $\hat{m}_L(u) \in I_L$, and so $m_L(u)$ divides $\hat{m}_L(u)$.  Combined with \cref{sun} and the fact that $\hat{m}_L(-u)$ is monic, this implies that
    \begin{equation}\label{f-factor}
    	 \hat{m}_L(-u) = \left( (-1)^{d_L}u + \frac{\epsilon}{2} \right) m_L(-u)
    \end{equation}
    for some $\epsilon \in \{\pm 1\}$.

    Now note that
    \begin{multline*}
        0
        =
        \begin{tikzpicture}[centerzero]
            \draw (0,-0.5) to[out=up,in=180] (0.3,0.2) to[out=0,in=up] (0.45,0) to[out=down,in=0] (0.3,-0.2) to[out=180,in=down] (0,0.5);
            \multdot{0.45,0}{west}{m_L(x)};
            \draw[module] (1.6,-0.5) \botlabel{L} -- (1.6,0.5);
        \end{tikzpicture}
        \overset{\cref{trick}}{=}
        \left[
            m_L(u)\
            \begin{tikzpicture}[centerzero]
                \draw (0,-0.5) to[out=up,in=180] (0.3,0.2) to[out=0,in=up] (0.45,0) to[out=down,in=0] (0.3,-0.2) to[out=180,in=down] (0,0.5);
                \uptriforce{0.45,0};
                \draw[module] (0.8,-0.5) \botlabel{L} -- (0.8,0.5);
            \end{tikzpicture}
        \right]_{u^{-1}}
        \\
        \overset{\cref{curlsup}}{=}
        \left[
            m_L(u) \left( 1-\frac{1}{2u} \right)
            \begin{tikzpicture}[centerzero]
                \draw (0,-0.5) -- (0,0.5);
                \bubgenr{u}{0.4,0};
                \downtriforce{0,0};
                \draw[module] (0.8,-0.5) \botlabel{L} -- (0.8,0.5);
            \end{tikzpicture}
            \ - \frac{m_L(u)}{2u}\
            \begin{tikzpicture}[centerzero]
                \draw (0,-0.5) -- (0,0.5);
                \uptriforce{0,0};
                \draw[module] (0.3,-0.5) \botlabel{L} -- (0.3,0.5);
            \end{tikzpicture}
        \right]_{u^{-1}}
        \overset{\cref{star}}{=}
        \left[
             \frac{\hat{m}_L(-u)}{u}\
            \begin{tikzpicture}[centerzero]
                \draw (0,-0.5) -- (0,0.5);
                \downtriforce{0,0};
                \draw[module] (0.3,-0.5) \botlabel{L} -- (0.3,0.5);
            \end{tikzpicture}
            \ - \frac{m_L(u)}{2u}\
            \begin{tikzpicture}[centerzero]
                \draw (0,-0.5) -- (0,0.5);
                \uptriforce{0,0};
                \draw[module] (0.3,-0.5) \botlabel{L} -- (0.3,0.5);
            \end{tikzpicture}
        \right]_{u^{-1}}
        \\
        \overset{\cref{f-factor}}{=}
        \left[
            \left( (-1)^{d_L} + \frac{\epsilon}{2u} \right) m_L(-u)\
            \begin{tikzpicture}[centerzero]
                \draw (0,-0.5) -- (0,0.5);
                \downtriforce{0,0};
                \draw[module] (0.3,-0.5) \botlabel{L} -- (0.3,0.5);
            \end{tikzpicture}
            \ - \frac{m_L(u)}{2u}\
            \begin{tikzpicture}[centerzero]
                \draw (0,-0.5) -- (0,0.5);
                \uptriforce{0,0};
                \draw[module] (0.3,-0.5) \botlabel{L} -- (0.3,0.5);
            \end{tikzpicture}
        \right]_{u^{-1}}
        \overset{\cref{trick+}}{=}
        \begin{tikzpicture}[centerzero]
            \draw (0,-0.5) -- (0,0.5);
            \multdot{0,0}{east}{[g]_{\ge 0}(x)};
            \draw[module] (0.3,-0.5) \botlabel{L} -- (0.3,0.5);
        \end{tikzpicture},
    \end{multline*}
    where
    \[
        g(u) = (-1)^{d_L} m_L(u) - \frac{1+\epsilon}{2u} m_L(u).
    \]
    Therefore, the polynomial
    \[
        [g(u)]_{u^{\ge 0}}
        = (-1)^{d_L} m_L(u) - \left[ \frac{1+\epsilon}{2u} m_L(u) \right]_{u^{\ge 0}}
    \]
    must be divisible by $m_L(u)$, which is only possible if $\left[ \frac{1+\epsilon}{2u} m_L(u) \right]_{u^{\ge 0}}=0$.  Since $m_L(u)$ has positive degree, this in turn can only hold if $\epsilon=-1$.  Thus,
    \[
        \hat{m}_L(-u) =
        \begin{cases}
            (u-\frac{1}{2}) m_L(-u) & \text{if } d_L \text{ is even}, \\
            (-u-\frac{1}{2}) m_L(-u) & \text{if } d_L \text{ is odd}.
        \end{cases}
    \]
    The result follows.
\end{proof}

%------------------------------------------------------------------------
\subsection{Admissibility in the Brauer case\label{sec:admissibleBrauer}}
%------------------------------------------------------------------------

Restrictions on the scalars by which bubbles can act have played an important role in the study of degenerate cyclotomic BMW algebras.  In this section, we show how such restrictions can be easily deduced from the results of this paper.
 We assume throughout this section that $\kk$ is a commutative ring in which $2$ is invertible.

Fix an $\AB$-module category $\mathbf{R} \colon \AB \to \cEnd_\kk(\cR)$.  Then $\mathbf{R}$ maps elements of $\End_\AB(\one)$ to natural transformations of the identity functor on $\cR$.  Evaluation on an element $V$ of $\cR$ then defines a ring homomorphism from $\End_\AB(\one)$ to the center $Z_V$ of the endomorphism algebra of $V$.  If $V$ is a brick, then $Z_V = \kk$.  However, it is sometimes useful to consider the more general situation.  In the literature on degenerate cyclotomic BMW algebras, the scalar by which $\multbubbler{n}$ acts is usually denoted $\omega_n$.  For $V$ an object in an $\AB$-module category $\cR$, let
\begin{equation} \label{patio}
    \Omega_V = (\omega^V_r)_{r \in \N},\qquad
    \omega^V_r
    :=
    \begin{tikzpicture}[centerzero]
        \multbubr{0,0}{r};
        \draw[module] (0.7,-0.5) \botlabel{V} -- (0.7,0.5);
    \end{tikzpicture}
    \in Z_V,\qquad r \in \N.
\end{equation}

\Cref{scratch} places a restriction on the possible sequences $\Omega_V$.  A sequence $\Omega = (\omega_r)_{r \in \N}$ in a commutative ring satisfying
\begin{equation} \label{cornerstone}
    \omega_{2r+1} = \frac{1}{2} \left( -\omega_{2r} + \sum_{n=0}^{2r} (-1)^n \omega_n \omega_{2r-n} \right)
\end{equation}
is called \emph{admissible} in \cite[Def.~2.10]{AMR06}.  It follows from \cref{scratch} that admissibility is simply a consequence of the relations of the category $\AB$.  We think that the analogue \cref{infgrass} of the infinite grassmannian relation is an elegant statement of this relation.

Suppose $m(u) \in \kk[u]$ is monic and factors completely as
\[
    m(u) = (u-a_1)(u-a_2) \dotsm (u-a_{d}).
\]
In \cite[Def.~2.5]{Goo11}, a sequence $\Omega = (\omega_r)_{r \in \N}$ in $\kk$ is called \emph{weakly admissible} for $m$ if it satisfies \cref{cornerstone} and
\begin{equation} \label{weak}
    \sum_{j=0}^{d} (-1)^j e_j(a_1,\dotsc,a_{d}) \omega_{j+n} = 0
    \quad \text{for all } n \in \N,
\end{equation}
where $e_j$ is the $j$-th elementary symmetric polynomial.  For $\Omega = \Omega_V$ as in \cref{patio}, the condition \cref{weak} is equivalent to the condition
\[
    \begin{tikzpicture}[centerzero]
        \multbubr{0,0}{x^n m(x)};
        \draw[module] (1.5,-0.4) \botlabel{V} -- (1.5,0.4);
    \end{tikzpicture}
    = 0
    \qquad \text{for all } n \in \N.
\]
In particular,
\begin{equation} \label{roses}
    \text{the sequence $\Omega_V$ is weakly admissible for $m$ if }
    \begin{tikzpicture}[centerzero]
        \draw (0,-0.4) -- (0,0.4);
        \multdot{0,0}{east}{m(x)};
        \draw[module] (0.3,-0.4) \botlabel{V} -- (0.3,0.4);
    \end{tikzpicture}
    = 0,
\end{equation}
Note the slightly confusing fact that ``weakly admissible'' is a \emph{stronger} condition than ``admissible''.

\begin{lem} \label{brew}
    The sequence $\Omega_V$ is weakly admissible for $m$ if and only if $\uptribubbler$ acts on $V$ by a rational function whose denominator divides $m$.  Equivalently, $\Omega_V$ is weakly admissible for $m$ if and only if
    \[
        m(u)\, \uptribubbler
        \begin{tikzpicture}[centerzero]
            \draw[module] (0,-0.4) \botlabel{V} -- (0,0.4);
        \end{tikzpicture}
        \in Z_V[u].
    \]
\end{lem}

\begin{proof}
    For $n \in \N$, we have
    \[
        \left[
            m(u)\
            \begin{tikzpicture}[centerzero={0,-0.1}]
                \uptribubr{-0.4,0};
                \draw[module] (0,-0.3) \botlabel{V} -- (0,0.3);
            \end{tikzpicture}
        \right]_{u^{-n-1}}
        =
        \left[
            u^n m(u)\
            \begin{tikzpicture}[centerzero={0,-0.1}]
                \uptribubr{-0.4,0};
                \draw[module] (0,-0.3) \botlabel{V} -- (0,0.3);
            \end{tikzpicture}
        \right]_{u^{-1}}
        \overset{\cref{trick}}{=}
        \begin{tikzpicture}[centerzero]
            \multbubr{0,0}{x^n m(x)};
            \draw[module] (1.5,-0.5) \botlabel{V} -- (1.5,0.5);
        \end{tikzpicture}
        \ .
    \]
    The result follows.
\end{proof}

We now turn our attention to another notion of admissibility that has played an important role in the literature.  For a finite sequence $\ba = (a_1,a_2,\dotsc,a_{d})$ in $\kk$, a sequence $\Omega = (\omega_r)_{r \in \N}$ in $\kk$ is called \emph{$\ba$-admissible} if
\begin{equation}\label{def:Omega}
    \Omega(u)
    := \sum_{r=0}^\infty \omega_r u^{-r}
    = -u + \frac{1}{2} + \left( u - \frac{(-1)^{d}}{2} \right) \prod_{i=1}^{d} \frac{u+a_i}{u-a_i}.
\end{equation}
If we define
\begin{equation} \label{OOmega}
    \OO_\Omega(u)
    := \frac{2}{2u-1} \Omega(u) + 1
    \in 1 + u^{-1} \kk \llbracket u^{-1} \rrbracket,
\end{equation}
then
\[
    \Omega \text{ is admissible}
    \iff
    \OO_\Omega(u) \OO_\Omega(-u) = 1.
\]
Furthermore, recalling \cref{metrobus},
\begin{equation} \label{bench}
    \Omega \text{ is $\ba$-admissible}
    \iff
    \OO_\Omega = \OO_m,
    \quad \text{for} \quad
    m(u) = \prod_{i=1}^{d}(u-a_i).
\end{equation}
The term \emph{$\ba$-admissible} was introduced in \cite[Def.~3.6]{AMR06}, although here we use  the equivalent formulation from \cite[Lem.~3.8]{AMR06}.

We can see the significance of these conditions by computing  the action of the bubbles $\multbubbler{n}$ on $L$.

\begin{cor} \label{cactus}
    If $\kk$ is a field, then we have
    \[
        u\ \uptribubbler
        \begin{tikzpicture}[centerzero]
            \draw[module] (0,-0.3) \botlabel{L} -- (0,0.3);
        \end{tikzpicture}
        = \sum_{n=0}^\infty u^{-n}\ \multbubbler{n}
        \begin{tikzpicture}[centerzero]
            \draw[module] (0,-0.3) \botlabel{L} -- (0,0.3);
        \end{tikzpicture}
        =
        -u + \frac{1}{2} + \left( (-1)^{d_L}u-\frac{1}{2} \right) \frac{m_L(-u)}{m_L(u)}.
    \]
    Thus, if $m_L(u) =\prod_{i=1}^{d}(u-a_i)$, then $\OO_L(u) = \OO_{\Omega}$ for $\Omega$ the $\ba$-admissible sequence defined by \cref{def:Omega}.
\end{cor}

\begin{proof}
    Since
    \[
        u\ \uptribubbler \overset{\cref{eatz}}{=} - u -\frac{1}{2} + \left( u - \frac{1}{2} \right) \bubblegenr{u}
    \]
    and $\bubblegenr{u}$ acts on $L$ by $\OO_L(u)$, the result follows from \cref{hemlock}.
\end{proof}

Thus, if $\kk$ is a field and $m_L(u) = (u-a_1)(u-a_2) \dotsm (u-a_d)$ as a product of linear factors in the algebraic closure of $\kk$, then it follows from \cref{cactus} that
\[
    u\ \uptribubbler
    \begin{tikzpicture}[centerzero]
        \draw[module] (0,-0.3) \botlabel{L} -- (0,0.3);
    \end{tikzpicture}
    = \sum_{n=0}^\infty u^{-n}\ \multbubbler{n}
    \begin{tikzpicture}[centerzero]
        \draw[module] (0,-0.3) \botlabel{L} -- (0,0.3);
    \end{tikzpicture}
    =
    - u + \frac{1}{2} + \left( u-\frac{(-1)^d}{2} \right) \prod_{i=1}^d \frac{u+a_i}{u-a_i}.
\]

%=======================================================================================
\section{Module categories over the affine Kauffman category\label{sec:KauffmanModules}}
%=======================================================================================

In this section, we study module categories over the affine Kauffman category $\AK(z,t)$.  Our treatment is parallel to that of \cref{sec:BrauerModules}.  As in that section, we let $\cR$ be a module category over $\AK(z,t)$ that is either locally finite abelian or schurian.  We assume throughout this section that $\kk$ is a field.

%----------------------------------------------
\subsection{Analysis of the minimal polynomial}
%----------------------------------------------

Let $L \in \cR$ be a brick, and let $J_L$ be the kernel of the algebra homomorphism
\begin{equation} \label{pear}
    \kk[u,u^{-1}] \to \End_\cR(\go L),\qquad
    f(u) \mapsto
    \begin{tikzpicture}[centerzero]
        \draw (-0.3,-0.4) -- (-0.3,0.4);
        \multdot[black]{-0.3,0}{east}{f(x)};
        \draw[module] (0,-0.4) \botlabel{L} -- (0,0.4);
    \end{tikzpicture}
    \, .
\end{equation}
Since now the dot is invertible, for all $k \in \Z$ and $g \in \kk[u,u^{-1}]$, the element $u^k g(u)$ generates the same ideal as $g$.  Let $m_L$ be the unique generator of $J_L$ that lies in $\kk[u]$, is monic, and satisfies $M := m_L(0) \neq 0$.  As in the previous section, we let $d_L = \deg m_L$.  Note that we can equally well find the corresponding minimal polynomial for the inverse dot, which has the same degree.  For a monic polynomial $f(u) \in \kk[u]$ with $f(0) \ne 0$, define
\begin{equation} \label{beijing}
    \check{f}(u) = f(0)^{-1} u^{\deg f} f(u^{-1}),
\end{equation}
which is another monic polynomial with nonzero constant term $\check{f}(0) = f(0)^{-1}$.  Then $\check{m}_L(u)$ is the minimal polynomial of the inverse dot.

The coefficients of the series $\bubblegenr[black]{u}$ must act by scalars in $\kk$.  Define
\begin{equation} \label{alipay}
    \rOO_L(u)
    :=
    \begin{tikzpicture}[centerzero]
        \bubgenr[black]{u}{0,0};
        \draw[module] (0.4,-0.4) \botlabel{L} -- (0.4,0.4);
    \end{tikzpicture}
    \in t + u^{-1} \kk \llbracket u^{-1} \rrbracket,
    \qquad
    \lOO_L(u)
    :=
    \begin{tikzpicture}[centerzero]
        \bubgenl[black]{u}{0,0};
        \draw[module] (0.4,-0.4) \botlabel{L} -- (0.4,0.4);
    \end{tikzpicture}
    \in t^{-1} + u^{-1} \kk \llbracket u^{-1} \rrbracket.
\end{equation}
Note that, unlike in the degenerate case, there is no analogue of the first relation in \cref{infgrass}, which is why we have two series $\rOO_L$ and $\lOO_L$, compared with the single series $\OO_L$ in the degenerate case.   However, we will show below that, in some module categories, there is a parallel relation between these series; see \cref{sneeze}.

We have the following analogue of \cref{cdmx}.

\begin{lem} \label{blackcdmx}
    If $g \in J_L\cap \kk[u]$ has nonzero constant term then
    \begin{equation} \label{blackstar}
        \hat{g}(u) := \left( 1 - zu - u^2 \right) g(u^{-1})u^{\deg g} \rOO_L(u^{-1})
    \end{equation}
    is a polynomial of degree $\deg g+2$ with nonzero constant term, and $\hat{g} \in J_L$.
\end{lem}

\begin{proof}
    Suppose $g \in J_L \cap \kk[u]$ has nonzero constant term.  A priori, $\hat{g}(u) \in \kk \Taylor{u}$ with nonzero constant term.  However, since $g(\dotstrand)=0$, we have, for $s< 0$,
    \begin{multline*}
        \left[ \hat{g}(u) \right]_{u^{\deg g+2-s}}
        \overset{\cref{blackstar}}{=} \left[ \left( 1-zu-u^2 \right) g(u^{-1})u^{\deg g} \rOO_L(u^{-1}) \right]_{u^{\deg g+2-s}}
        \\
        = \left[ \left( u^2-zu-1 \right) g(u) \rOO_L(u) \right]_{u^s}
        \overset{\cref{alipay}}{=}
        \left[
            \left( u^2-zu-1 \right) g(u)\
            \begin{tikzpicture}[centerzero={0,-0.1}]
                \bubgenr[black]{u}{-0.4,0};
                \draw[module] (0,-0.3) \botlabel{L} -- (0,0.3);
            \end{tikzpicture}
        \right]_{u^s}
        \\
        \overset{\cref{eatzk}}{=}
        \left[
            g(u) \Big( \left( t^{-1}-z \right) u^2 - t^{-1} + z \left( u^2-1 \right)\ \uptribubbler[black] \Big)
            \begin{tikzpicture}[centerzero={0,-0.1}]
                \draw[module] (0,-0.3) \botlabel{L} -- (0,0.3);
            \end{tikzpicture}
        \right]_{u^s}
        \\
        \overset{\dagger}{=}
        \left[
            z(u^{2-s}-u^{-s}) g(u) \
            \begin{tikzpicture}[centerzero={0,-0.1}]
                \uptribubr[black]{-0.5,0};
                \draw[module] (0,-0.3) \botlabel{L} -- (0,0.3);
            \end{tikzpicture}
        \right]_{u^{0}}
        \overset{\cref{blacktrick}}{=}
        \begin{tikzpicture}[centerzero={0,-0.1}]
            \multbubr{-2.8,0}{z(x^{2-s}-x^{-s})g(x)};
            \draw[module] (0,-0.3) \botlabel{L} -- (0,0.3);
        \end{tikzpicture}=0,
    \end{multline*}
    where, in the second equality, we replaced $u$ by $u^{-1}$, and the equality $\dagger$ holds because $s<0$ and $g(u) \left( \left( t^{-1}-z \right) u^2 - t^{-1} \right)$ is a polynomial.  From this, we can conclude that $\hat{g} \in \kk[u]$.  On the other hand, if $s=0$, a similar argument shows that $\left[ \hat{g}(u) \right]_{u^{\deg g+2}}=-t^{-1}g(0)$, and so the degree of $\hat{g}$ is precisely $\deg g+2$.

    Next, note that
    \begin{multline*}
        0
        =
        \begin{tikzpicture}[centerzero]
            \draw (0.3,0.2) to[out=0,in=up] (0.45,0) to[out=down,in=0] (0.3,-0.2) to[out=180,in=down] (0,0.5);
            \draw[wipe] (0,-0.5) to[out=up,in=180] (0.3,0.2);
            \draw (0,-0.5) to[out=up,in=180] (0.3,0.2);
            \multdot[black]{0.45,0}{west}{(x^2-1) g(x)};
            \draw[module] (2.2,-0.5) \botlabel{L} -- (2.2,0.5);
        \end{tikzpicture}
        \overset{\cref{blacktrick}}{=}
        \left[
            (u^2-1) g(u)\
            \begin{tikzpicture}[centerzero]
                \draw (0.3,0.2) to[out=0,in=up] (0.45,0) to[out=down,in=0] (0.3,-0.2) to[out=180,in=down] (0,0.5);
                \draw[wipe] (0,-0.5) to[out=up,in=180] (0.3,0.2);
                \draw (0,-0.5) to[out=up,in=180] (0.3,0.2);
                \uptriforce[black]{0.45,0};
                \draw[module] (0.7,-0.5) \botlabel{L} -- (0.7,0.5);
            \end{tikzpicture}
        \right]_{u^{0}}
        \\
        \overset{\cref{blackcurl}}{=}
        \left[
           (u^2-zu-1) g(u)\
            \begin{tikzpicture}[centerzero]
                \draw (0,-0.5) -- (0,0.5);
                \downtriforce[black]{0,0};
                \bubgenr[black]{u}{0.4,0};
                \draw[module] (0.85,-0.5) \botlabel{L} -- (0.85,0.5);
            \end{tikzpicture}
            + z u^2 g(u)
            \left(\,
                \begin{tikzpicture}[centerzero]
                    \draw (0,-0.5) -- (0,0.5);
                \end{tikzpicture}
                -
                \begin{tikzpicture}[centerzero]
                    \draw (0,-0.5) -- (0,0.5);
                    \uptriforce[black]{0,0};
                \end{tikzpicture}
            \right)
            \begin{tikzpicture}[centerzero]
                \draw[module] (0,-0.5) \botlabel{L} -- (0,0.5);
            \end{tikzpicture}
        \right]_{u^{0}}
        \\
        \overset{\cref{blackstar}}{=}
        \left[
            \hat{g}(u^{-1})u^{\deg g+2}\
            \begin{tikzpicture}[centerzero]
                \draw (0,-0.5) -- (0,0.5);
                \downtriforce[black]{0,0};
                \draw[module] (0.3,-0.5) \botlabel{L} -- (0.3,0.5);
            \end{tikzpicture}
            \ - z u^2 g(u)\
            \begin{tikzpicture}[centerzero]
                \draw (0,-0.5) -- (0,0.5);
                \uptriforce[black]{0,0};
                \draw[module] (0.3,-0.5) \botlabel{L} -- (0.3,0.5);
            \end{tikzpicture}
        \right]_{u^0}
        \overset{\cref{blacktrick}}{=}
        \begin{tikzpicture}[centerzero]
            \draw (0,-0.5) -- (0,0.5);
            \multdot[black]{0,0}{west}{\hat{g}(x)x^{-\deg g-2} - zx^2 g(x)};
            \draw[module] (3.6,-0.5) \botlabel{L} -- (3.6,0.5);
        \end{tikzpicture}.
    \end{multline*}
    Therefore, $\hat{g}(u)u^{-\deg g-2} - zu^2 g(u) \in J_L$, and so $\hat{g}(u) \in J_L$, as desired.
\end{proof}

Note that, rearranging \cref{blackstar}, we obtain the formula
\[
    \rOO_L
    = \frac{\hat{g}(u^{-1})u^{\deg g+2}}{\left( u^2 - zu -1 \right) g(u)},
\]
showing that $\rOO_L$ is a rational function.  Furthermore, applying \cref{blackcdmx} with $g=m_L$, we see that $m_L$ must divide $\hat{m}_L$, and so we can factor $\hat{m}_L(u)=h(u)m_L(u)$, where $h(u)$ is a quadratic polynomial.  Thus, we find that
\[
    \rOO_L
    = \frac{h(u^{-1})m_L(u^{-1})u^{d_L+2}}{\left( u^2 - zu -1 \right) m_L(u)}.
\]
It follows from \cref{alipay,blackstar} that $h(0) = tM^{-1}$.
\details{
    We have $\hat{m}_L(0) = Mh(0)$.  By \cref{alipay,blackstar}, $\hat{m}_L(0) = t$, and thus $h(0) = tM^{-1}$.
}
Thus, we can write $u^2h(u^{-1}) = t M^{-1}(u^2-1)+au+b$ for some $a,b \in \kk$.  Using \cref{beijing}, we then have
\begin{equation}\label{401taxi}
    \rOO_L
    = \frac{u^2-1}{u^2-zu-1} \left( t + \frac{(au+b)M}{u^2-1} \right) \frac{\check{m}_L(u)}{m_L(u)}.
\end{equation}

\begin{lem}
    Suppose that $f \in \kk[u]$ and that $r \in \N$ is odd and satisfies $r \ge \deg f$.  Then
    \begin{equation} \label{suite}
        \left[
            \frac{f(u)u}{u^2-1}
            \left(\,
                \begin{tikzpicture}[centerzero]
                    \draw (0,-0.5) -- (0,0.5);
                    \uptriforce[black]{0,0};
                \end{tikzpicture}
                -
                \begin{tikzpicture}[centerzero]
                    \draw (0,-0.5) -- (0,0.5);
                \end{tikzpicture}\,
            \right)
            +
            \frac{f(u^{-1}) u^r}{u^2-1}\,
            \begin{tikzpicture}[centerzero]
                \draw (0,-0.5) -- (0,0.5);
                \downtriforce[black]{0,0};
            \end{tikzpicture}\,
        \right]_{u^0}
        =
        \begin{tikzpicture}[centerzero]
            \draw (0,-0.5) -- (0,0.5);
            \multdot[black]{0,0}{west}{f(x)(x^{-1}+x^{-3} + \dotsb + x^{2-r})};
        \end{tikzpicture}
        .
    \end{equation}
\end{lem}

\begin{proof}
    By linearity, it suffices to consider the case $f = u^n$, $n \in N$.  In this case, letting $x$ denote the dot, the left-hand side of \cref{suite} is
    \begin{align*}
        &\left[
            \frac{u^{n-1}}{1-u^{-2}} \left( \frac{1}{1-xu^{-1}} - 1 \right)
            + \frac{u^{r-n-2}}{1-u^{-2}} \frac{1}{1-x^{-1}u^{-1}}
        \right]_{u^0}
        \\
        &\qquad \qquad =
        \left[
            u^{n-1}(1+u^{-2} + u^{-4} + \dotsb)(xu^{-1} + x^2 u^{-2} + \dotsb)
        \right.
        \\
        & \qquad \qquad \qquad \qquad
        \left.
            + u^{r-n-2} (1 + u^{-2} + u^{-4} + \dotsb) (1 + x^{-1}u^{-1} + x^{-2}u^{-2} + \dotsb)
        \right]_{u^0}
        \\
        &\qquad \qquad = x^{n-1} + x^{n-3} + \dotsb + x^{n+2-r},
    \end{align*}
    as desired.
\end{proof}

\begin{cor} \label{sofa}
    Suppose $f \in \kk[u] \cap J_L$, $r \in \N$, $r \ge \deg f$, and $r$ is odd.  Then
    \[
        \left[
            \frac{f(u^{-1}) u^r}{u^2-1}\,
            \begin{tikzpicture}[centerzero]
                \draw (0,-0.5) -- (0,0.5);
                \downtriforce[black]{0,0};
                \draw[module] (0.3,-0.5) -- (0.3,0.5);
            \end{tikzpicture}\,
        \right]_{u^0}
        =
        \left[
            \frac{f(u)u}{u^2-1}
            \left(\,
                \begin{tikzpicture}[centerzero]
                    \draw (0,-0.5) -- (0,0.5);
                \end{tikzpicture}
                -
                \begin{tikzpicture}[centerzero]
                    \draw (0,-0.5) -- (0,0.5);
                    \uptriforce[black]{0,0};
                \end{tikzpicture}
                \,
            \right)
            \begin{tikzpicture}[centerzero]
                \draw[module] (0,-0.5) \botlabel{L} -- (0,0.5);
            \end{tikzpicture}
        \right]_{u^0}
        .
    \]
\end{cor}

For monic $f \in \kk[u]$ with $f(0) \ne 0$, define
\begin{align} \label{metrobusKr}
    \rOO_f(u)
    &:=
    \begin{dcases}
        \left( \frac{tu^2 - zf(0) - t}{u^2-zu-1} \right) \frac{\check{f}(u)}{f(u)}
        & \text{if } \deg f \text{ is even},
        \\
        \left( \frac{tu^2 - zf(0)u -t}{u^2-zu-1} \right) \frac{\check{f}(u)}{f(u)}
        & \text{if } \deg f \text{ is odd}
    \end{dcases}
    \in t + u^{-1} \kk \llbracket u^{-1} \rrbracket,
    \\ \label{metrobusKl}
    \lOO_f(u)
    &:=
    \begin{dcases}
        \left( \frac{t^{-1}u^2 + zf(0)^{-1} - t^{-1}}{u^2+zu-1} \right) \frac{f(u)}{\check{f}(u)}
        & \text{if } \deg f \text{ is even},
        \\
        \left( \frac{t^{-1}u^2 + zf(0)^{-1}u - t^{-1}}{u^2+zu-1} \right) \frac{f(u)}{\check{f}(u)}
        & \text{if } \deg f \text{ is odd}
    \end{dcases}
    \in t^{-1} + u^{-1} \kk \llbracket u^{-1} \rrbracket.
\end{align}

\begin{prop} \label{sneeze}
    For monic $f \in \kk[u]$ with $f(0) \ne 0$, the following statements are equivalent:
    \begin{enumerate}
        \item \label{sneeze1} we have an equality of rational functions $\rOO_f(u^{-1}) = \lOO_f(u)$;
        \item \label{sneeze2} we have $z = t^{-1} f(0) - t f(0)^{-1}$ if $\deg f$ is even, and $f(0) = \pm t$ is $\deg f$ is odd;
        \item \label{sneeze3} we have $\rOO_f(u) \lOO_f(u) = 1$.
    \end{enumerate}
\end{prop}

\begin{proof}
    Let $d = \deg f$.  We first show that statements \cref{sneeze1} and \cref{sneeze2} are equivalent.  First, suppose that $d$ is even.  Then, using \cref{beijing}, we have
    \[
        \rOO_f(u) = \frac{tf(0)^{-1}u^2 - z - tf(0)^{-1}}{u^2-zu-1} \frac{f(u^{-1})}{f(u)} u^d,
        \quad
        \lOO_f(u) = \frac{t^{-1}f(0)u^2 + z - t^{-1}f(0)}{u^2+zu-1} \frac{f(u)}{f(u^{-1})} u^{-d},
    \]
    and the result follows easily.
    \details{
        We have
        \[
            \rOO_f(u^{-1}) = \frac{tf(0)^{-1}u^{-2} - z - tf(0)^{-1}}{u^{-2}-zu^{-1}-1} \frac{f(u)}{f(u^{-1})} u^{-d}
            = \frac{(z+tf(0)^{-1})u^2 - tf(0)^{-1}}{u^2+zu-1}.
        \]
    }
    Similarly, when $d$ is odd, we have
    \[
        \rOO_f(u) = \frac{tf(0)^{-1}u^2 - zu - tf(0)^{-1}}{u^2-zu-1} \frac{f(u^{-1})}{f(u)} u^d,
        \quad
        \lOO_f(u) = \frac{t^{-1}f(0)u^2 + z - t^{-1}f(0)}{u^2+zu-1} \frac{f(u)}{f(u^{-1})} u^{-d},
    \]
    and so
    \[
        \rOO_f(u^{-1}) = \lOO_f(u)
        \iff t f(0)^{-1} = t^{-1} f(0)
        \iff f(0) = \pm t,
    \]
    as desired.

    Now we show that statements \cref{sneeze2} and \cref{sneeze3} are equivalent.  We have
    \[
        \rOO_f(u)\lOO_f(u)
        =
        \begin{dcases}
            \frac{(u^2-1)^2 + z(u^2-1)(tf(0)^{-1}-t^{-1}f(0)) - z^2}{(u^2-1)^2-z^2u^2}
            & \text{if $d$ is even}, \\
            \frac{(u^2-1)^2 + z(u^2-1)(tf(0)^{-1}-t^{-1}f(0))u - z^2u^2}{(u^2-1)^2-z^2u^2}
            & \text{if $d$ is odd},
        \end{dcases}
   \]
    and the result follows from a straightforward computation.
\end{proof}

\begin{rem} \label{netherlands}
    Note that, for $q \in \kk^\times$, we have $z = q - q^{-1}$ if and only if $q$ and $-q^{-1}$ are the two roots of the quadratic polynomial $u^2-zu-1$.  If $f \in \kk[u]$ is even degree and satisfies the conditions of \cref{sneeze}, then we must have $z = q - q^{-1}$ for some $q \in \kk^\times$.  In this case, by the uniqueness of roots of the quadratic polynomial, we must have $f(0) = qt$ or $f(0) = -q^{-1}t$.
\end{rem}

Recall that $M = m_L(0)$, and that $d_L = \deg m_L$.

\begin{theo} \label{backhome}
    We have
    \[
        \rOO_L(u) = \rOO_{m_L}(u)
        \qquad \text{and} \qquad
        \lOO_L(u) = \lOO_{m_L}(u) = \rOO_{m_L}(u^{-1}).
    \]
    If $d_L$ is even, then $z = t^{-1} M - t M^{-1}$.  If $d_L$ is odd, then $M = \pm t$.
\end{theo}

\begin{proof}
    The statements $\rOO_L = \rOO_{m_L}$ and $\lOO_L = \lOO_{m_L}$ are equivalent by \cref{swap}.
    %(The fact that $m_L(0)$ is inverted when passing from \cref{metrobusKr} to \cref{metrobusKl}, for $f=m_L$, follows from the fact that the isomorphism of \cref{swap} inverts the dot.)
    The other assertions of the theorem follow from \cref{sneeze}.  Therefore, we need only prove that $\rOO_L = \rOO_{m_L}$.

    Note that
    \begin{multline*}
        0
        =
        \begin{tikzpicture}[centerzero]
            \draw (0.3,0.2) to[out=0,in=up] (0.45,0) to[out=down,in=0] (0.3,-0.2) to[out=180,in=down] (0,0.5);
            \draw[wipe] (0,-0.5) to[out=up,in=180] (0.3,0.2);
            \draw (0,-0.5) to[out=up,in=180] (0.3,0.2);
            \multdot[black]{0.45,0}{west}{m_L(x)};
            \draw[module] (1.7,-0.5) \botlabel{L} -- (1.7,0.5);
        \end{tikzpicture}
        \overset{\cref{blacktrick}}{=}
        \left[
            m_L(u)\
            \begin{tikzpicture}[centerzero]
                \draw (0.3,0.2) to[out=0,in=up] (0.45,0) to[out=down,in=0] (0.3,-0.2) to[out=180,in=down] (0,0.5);
                \draw[wipe] (0,-0.5) to[out=up,in=180] (0.3,0.2);
                \draw (0,-0.5) to[out=up,in=180] (0.3,0.2);
                \uptriforce[black]{0.45,0};
                \draw[module] (0.8,-0.5) \botlabel{L} -- (0.8,0.5);
            \end{tikzpicture}
        \right]_{u^{0}}
        \\
        \overset{\cref{blackcurl}}{=}
        \left[
            \frac{(u^2-zu-1)m_L(u)}{u^2-1}\
            \begin{tikzpicture}[centerzero]
                \draw (0,-0.5) -- (0,0.5);
                \downtriforce[black]{0,0};
                \bubgenr[black]{u}{0.4,0};
                \draw[module] (0.8,-0.5) \botlabel{L} -- (0.8,0.5);
            \end{tikzpicture}
            + \frac{zu^2m_L(u)}{u^2-1}
            \left(\,
                \begin{tikzpicture}[centerzero]
                    \draw (0,-0.5) -- (0,0.5);
                \end{tikzpicture}
                -
                \begin{tikzpicture}[centerzero]
                    \draw (0,-0.5) -- (0,0.5);
                    \uptriforce[black]{0,0};
                \end{tikzpicture}
            \right)
            \begin{tikzpicture}[centerzero]
                \draw[module] (0,-0.5) \botlabel{L} -- (0,0.5);
            \end{tikzpicture}
        \right]_{u^{0}}
        \\
        \overset{\cref{401taxi}}{\underset{\cref{beijing}}{=}}
        \left[
            \left( t \check{m}_L(u) + \frac{(au+b)m_L(u^{-1})u^{d_L}}{u^2-1} \right)\
            \begin{tikzpicture}[centerzero]
                \draw (0,-0.5) -- (0,0.5);
                \downtriforce[black]{0,0};
                \draw[module] (0.3,-0.5) \botlabel{L} -- (0.3,0.5);
            \end{tikzpicture}
            \ +\frac{z u^2 m_L(u)}{u^2-1}\
            \left(\,
                \begin{tikzpicture}[centerzero]
                    \draw (0,-0.5) -- (0,0.5);
                \end{tikzpicture}
                -
                \begin{tikzpicture}[centerzero]
                    \draw (0,-0.5) -- (0,0.5);
                    \uptriforce[black]{0,0};
                \end{tikzpicture}
            \right)
            \begin{tikzpicture}[centerzero]
                \draw[module] (0,-0.5) \botlabel{L} -- (0,0.5);
            \end{tikzpicture}
        \right]_{u^{0}}
        \\
        \overset{\cref{blacktrick}}{=}
        \left[
            \frac{(au+b)m_L(u^{-1})u^{d_L}}{u^2-1}\
            \begin{tikzpicture}[centerzero]
                \draw (0,-0.5) -- (0,0.5);
                \downtriforce[black]{0,0};
                \draw[module] (0.3,-0.5) \botlabel{L} -- (0.3,0.5);
            \end{tikzpicture}
            \ +\frac{z u^2 m_L(u)}{u^2-1}\
            \left(\,
                \begin{tikzpicture}[centerzero]
                    \draw (0,-0.5) -- (0,0.5);
                \end{tikzpicture}
                -
                \begin{tikzpicture}[centerzero]
                    \draw (0,-0.5) -- (0,0.5);
                    \uptriforce[black]{0,0};
                \end{tikzpicture}
            \right)
            \begin{tikzpicture}[centerzero]
                \draw[module] (0,-0.5) \botlabel{L} -- (0,0.5);
            \end{tikzpicture}
        \right]_{u^{0}},
    \end{multline*}
    where, in the last equality, we used the fact that $\check{m}_L(u)$ is the minimal polynomial of the inverse dot.

    First consider the case where $d_L$ is even.  Then, we can apply \cref{sofa} with $f(u)=am_L(u)$, $r=d_L+1$, and with $f(u) = bu m_L(u)$, $r=d_L+1$, to get
    \begin{multline} \label{question}
        0
        =
        \left[
            \frac{(au+b)m_L(u^{-1})u^{d_L}}{u^2-1}\
            \begin{tikzpicture}[centerzero]
                \draw (0,-0.5) -- (0,0.5);
                \downtriforce[black]{0,0};
                \draw[module] (0.3,-0.5) \botlabel{L} -- (0.3,0.5);
            \end{tikzpicture}
            \ +\frac{z u^2 m_L(u)}{u^2-1}\
            \left(\,
                \begin{tikzpicture}[centerzero]
                    \draw (0,-0.5) -- (0,0.5);
                \end{tikzpicture}
                -
                \begin{tikzpicture}[centerzero]
                    \draw (0,-0.5) -- (0,0.5);
                    \uptriforce[black]{0,0};
                \end{tikzpicture}
            \right)
            \begin{tikzpicture}[centerzero]
                \draw[module] (0,-0.5) \botlabel{L} -- (0,0.5);
            \end{tikzpicture}
        \right]_{u^{0}}
        \\
        = -
        \left[
            \frac{(zu^2+bu^2+au)m_L(u)}{u^2-1}\
            \left(\,
                \begin{tikzpicture}[centerzero]
                    \draw (0,-0.5) -- (0,0.5);
                    \uptriforce[black]{0,0};
                \end{tikzpicture}
                -
                \begin{tikzpicture}[centerzero]
                    \draw (0,-0.5) -- (0,0.5);
                \end{tikzpicture}
            \right)
            \begin{tikzpicture}[centerzero]
                \draw[module] (0,-0.5) \botlabel{L} -- (0,0.5);
            \end{tikzpicture}
        \right]_{u^{0}}
        \overset{\cref{blacktrick+}}{=}
        \begin{tikzpicture}[centerzero]
            \draw (0,-0.5) -- (0,0.5);
            \multdot[black]{0,0}{east}{g(x)};
            \draw[module] (0.3,-0.5) \botlabel{L} -- (0.3,0.5);
        \end{tikzpicture},
    \end{multline}
    where $g(u) = -\left[ \frac{(zu^2+bu^2+au)m_L(u)}{u^2-1} \right]_{u^{\ge 1}}$.  We take the strictly positive powers of $u$, instead of the nonnegative powers of $u$, because of the subtraction of the identity strand in the penultimate expression above.  Note that $\deg g \leq d_L$, and its constant term is zero.  Since \cref{question} implies that $g \in J_L$, it follows that $u^{-1}g(u) \in \kk[u] \cap J_L$.  This contradicts the definition of $m_L$ unless $g=0$.  Therefore,
    \begin{equation}\label{answer}
        \frac{(zu^2+bu^2+au)m_L(u)}{u^2-1}
        = m_L(u) \left( (z+b) + au^{-1} + (z+b)u^{-2} + au^{-3} + \dotsb \right)
        \in \kk\llbracket u^{-1}\rrbracket.
    \end{equation}
    Since $d_L$ is even, we must have $d_L \geq 2$, and so the vanishing of the $u^{d_L}$ and $u^{d_L-1}$ terms of \cref{answer} requires that $b=-z$, $a=0$, since the entire series vanishes. This implies that $\rOO_L = \rOO_{m_L}$ when $d_L$ is even.

    On the other hand, if $d_L$ is odd, we can apply \cref{sofa} with $f(u)=bm_L(u)$, $r=d_L$, and with $f(u)= a u m_L(u)$, $r=d_L+2$, giving us
    \begin{multline*}
        0
        =
        \left[
            \frac{(au+b)m_L(u^{-1})u^{d_L}}{u^2-1}\
            \begin{tikzpicture}[centerzero]
                \draw (0,-0.5) -- (0,0.5);
                \downtriforce[black]{0,0};
                \draw[module] (0.3,-0.5) \botlabel{L} -- (0.3,0.5);
            \end{tikzpicture}
            \ +\frac{m_L(u)(zu^2)}{u^2-1}\
            \left(\,
                \begin{tikzpicture}[centerzero]
                    \draw (0,-0.5) -- (0,0.5);
                \end{tikzpicture}
                -
                \begin{tikzpicture}[centerzero]
                    \draw (0,-0.5) -- (0,0.5);
                    \uptriforce[black]{0,0};
                \end{tikzpicture}
            \right)
            \begin{tikzpicture}[centerzero]
                \draw[module] (0,-0.5) \botlabel{L} -- (0,0.5);
            \end{tikzpicture}
        \right]_{u^{0}}\\
        =
        \left[
            \frac{m_L(u)(zu^2+au^2+bu)}{u^2-1}\
            \left(\,
                \begin{tikzpicture}[centerzero]
                    \draw (0,-0.5) -- (0,0.5);
                \end{tikzpicture}
                -
                \begin{tikzpicture}[centerzero]
                    \draw (0,-0.5) -- (0,0.5);
                    \uptriforce[black]{0,0};
                \end{tikzpicture}
            \right)
            \begin{tikzpicture}[centerzero]
                \draw[module] (0,-0.5) \botlabel{L} -- (0,0.5);
            \end{tikzpicture}
        \right]_{u^{0}}
        \overset{\cref{blacktrick+}}{=}
        \begin{tikzpicture}[centerzero]
            \draw (0,-0.5) -- (0,0.5);
            \multdot[black]{0,0}{east}{g(x)};
            \draw[module] (0.3,-0.5) \botlabel{L} -- (0.3,0.5);
        \end{tikzpicture},
    \end{multline*}
    where $g(u) = -\left[ \frac{(zu^2+au^2+bu)m_L(u)}{u^2-1} \right]_{u^{\ge 1}} $. An argument analogous to \cref{answer} shows that we must have $a=-z$, and, if $d_L>1$, that $b=0$.  This implies that $\rOO_L = \rOO_{m_L}$ when $d_L > 1$ and $d_L$ is odd.

    It remains to consider the case $d_L = 1$, when we have $m_L(u) = u - M$.  First note that
    \[
        M\
        \begin{tikzpicture}[centerzero]
            \draw (0,-0.5) -- (0,0.5);
            \draw[module] (0.3,-0.5) \botlabel{L} -- (0.3,0.5);
        \end{tikzpicture}
        =
        \begin{tikzpicture}[centerzero]
            \draw (0,-0.5) -- (0,0.5);
            \singdot[black]{0,0};
            \draw[module] (0.3,-0.5) \botlabel{L} -- (0.3,0.5);
        \end{tikzpicture}
        \overset{\cref{chicken}}{=} t\
        \begin{tikzpicture}[centerzero]
            \draw (0,-0.5) to[out=up,in=180] (0.3,0.2) to[out=0,in=up] (0.45,0) to[out=down,in=0] (0.3,-0.2);
            \draw[wipe] (0.3,-0.2) to[out=180,in=down] (0,0.5);
            \draw (0.3,-0.2) to[out=180,in=down] (0,0.5);
            \singdot[black]{0.01,-0.3};
            \draw[module] (0.7,-0.5) \botlabel{L} -- (0.7,0.5);
        \end{tikzpicture}
        \overset{\cref{kauffdot}}{=} t\
        \begin{tikzpicture}[centerzero]
            \draw (0.3,0.2) to[out=0,in=up] (0.45,0) to[out=down,in=0] (0.3,-0.2) to[out=180,in=down] (0,0.5);
            \draw[wipe] (0,-0.5) to[out=up,in=180] (0.3,0.2);
            \draw (0,-0.5) to[out=up,in=180] (0.3,0.2);
            \multdot[black]{0.45,0}{west}{-1};
            \draw[module] (1.1,-0.5) \botlabel{L} -- (1.1,0.5);
        \end{tikzpicture}
        \overset{\cref{drops}}{=} M^{-1} t^2\
        \begin{tikzpicture}[centerzero]
            \draw (0,-0.5) -- (0,0.5);
            \draw[module] (0.3,-0.5) \botlabel{L} -- (0.3,0.5);
        \end{tikzpicture}
        \, .
    \]
    Thus, $M = \pm t$.  We can then compute directly
    \begin{multline*}
        \rOO_L(u)
        \overset{\cref{alipay}}{=}
        \begin{tikzpicture}[centerzero]
            \bubgenr[black]{u}{0,0};
            \draw[module] (0.4,-0.4) \botlabel{L} -- (0.4,0.4);
        \end{tikzpicture}
        \overset{\cref{eatzk}}{\underset{\cref{skein}}{=}}
        \frac{(t^{-1}-z)u^2 - t^{-1}}{u^2-zu-1} + \frac{zu(u^2-1)}{(u^2-zu-1)(u-M)}
        \left( \frac{t-t^{-1}}{z} + 1 \right)
        \\
        =
        \left( \frac{tu^2 \pm ztu - t}{u^2-zu-1} \right) \frac{u \mp t^{-1}}{u \mp t},
    \end{multline*}
    which matches $\rOO_{m_L}$.
\end{proof}

\begin{cor} \label{backtri}
    We have
    \begin{equation}\label{backtri1}
        \begin{tikzpicture}[centerzero]
            \uptribubr[black]{-0.5,0};
            \draw[module] (0,-0.4) \botlabel{L} -- (0,0.4);
        \end{tikzpicture}
        =
        \begin{dcases}
            \frac{u^2}{u^2-1}- t^{-1}z^{-1} +\left(tz^{-1}-\frac{M}{u^2-1}\right) \frac{\check{m}_L(u)}{m_L(u)} & \text{if } d_L \text{ is even}, \\
            \frac{u^2}{u^2-1}- t^{-1}z^{-1} + \left(tz^{-1}-\frac{uM}{u^2-1}\right) \frac{\check{m}_L(u)}{m_L(u)} & \text{if } d_L \text{ is odd}.
    	\end{dcases}
    \end{equation}
    and
    \begin{equation}\label{backtri2}
        \begin{tikzpicture}[centerzero]
            \uptribubl[black]{-0.5,0};
            \draw[module] (0,-0.4) \botlabel{L} -- (0,0.4);
        \end{tikzpicture}
        =
        \begin{dcases}
            \frac{u^2}{u^2-1}+tz^{-1}-\left(t^{-1}z^{-1}+\frac{M^{-1}}{u^2-1}\right) \frac{m_L(u)}{\check{m}_L(u)} & \text{if } d_L \text{ is even}, \\
            \frac{u^2}{u^2-1}+tz^{-1}-\left(t^{-1}z^{-1}+\frac{uM^{-1}}{u^2-1}\right) \frac{m_L(u)}{\check{m}_L(u)} & \text{if } d_L \text{ is odd}.
    	\end{dcases}
    \end{equation}
\end{cor}

\begin{proof}
    By \cref{eatzk}, we have
    \begin{gather*}
        \uptribubbler[black]
        =
        \left( \frac{u^2}{u^2-1}- t^{-1}z^{-1} \right) 1_\one + \frac{u^2-zu-1}{z(u^2-1)}\, \bubblegenr[black]{u}
        \, ,
        \\
        \uptribubblel[black]
        = \left( \frac{u^2}{u^2-1}+tz^{-1} \right) 1_\one - \frac{u^2+zu-1}{z(u^2-1)}\, \bubblegenl[black]{u}\, .
    \end{gather*}
    Thus, \cref{backtri1,backtri2} follow from \cref{backhome}.
\end{proof}

Fix $q$ such that $z=q-q^{-1}$ (extending $\kk$ if needed). Since $z \neq 0$, we have $q\neq \pm 1$.  For a polynomial $f \in \kk[u]$ satisfying the conditions of \cref{sneeze}, define
\begin{equation} \label{what-is-epsilon}
    (\epsilon_1(f),\epsilon_2(f))
    =
    \begin{cases}
		(1,1) & \text{if } \deg f \text{ is even and } f(0) = qt, \\
		(-1,-1) & \text{if } \deg f \text{ is even and } f(0) = -q^{-1}t, \\	
		(1,-1) & \text{if } \deg f \text{ is odd and } f(0) = t, \\
		(-1,1) & \text{if } \deg f \text{ is odd and } f(0) = -t. \\
	\end{cases}
\end{equation}
One of these cases must hold by \cref{netherlands}.  Note that, in each of these cases, we have
\begin{equation}\label{f-degree}
	\deg f\equiv\frac{\epsilon_1(f)+\epsilon_2(f)}{2}+1\pmod 2 \qquad f(0) = \epsilon_1(f) q^{\frac{\epsilon_1(f)+\epsilon_2(f)}{2}}t.
\end{equation}
Since $q \ne \pm 1$, the pair $(\epsilon_1(f),\epsilon_2(f))$ is uniquely determined by the second equation in \cref{f-degree}.  We can then rewrite \cref{metrobusKr,metrobusKl} as
\begin{equation}\label{metrobusKr-1}
    \rOO_f(u) = t \frac{\left( u-q^{\epsilon_1(f)} \right) \left( u+q^{\epsilon_2(f)} \right)}{(u-q)(u+q^{-1})}\frac{\check{f}(u)}{f(u)}, \qquad
    \lOO_f(u) = t^{-1}\frac{\left( u-q^{-\epsilon_1(f)} \right) \left( u+q^{-\epsilon_2(f)} \right)}{(u+q)(u-q^{-1})}\frac{f(u)}{\check{f}(u)}.
\end{equation}
\details{
    When $\deg f$ is even, we have $\epsilon_1(f) = \epsilon_2(f)$, and so \cref{metrobusKr,metrobusKl} become
    \begin{gather*}
        \rOO_f(u) = t \frac{u^2-1-zt^{-1}F}{u^2-zu-1} \frac{\check{f}(u)}{f(u)}
        = t \frac{u^2-1- \epsilon_1(f) \left( q^{\epsilon_1(f)+1} - q^{\epsilon_1(f)-1} \right)}{(u-q)(u+q^{-1})} \frac{\check{f}(u)}{f(u)},
        \\
        \lOO_f(u) = t^{-1} \frac{u^2-1 + ztF^{-1}}{u^2+zu-1} \frac{\check{f}(u)}{f(u)}
        = t^{-1} \frac{u^2-1 + \epsilon_1(f) \left( q^{-\epsilon_1(f)+1} - q^{-\epsilon_1(f)-1} \right)}{(u+q)(u-q^{-1})} \frac{\check{f}(u)}{f(u)}.
    \end{gather*}
    When $\epsilon_1(f) = 1$, we have
    \begin{gather*}
        u^2 - 1 - \epsilon_1(f) \left( q^{\epsilon_1(f)+1} - q^{\epsilon_1(f)-1} \right)
        = u^2 - 1 - (q^2-1) = (u-q)(u+q),
        \\
        u^2 - 1 + \epsilon_1(f) \left( q^{-\epsilon_1(f)+1} - q^{-\epsilon_1(f)-1} \right)
        = u^2 - 1 + (1-q^{-2}) = (u-q^{-1})(u+q^{-1}),
    \end{gather*}
    as desired.  Similarly, when $\epsilon_1(f) = -1$, we have
    \begin{gather*}
        u^2 - 1 - \epsilon_1(f) \left( q^{\epsilon_1(f)+1} - q^{\epsilon_1(f)-1} \right)
        = u^2 - 1 + (1-q^{-2}) = (u-q^{-1})(u+q^{-1}),
        \\
        u^2 - 1 + \epsilon_1(f) \left( q^{-\epsilon_1(f)+1} - q^{-\epsilon_1(f)-1} \right)
        = u^2 - 1 - (q^2-1)
        = (u-q)(u+q).
    \end{gather*}

    When $\deg f$ is odd, we have $\epsilon_1(f) = -\epsilon_2(f)$, and so
    \begin{gather*}
        \rOO_f(u) = t \frac{u^2-1-zt^{-1}Fu}{u^2-zu-1} \frac{\check{f}(u)}{f(u)}
        = t \frac{u^2-1 - \epsilon_1(f)(q - q^{-1})u}{(u-q)(u+q^{-1})} \frac{\check{f}(u)}{f(u)},
        \\
        \lOO_f(u) = t^{-1} \frac{u^2-1 + ztF^{-1}u}{u^2+zu-1} \frac{\check{f}(u)}{f(u)}
        =  t^{-1} \frac{u^2-1 + \epsilon_1(f) (q-q^{-1})u}{(u+q)(u-q^{-1})} \frac{\check{f}(u)}{f(u)}
    \end{gather*}
    When $\epsilon_1(f) = 1$, we have
    \begin{gather*}
        u^2-1 - \epsilon_1(f)(q - q^{-1})u = u^2 - 1 - qu + q^{-1}u = (u-q)(u+q^{-1}),
        \\
        u^2-1 + \epsilon_1(f) (q-q^{-1})u = u^2 - 1 + qu - q^{-1}u = (u-q^{-1})(u+q),
    \end{gather*}
    as desired.  Similarly, when $\epsilon_1(f)=-1$, we have
    \begin{gather*}
        u^2-1 - \epsilon_1(f)(q - q^{-1}) = u^2 - 1 + q - q^{-1} = (u-q^{-1})(u+q),
        \\
        u^2-1 + \epsilon_1(f) (q-q^{-1})u = u^2 - 1 -qu + q^{-1}u = (u-q)(u+q^{-1}).
    \end{gather*}
}
Note that the equality $\rOO_{f}(u^{-1}) = \lOO_{f}(u)$ means that when we expand $\rOO_{f}(u^{-1})$ as a Taylor series at $u=\infty$, the constant term is $t$, whereas at $u=0$, it is $t^{-1}$.

%----------------------------------------------------------------------------
\subsection{Admissibility in the Kauffman case\label{sec:admissibleKauffman}}
%----------------------------------------------------------------------------

In this section, we explain how admissibility conditions for cyclotomic BMW algebras can be deduced from our results.  This analysis is very similar to that from \cref{sec:admissibleBrauer}, which considered the degenerate case.  Therefore, we will be brief here.  Although different notions of admissibility were originally introduced by Rui and Xu \cite{RX09} and by Wilcox and Yu \cite{WY11}, these were shown to be equivalent in \cite[Th.~4.4]{Goo10}.  We therefore restrict our attention to the conditions from \cite{RX09}, which are also used in \cite{GRS22}.  Throughout, we will use \cref{glass} to move between our definition of the affine Kauffman category and that in \cite{GRS22}.

Fix an $\AK$-module category $\mathbf{R} \colon \AK \to \cEnd_\kk(\cR)$.  As explained in \cref{sec:admissibleBrauer}, $\cR$ defines a ring homomorphism from $\End_\AB(\one)$ to the center $Z_V$ of the endomorphism algebra of $V$.  For $V$ an object in an $\AK$-module category $\cR$, let
\[
    \Omega_V = \left( \omega^V_r \right)_{r \in \Z},\qquad
    \omega^V_r
    :=
    \begin{tikzpicture}[centerzero]
        \multbubr[black]{0,0}{r};
        \draw[module] (0.7,-0.5) \botlabel{V} -- (0.7,0.5);
    \end{tikzpicture}
    \in Z_V,\qquad r \in \Z.
\]

For a sequence $\Omega = (\omega_r)_{r \in \Z}$ in $\kk$ with $\omega_0 = \frac{t-t^{-1}}{z}+1$, define
\[
    \Omega^{\ge 0} = \left( \omega_r \right)_{r \ge 0}
    \qquad \text{and} \qquad
    \Omega^{\le 0} = \left( \omega_r \right)_{r \le 0}.
\]
Then, inspired by \cref{eatzk}, we define
\[
    \begin{gathered}
        \rOO_\Omega(u) := \frac{(t^{-1}-z)u^2 - t^{-1}}{u^2-zu-1} + \frac{z(u^2-1)}{u^2-zu-1} \Omega^{\ge 0}
        \in t 1_\one + u^{-1} \kk \llbracket u^{-1} \rrbracket,
        \\
        \lOO_\Omega(u) := \frac{(t+z)u^2-t}{u^2+zu-1} - \frac{z(u^2-1)}{u^2+zu-1} \Omega^{\le 0}
        \in t^{-1} 1_\one + u^{-1} \kk \llbracket u^{-1} \rrbracket.
    \end{gathered}
\]
Suppose $m \in \kk[u]$ is monic, has nonzero constant term, and
\[
    m(u) = \prod_{a=1}^d (u-a_i)
\]
in the algebraic closure of $\kk$.

In \cite[Def.~2.19]{RX09} and \cite[Def.~1.7]{GRS22}, the data $(m,\Omega)$ is called \emph{admissible} if it satisfies two conditions:
\begin{enumerate}
    \item \label{fishy} The condition \cite[Def.~1.7(1)]{GRS22} is equivalent to the statement that
        \begin{equation} \label{narita}
            \rOO_\Omega \lOO_\Omega = 1,
        \end{equation}
        which holds for the sequence $\Omega_V$ by  \cref{infgrassk}.
    \item When $\Omega = \Omega_V$, the  condition \cite[Def.~1.7(2)]{GRS22} is equivalent to the statement
        \[
            \begin{tikzpicture}[centerzero]
                \multbubr[black]{0,0}{x^n m(x)};
                \draw[module] (1.5,-0.4) \botlabel{V} -- (1.5,0.4);
            \end{tikzpicture}
            = 0
            \qquad \text{for all } n \in \Z.
        \]
        This holds automatically if
        \[
            \begin{tikzpicture}[centerzero]
                \draw (0,-0.4) -- (0,0.4);
                \multdot[black]{0,0}{east}{m(x)};
                \draw[module] (0.3,-0.4) \botlabel{V} -- (0.3,0.4);
            \end{tikzpicture}
            = 0.
        \]
\end{enumerate}

The condition \cite[Assumption~1.9]{GRS22} (see also \cite[Def.~2.20(2)]{RX09}) on the sequence $\ba = (a_1,a_2,\dotsc,a_{d})$ (denoted $\mathbf{u}$ in \cite{GRS22}) is, in our language,
\begin{equation} \label{battery}
    m(0)
    = \prod_{i=1}^{d}(-a_i)
    =
    \begin{cases}
        \pm t & \text{if } \deg m \text{ is odd}, \\
        qt \text{ or } -q^{-1}t & \text{if } \deg m \text{ is even},
    \end{cases}
\end{equation}
where, in the second case, $z = q - q^{-1}$.  By \cref{sneeze,netherlands}, this must hold if condition \cref{fishy} holds for $\rOO_{m}$, $\lOO_m$.  Assuming $m$ satisfies \cref{battery}, the notion of \emph{$\ba$-admissibility} for $\Omega$ is defined in \cite[Def.~1.10]{GRS22}; see also \cite[Lem.~2.28]{RX09}.  In our language, it follows as in the proof of \cref{backtri} that
\begin{equation} \label{ladder}
    \Omega \text{ is $\ba$-admissible}
    \iff
    \rOO_\Omega = \rOO_m \text{ and } \lOO_\Omega = \lOO_m \text{ for } m(u) = \prod_{i=1}^{d}(u-a_i).
\end{equation}
Note that, if \cref{narita,battery} are satisfied, then $\rOO_\Omega = \rOO_m$ if and only if $\lOO_\Omega = \lOO_m$, by \cref{sneeze}.

%=================================================================
\section{Cyclotomic Brauer categories\label{sec:cyclotomicBrauer}}
%=================================================================

In this section, we consider the cyclotomic Brauer categories introduced in \cite{RS19}.  These are a categorical analogue of the degenerate cyclotomic BMW algebras; as \cref{mexico} below shows, the endomorphism algebras in these categories are degenerate cyclotomic BMW algebras, though the parameters that appear are slightly subtle.  Throughout this section, we assume that $\kk$ is a field whose characteristic is not $2$.

%-------------------------------------------
\subsection{Definition and first properties}
%-------------------------------------------

Fix a monic polynomial $p(u) \in \kk[u]$ and a power series
\[
    \OO(u) = \sum_{r \ge -1} \OO^{(r)} u^{-r-1} \in 1 + u^{-1} \kk \llbracket u^{-1} \rrbracket.
\]
Let $\cI(p,\OO)$ be the left tensor ideal of $\AB$ generated by
\begin{equation} \label{frijoles}
    \multdotstrand{east}{p(x)}
    \qquad \text{and} \qquad
    \left[ \bubblegenr{u} \right]_{u^{-r-1}} - \OO^{(r)} 1_\one,\quad r \in \N.
\end{equation}
The corresponding \emph{cyclotomic Brauer category} is the quotient
\[
    \CB(p,\OO) := \AB/\cI(p,\OO).
\]
It follows from \cref{infgrass} that $\CB(p,\OO)$ is the zero category unless
\begin{equation} \label{lilac}
    \OO(u) \OO(-u) = 1.
\end{equation}
Therefore, we assume for the remainder of this section that \cref{lilac} is satisfied.

\begin{rem}
    The above definition coincides with the \emph{(specialized) cyclotomic Brauer category} of \cite[Def.~1.7]{RS19}, except that some additional conditions are placed on the parameters there.  Note also that \cite{RS19} works with \emph{right} tensor ideals instead of left tensor ideals.  This agrees with the fact that we consider the reverse of their affine Brauer category; see \cref{RSflip}.  We omit the word ``specialized'' in the terminology since we do not use the other cyclotomic Brauer category of \cite[Def.~1.5]{RS19}, where they take the quotient by
    $
        \begin{tikzpicture}[centerzero]
            \draw (0,-0.2) -- (0,0.2);
            \multdot{0,0}{east}{p(x)};
        \end{tikzpicture}
    $.
\end{rem}

Let $L(p,\OO)$ denote the image of $\one$ in the quotient $\CB(p,\OO)$.  The category $\CB(p,\OO)$ is no longer monoidal, but it is a left module category over $\AB$ and it is generated under this action by $L(p,\OO)$.  By \cite[Th.~B]{RS19}, all the  endomorphisms of $L(p,\OO)$ are polynomials in the bubbles, which have all been specialized to scalars.  This shows that the element $L(p,\OO)$ is a brick if it is nonzero.  Recall, from \cref{strada}, that $m_{L(p,\OO)}$ is the monic minimal polynomial of the action of the dot on $1_{\go L(p,\OO)}$.  Recall also the definition \cref{metrobus} of $\OO_f$.

\begin{prop}[{\cite[Th.~C]{RS19}}]\label{RS-thC}
    Suppose $f \in \kk[u]$ is monic.  The endomorphism ring \linebreak $\End(\go L(f,\OO_f))$ in the category $\CB(f,\OO_f)$ is $\C[\dotstrand]/(f(\dotstrand))$.
\end{prop}

In fact, \cite[Th.~C]{RS19} describes a basis for \emph{all} morphism spaces in $\CB(f,\OO_f)$, but for our purposes, we only need to know the result for endomorphisms of $\go L(f,\OO_f)$.

\begin{proof}
    The category denoted $\CB^f(\omega)$ in \cite{RS19} is denoted
    $\CB(f,\OO_\Omega)$ in our language, where $\Omega = \omega =
    (\omega_k)_{k \in \N}$ is a sequence in $\kk$ and $\OO_\Omega$ is
    defined in \cref{OOmega}.  As we explain below in \cref{bench}, if
    the $d$-tuple $\ba$ are the roots of $f$ with multiplicity, then $\OO_\Omega$ is $\ba$-admissible, in the language of \cite{RS19}, if $\OO_\Omega = \OO_f$.  Thus, the result follows from \cite[Th.~C]{RS19}.
\end{proof}

\begin{theo} \label{mexico}
    The cyclotomic Brauer category $\CB(p,\OO)$ is not the zero category if and only if $\OO = \OO_f$ for some positive-degree polynomial $f$ dividing $p$.  If this category is nonzero, then $\OO = \OO_{m_{L(p,\OO)}}$, and $m_{L(p,\OO)}$ is divisible by any polynomial $f$ that divides $p$ and satisfies $\OO = \OO_f$.
\end{theo}

\begin{proof}
    First suppose that $\OO = \OO_f$ for some positive-degree polynomial $f$ dividing $p$.  Then we have the quotient functor $\CB(p,\OO) \to \CB(f,\OO)$.  \Cref{RS-thC} implies that $\CB(f,\OO)$ is not the zero category, and hence $\CB(p,\OO)$ is not the zero category.  For the other direction, suppose that $\CB(p,\OO)$ is not the zero category.  Then $m_{L(p,\OO)}$ has positive degree and \cref{hemlock} implies that $\OO = \OO_{m_{L(p,\OO)}}$.

    Now let $f$ be a monic polynomial dividing $p$ and satisfying $\OO = \OO_f$.  Then we have the quotient functor $\CB(p,\OO) \to \CB(f,\OO)$.  It follows that $m_{L(f,\OO)}$ divides $m_{L(p,\OO)}$.  On the other hand, \cref{RS-thC} implies that the endomorphisms $\multdotstrand{east}{n}$, $0 \le n \le \deg f - 1$, are linearly independent in $\CB(f,\OO)$.  It follows that $f = m_{L(f,\OO)}$, and so $f$ divides $m_{L(p,\OO)}$, as claimed.
\end{proof}

\begin{cor} \label{tacos}
    If $\CB(p,\OO)$ is not the zero category, then it is isomorphic to
    \[
        \CB(m) := \CB(m,\OO_m)
        \qquad \text{for } m = m_{L(p,\OO)}.
    \]
\end{cor}

If $L$ is a brick in an $\AB$-module category $\cR$, then it follows from \cref{hemlock} that the action of $\AB$ on $L$ factors through $\CB(m_L)$.  Thus, we have the following proposition.

\begin{prop} \label{delah}
    Suppose $\ba = (a_1,a_2,\dotsc,a_{d})$ is a finite sequence in $\kk$.  The sequence $\Omega$ is $\ba$-admissible if and only if $\Omega = \Omega_L$ for some brick $L$ in an $\AB$-module category $\cR$ satisfying  $m_L(u) = \prod_{i=1}^{d} (u-a_i)$.
\end{prop}

\begin{proof}

    If $L$ is a brick in an $\AB$-module category $\cR$ with $m_L(u) = \prod_{i=1}^{d} (u-a_i)$, then $\Omega_L$ is $\ba$-admissible by \cref{cactus}.

    Conversely, suppose that the sequence $\Omega$ is $\ba$-admissible.  Define $m(u) = \prod_{i=1}^{d} (u-a_i)$, and let $L$ be the object of the cyclotomic Brauer category $\CB(m)$ corresponding to the unit object $\one$ of $\AB$.  Then it follows from \cref{RS-thC} that $m = m_L$.  Thus,
    \[
        \OO_\Omega \overset{\cref{bench}}{=} \OO_{m_L} \overset{\text{Th.~}\ref{hemlock}}{=} \OO_L
        \qquad \implies \qquad
        \Omega = \Omega_L.\qedhere
    \]
\end{proof}

\begin{cor}[{\cite[Cor.~3.9]{AMR06}}] \label{beef}
    If a sequence $\Omega$ in $\kk$ is $\ba$-admissible for some finite sequence $\ba$ in $\kk$, then $\Omega$ is weakly admissible (and hence also admissible).
\end{cor}

\begin{proof}
    Suppose $\Omega$ is $\ba$-admissible, with $\ba = (a_1,a_2,\dotsc,a_{d})$.  Then, by \cref{delah}, $\Omega = \Omega_L$ for some object $L$ in an $\AB$-module category $\cR$ satisfying $\End_\cR(L) = \kk$ and $m_L(u) = \prod_{i=1}^{d} (u-a_i)$.  It then follows from \cref{roses} that $\Omega$ is weakly admissible.
\end{proof}

%--------------------------------
\subsection{Calculation of $m_L$}
%--------------------------------

It follows from \cref{mexico} that, when $\CB(p,\OO)$ is not the zero category, then $m_{L(p,\OO)}$ is the unique monic polynomial of maximal degree in the set of polynomials $f$ that divide $p$ and satisfy $\OO = \OO_f$.  Our next goal is to describe $m_{L(p,\OO)}$ explicitly.  To simplify notation, set
\[
    L = L(p,\OO),\qquad
%    I = I_{L(p,\OO)},\qquad
    m = m_{L(p,\OO)}\qquad d=\deg m.
\]
Recall that $I_L$ is the ideal of $\kk[u]$ generated by $m_L$.  The following lemma will be useful.

\begin{lem} \label{hack}
    If elements $x,y,z,w$ of a principal ideal domain $R$ satisfy $x/y=z/w$ in the fraction field of $R$, then $x/y = a \gcd(x,z)/\gcd(y,w)$ for some unit $a \in R$.
\end{lem}

\begin{proof}
    Using $\sim$ to indicate equality up to multiplication by a unit in $R$, we have
    \[
        w \gcd(x,z)
        \sim \gcd(xw,zw)
        \sim \gcd(yz,zw)
        \sim z \gcd(y,w),
    \]
    and the result follows.
\end{proof}

\begin{convention} \label{monic}
    For $a_1,a_2,\dotsc,a_n \in \kk[u]$, the notation $\gcd(a_1,a_2,\dotsc,a_n)$ will denote the unique \emph{monic} gcd of the elements $a_1,a_2,\dotsc,a_n$.
\end{convention}

By \cref{cdmx}, we have
\begin{equation} \label{apple}
    \hat{p}(u) := (-u-\tfrac{1}{2})p(-u)\OO(-u) \in I_L.
\end{equation}
As in \cref{meteor}, by \cref{lilac} and the definition of $\hat{p}$, we have
\[
    \frac{\hat{p}(-u)}{(u-\tfrac{1}{2})p(u)}
    = \OO
    = \frac{(-u-\tfrac{1}{2})p(-u)}{\hat{p}(u)}.
\]
Since the leading term of $\OO$ is $1$, it follows from \cref{hack} that
\begin{equation} \label{cutters1}
    \OO(u) = (-1)^{\deg Q}\frac{Q(-u)}{Q(u)},
    \qquad \text{where} \quad
    Q(u)=\gcd \left( (u-\tfrac{1}{2})p(u),\hat{p}(u) \right) \in I_L.
\end{equation}

\begin{lem} \label{garlic1}
    If $\deg Q$ is odd, then $Q$ is divisible by $u$, and $\frac{Q(u)}{u} \in I_L$.
\end{lem}

\begin{proof}
    Suppose $\deg Q$ is odd.  Since $Q \in I_L$, the polynomial $m$ divides $Q$.  Let  $f=Q/m$.  Since
    \begin{equation}\label{biscuit}
        \frac{\left( (-1)^{d}u-\frac{1}{2} \right) m(-u)}{(u-\frac{1}{2})m(u)}
        = \OO
        = -\frac{Q(-u)}{Q(u)}
        = -\frac{f(-u)m(-u)}{f(u)m(u)},	
    \end{equation}
    we have
    \[
        \left( (-1)^{d}u-\tfrac{1}{2} \right) f(u)
        = - \left( u-\tfrac{1}{2} \right) f(-u).
    \]
    Thus, if $d$ is even, then $f(u)$ is an odd function of $u$.  Similarly, if $d$ is odd, then $(u+1/2)f(u)$ is an odd function of $u$.  In both cases, $u=0$ is a root of $f$, and hence a root of $Q$.  So, $Q$ is divisible by $u$.

    Next, setting $\bar{Q}(u) = \frac{Q(u)}{u}$, we have
    \begin{multline*}
        0
        =
        \begin{tikzpicture}[centerzero]
            \draw (0,-0.5) to[out=up,in=180] (0.3,0.2) to[out=0,in=up] (0.45,0) to[out=down,in=0] (0.3,-0.2) to[out=180,in=down] (0,0.5);
            \multdot{0.45,0}{west}{Q(x)};
            \draw[module] (1.6,-0.5) \botlabel{L} -- (1.6,0.5);
        \end{tikzpicture}
        \overset{\cref{trick}}{=}
        \left[
            Q(u)\
            \begin{tikzpicture}[centerzero]
                \draw (0,-0.5) to[out=up,in=180] (0.3,0.2) to[out=0,in=up] (0.45,0) to[out=down,in=0] (0.3,-0.2) to[out=180,in=down] (0,0.5);
                \uptriforce{0.45,0};
                \draw[module] (0.8,-0.5) \botlabel{L} -- (0.8,0.5);
            \end{tikzpicture}
        \right]_{u^{-1}}
        \\
        \overset{\cref{curlsup}}{=}
        \left[
            \left( 1-\frac{1}{2u} \right) Q(u)\
            \begin{tikzpicture}[centerzero]
                \draw (0,-0.5) -- (0,0.5);
                \bubgenr{u}{0.4,0};
                \downtriforce{0,0};
                \draw[module] (0.8,-0.5) \botlabel{L} -- (0.8,0.5);
            \end{tikzpicture}
            \ - \frac{Q(u)}{2u}\
            \begin{tikzpicture}[centerzero]
                \draw (0,-0.5) -- (0,0.5);
                \uptriforce{0,0};
                \draw[module] (0.3,-0.5) \botlabel{L} -- (0.3,0.5);
            \end{tikzpicture}
        \right]_{u^{-1}}
        =
        \left[
            \left( 1-\frac{1}{2u} \right) Q(u) \OO(u)\
            \begin{tikzpicture}[centerzero]
                \draw (0,-0.5) -- (0,0.5);
                \downtriforce{0,0};
                \draw[module] (0.4,-0.5) \botlabel{L} -- (0.4,0.5);
            \end{tikzpicture}
            \ - \frac{Q(u)}{2u}\
            \begin{tikzpicture}[centerzero]
                \draw (0,-0.5) -- (0,0.5);
                \uptriforce{0,0};
                \draw[module] (0.3,-0.5) \botlabel{L} -- (0.3,0.5);
            \end{tikzpicture}
        \right]_{u^{-1}}
        \\
        \overset{\cref{biscuit}}{=}
        \left[
            \left( -Q(-u) - \frac{1}{2} \bar{Q}(-u) \right)
            \begin{tikzpicture}[centerzero]
                \draw (0,-0.5) -- (0,0.5);
                \downtriforce{0,0};
                \draw[module] (0.4,-0.5) \botlabel{L} -- (0.4,0.5);
            \end{tikzpicture}
            \ - \frac{1}{2} \bar{Q}(u)\
            \begin{tikzpicture}[centerzero]
                \draw (0,-0.5) -- (0,0.5);
                \uptriforce{0,0};
                \draw[module] (0.3,-0.5) \botlabel{L} -- (0.3,0.5);
            \end{tikzpicture}
        \right]_{u^{-1}}
        \overset{\cref{trick}}{=}
        \begin{tikzpicture}[centerzero]
            \draw (0,-0.5) -- (0,0.5);
            \multdot{0,0}{east}{-Q(x)-\bar{Q}(x)};
            \draw[module] (0.3,-0.5) \botlabel{L} -- (0.3,0.5);
        \end{tikzpicture}
        \ .
    \end{multline*}
    Thus, $-Q-\bar{Q} \in I_L$, and hence $\bar{Q} \in I_L$, completing the proof.
\end{proof}

We are now ready to give the explicit expression for $m$.  Recall the definition \cref{apple} of $\hat{p}$.

\begin{theo} \label{UNAM}
    We have
    \[
        m_{L(p,\OO)}(u)
        =
        \begin{cases}
            \frac{\gcd(p(u),\hat{p}(u))}{u} & \text{if $\gcd((u-1/2)p(u),\hat{p}(u))$ has odd degree}, \\
            \gcd(p(u),\hat{p}(u)) & \text{otherwise}.
        \end{cases}
    \]
\end{theo}

\begin{proof}
    It follows from the definition \cref{cutters1} of $Q$ that, if $Q$ divides $p$, then $Q = \gcd(p,\hat{p})$.  On the other hand, if $Q$ does \emph{not} divide $p$, then $u-\frac{1}{2}$ divides $Q(u)$ and
    \[
        \frac{Q(u)}{u-\frac{1}{2}} = \gcd(p(u),\hat{p}(u)).
    \]
    Note that $\gcd(p,\hat{p}) \in I_L$ by \cref{cdmx}.  We break the proof into four cases.

    \smallskip

    \emph{Case 1: $Q$ has even degree and divides $p$}.  In this case, it follows from \cref{mexico,cutters1} that $Q$ divides $m$.  Since $Q$ are both monic, we have $m=Q$, as desired.

    \smallskip

    \emph{Case 2: $Q$ has odd degree and divides $p$}.  By \cref{garlic1}, $Q(u)$ is divisible by $u$, and so
    \[
        \OO(u)
        \overset{\cref{cutters1}}{=} -\frac{Q(-u)}{Q(u)}
        = \frac{\frac{Q(-u)}{-u}}{\frac{Q(u)}{u}}.
    \]
    Then it follows from \cref{mexico} that $\frac{Q(u)}{u}$ divides $m(u)$.  The polynomial $m(u)$ also divides $\frac{Q(u)}{u}$, by \cref{garlic1}. Since both polynomials are monic, we have $m(u)= \frac{Q(u)}{u}$, as desired.

    \smallskip

    \emph{Case 3: $Q$ has even degree and does not divide $p$}.  We have
    \[
        \OO(u)
        \overset{\cref{cutters1}}{=} \frac{Q(-u)}{Q(u)}
        = \frac{(-u-\frac{1}{2}) \frac{Q(-u)}{-u-\frac{1}{2}}}{(u-\frac{1}{2}) \frac{Q(u)}{u-\frac{1}{2}}}.
    \]
    Thus, by \cref{mexico}, $\frac{Q(u)}{u-\frac{1}{2}} = \gcd(p(u),\hat{p}(u))$ divides $m(u)$.  Since $\frac{Q(u)}{u-\frac{1}{2}}$ is monic, the result follows.

    \smallskip

    \emph{Case 4: $Q$ has odd degree and does not divide $p$}.  We have
    \[
        \OO(u)
        \overset{\cref{cutters1}}{=} - \frac{Q(-u)}{Q(u)}
        = \frac{(-u-\frac{1}{2}) \frac{Q(-u)}{(-u)(-u-\frac{1}{2})}}{(u-\frac{1}{2}) \frac{Q(u)}{u(u-\frac{1}{2})}},
    \]
    and the result again follows.
\end{proof}

Many of the relations \cref{frijoles} are redundant. We now give a more efficient presentation of cyclotomic Brauer categories.  By \cref{tacos}, it suffices to consider the categories $\CB(m)$ for $m \in \kk[u]$ of positive degree.

\begin{prop}
    For $m \in \kk[u]$, the category $\CB(m)$ is isomorphic to the quotient of $\AB$ by the left tensor ideal generated by
    \begin{equation} \label{pine}
        \begin{tikzpicture}[centerzero]
            \draw (0,-0.2) -- (0,0.2);x
            \multdot{0,0}{east}{m(x)};
        \end{tikzpicture}
        \qquad \text{and} \qquad
        \left[ \bubblegenr{u} \right]_{u^{-r-1}} - \OO_m^{(r)} 1_\one,
        \quad
        r \in 2\N,\ 0 \le r < \deg m.
    \end{equation}
\end{prop}

\begin{proof}
    Let $\CB(m)'$ be the quotient of $\AB$ by the left tensor ideal generated by \cref{pine}.  It suffices to show that $\multbubbler{r}$ is a scalar multiple of $1_\one$ in $\CB(m)'$ for all $r \in \N$.  Since
    \[
        \bubblegenr{u}
        \overset{\cref{eatz}}{=} \tfrac{2u}{2u-1} \uptribubbler + 1
        = 1 + \left( \sum_{r=0}^\infty \frac{1}{(2u)^r} \right) \uptribubbler
        = 1 + \sum_{r=0}^\infty \left( \sum_{s=0}^r \frac{1}{2^{r-s}} \multbubbler{s} \right) u^{-r-1},
    \]
    it follows from \cref{pine} that $\multbubbler{r}$ is a scalar multiple of $1_\one$ in $\CB(m)'$ for $r \in 2\N$, $0 \le r < \deg m$.  Furthermore, by \cref{scratch}, for $r \in 2\N+1$, the bubble $\multbubbler{r}$ can be expressed in terms of $\multbubbler{s}$ for $s<r$.  Therefore, $\multbubbler{r}$ is a scalar multiple of $1_\one$ in $\CB(m)'$ for all $0 \le r < \deg m$.

    Now, since
    $
        \begin{tikzpicture}[centerzero]
            \draw (0,-0.2) -- (0,0.2);x
            \multdot{0,0}{east}{m(x)};
        \end{tikzpicture}
        = 0
    $
    in $\CB(m)'$, any bubble $\multbubbler{r}$ for $r \ge \deg m$ can be written as a linear combination of $\multbubbler{s}$ for $s < r$.  Therefore, $\multbubbler{r}$ is a scalar multiple of $1_\one$ in $\CB(m)'$ for \emph{all} $r \in \N$.
\end{proof}

%-----------------------------------
\subsection{Cyclotomic BMW algebras}
%-----------------------------------

Since we have worked exclusively in the affine Brauer \emph{category},
it might not be clear to the reader whether our arguments yield any
results about the degenerate cyclotomic BMW \emph{algebras}, which is
the context of the earlier admissibility results we discussed.  In the
remainder of this section, we explain how our arguments do, in fact, yield such results.

For a polynomial $p \in \kk[u]$ that factors completely in $\kk$ (we can always pass to the algebraic closure of $\kk$ to ensure that this is the case), $\Omega = (\omega_r)_{r \in \N}$ a sequence of elements of $\kk$, and $n \in \N$, we denote by $W_n(p,\Omega)$ the corresponding $n$-strand degenerate cyclotomic BMW algebra.  (We do not assume any admissibility conditions on these data.)  The algebra $W_n(p,\Omega)$ is denoted $W_{d,n}(a_1,\dotsc,a_{d})$ in \cite[Def.~2.2]{Goo11}, where $d = \deg p$ and $a_1,\dotsc,a_{d}$ are the roots of $p$, with multiplicity.  This algebra is generated by elements $e_i$, $s_i$, $x_j$, $1 \le i \le n-1$, $1 \le j \le n$.  It will be useful to use string diagrams to represent elements in $W_n(p,\Omega)$ in the usual way, numbering strands \emph{from right to left}.  We will use thick blue strings when drawing diagrams representing elements of $W_n(p,\Omega)$ in this way.  Thus, we represent the generators of $W_2(p,\Omega)$ as follows:
\[
    e_1 =\,
    \begin{tikzpicture}[centerzero]
        \draw[alg] (-0.2,-0.35) -- (-0.2,-0.3) arc(180:0:0.2) -- (0.2,-0.35);
        \draw[alg] (-0.2,0.35) -- (-0.2,0.3) arc(180:360:0.2) -- (0.2,0.35);
    \end{tikzpicture}
    \ ,\qquad
    s_1 =
    \begin{tikzpicture}[centerzero]
        \draw[alg] (0.2,-0.35) -- (-0.2,0.35);
        \draw[alg] (-0.2,-0.35) -- (0.2,0.35);
    \end{tikzpicture}
    \ ,\qquad
    x_1 =\,
    \begin{tikzpicture}[centerzero]
        \draw[alg] (-0.2,-0.35) -- (-0.2,0.35);
        \draw[alg] (0.2,-0.35) -- (0.2,0.35);
        \singdot{0.2,0};
    \end{tikzpicture}
    \ ,\qquad
    x_2 =\,
    \begin{tikzpicture}[centerzero]
        \draw[alg] (-0.2,-0.35) -- (-0.2,0.35);
        \draw[alg] (0.2,-0.35) -- (0.2,0.35);
        \singdot{-0.2,0};
    \end{tikzpicture}
    \ .
\]

Define $I(p,\Omega)$ to be the kernel of the map
\[
    \kk[u] \to W_2(p,\Omega),\qquad g(u) \mapsto g(x_1)e_1,
\]
and let $f_{p,\Omega}$ be the unique monic generator of this ideal.  Equivalently, $f_{p,\Omega}$ is the unique monic polynomial of minimal degree such that $f_{p,\Omega}(x_1)e_1=0$.  Since $p(x_1)e_1=0$ by the definition of $W_2(p,\Omega)$, we know that $p \in I(p,\Omega)$, and so $f_{p,\Omega}$ divides $p$.

In $\AB$,  the adjunction relations (the third and fourth equalities in \cref{brauer}) imply, for $n \in \N$,
\begin{equation} \label{snake}
    \begin{tikzpicture}[centerzero]
        \draw (0,-0.7) -- (0,0.7);
        \multdot{0,0}{east}{n};
    \end{tikzpicture}
    \ =\
    \begin{tikzpicture}[centerzero]
        \draw (0.2,-0.7) -- (0.2,-0.3) arc(0:180:0.2) arc(360:180:0.2) -- (-0.6,0.3) arc(180:0:0.2) arc(180:360:0.2) -- (0.2,0.7);
        \multdot{0.2,0.3}{west}{n};
    \end{tikzpicture}
    \ .
\end{equation}
This shows that for any object $L$ in an $\AB$-module category, and any polynomial $f$, we have
\begin{equation}\label{bear2}
    \begin{tikzpicture}[centerzero]
        \draw (0,-0.4) -- (0,0.4);
        \multdot{0,0}{east}{f(x)};
        \draw[module] (0.3,-0.4) \botlabel{L} -- (0.3,0.4);
    \end{tikzpicture}
    = 0
    \qquad \iff \qquad
    \begin{tikzpicture}[centerzero]
        \draw (-0.2,-0.5) -- (-0.2,-0.3) arc(180:0:0.2) -- (0.2,-0.5);
        \draw (-0.2,0.5) -- (-0.2,0.3) arc(180:360:0.2) -- (0.2,0.5);
        \multdot{0.2,0.3}{west}{f(x)};
        \draw[module] (1.1,-0.5) \botlabel{L} -- (1.1,0.5);
    \end{tikzpicture}
    = 0.
\end{equation}
However, in the \emph{algebra} $W_2(p,\Omega)$, the implication \cref{bear2} may not hold, since the equality \cref{snake} does not make sense in this algebra.

The closest analogue to \cref{hemlock,cactus} that we have found in the literature is \cite[Th. 5.2]{Goo11}, more specifically, the implication $(1) \Rightarrow (4)$, the ``only if'' direction of \cref{goodman} below.  (Note that the proof in the published version of \cite[Th. 5.2]{Goo11} suggests that there was a misprint and the points labelled 1.\ and 2.\ should have been (3) and (4). We will use the latter notation which matches the proof.  Note also that the implication $(1) \Leftarrow (4)$, the ``if'' direction of \cref{goodman}, was already shown in \cite[Th.~5.5]{AMR06}.)  In our language, the implication $(1) \Leftrightarrow (4)$ of \cite[Th.~5.2]{Goo11} is the following result.

\begin{prop}[{\cite[Th.~5.2]{Goo11}}] \label{goodman}
    We have $f_{p,\Omega}=p$ if and only if $\OO_\Omega = \OO_p$, where $\OO_p$ is defined in \cref{metrobus}.
\end{prop}

Let us first show how one can deduce \cref{hemlock,cactus} from \cref{goodman}.  Suppose $L$ is a brick in an $\AB$-module category $\cR$, and define $\OO_L$ and $m_L$ as in \cref{sec:BrauerModules}.  Let $m_L(u) = (u-a_1)(u-a_2) \dotsm (u-a_{d_L})$ as a product of linear factors in the algebraic closure of $\kk$, and let $\ba = (a_1,a_2,\dotsc,a_{d_L})$.  It is straightforward to verify that we have a homomorphism of associative algebras from $W_2(m_L,\Omega_L)$ to $\End_\cR(\go^{\otimes 2} L)$ given by
\begin{equation} \label{onion}
    e_1 \mapsto\,
    \begin{tikzpicture}[centerzero]
        \draw (-0.2,-0.35) -- (-0.2,-0.3) arc(180:0:0.2) -- (0.2,-0.35);
        \draw (-0.2,0.35) -- (-0.2,0.3) arc(180:360:0.2) -- (0.2,0.35);
        \draw[module] (0.5,-0.35) \botlabel{L} -- (0.5,0.35);
    \end{tikzpicture}
    \ ,\qquad
    s_1 \mapsto
    \begin{tikzpicture}[centerzero]
        \draw[-] (0.2,-0.35) -- (-0.2,0.35);
        \draw[-] (-0.2,-0.35) -- (0.2,0.35);
        \draw[module] (0.5,-0.35) \botlabel{L} -- (0.5,0.35);
    \end{tikzpicture}
    \ ,\qquad
    x_1 \mapsto\,
    \begin{tikzpicture}[centerzero]
        \draw (-0.2,-0.35) -- (-0.2,0.35);
        \draw (0.2,-0.35) -- (0.2,0.35);
        \singdot{0.2,0};
        \draw[module] (0.5,-0.35) \botlabel{L} -- (0.5,0.35);
    \end{tikzpicture}
    \ ,\qquad
    x_2 \mapsto\,
    \begin{tikzpicture}[centerzero]
        \draw (-0.2,-0.35) -- (-0.2,0.35);
        \draw (0.2,-0.35) -- (0.2,0.35);
        \singdot{-0.2,0};
        \draw[module] (0.5,-0.35) \botlabel{L} -- (0.5,0.35);
    \end{tikzpicture}
    \ .
\end{equation}
By the definition of $m_L(u)$ as the minimal polynomial of the action of $x_1$ on $L$, the elements
\[
    \begin{tikzpicture}[centerzero]
        \draw (0,-0.4) -- (0,0.4);
        \multdot{0,0}{east}{n};
        \draw[module] (0.3,-0.4) \botlabel{L} -- (0.3,0.4);
    \end{tikzpicture},
    \qquad 0 \le n < d_L,
\]
are linearly independent.  By \cref{bear2}, this implies that the elements
\[
    \begin{tikzpicture}[centerzero]
        \draw (-0.2,-0.5) -- (-0.2,-0.3) arc(180:0:0.2) -- (0.2,-0.5);
        \draw (-0.2,0.5) -- (-0.2,0.3) arc(180:360:0.2) -- (0.2,0.5);
        \multdot{0.2,0.3}{east}{n};
        \draw[module] (0.5,-0.5) \botlabel{L} -- (0.5,0.5);
    \end{tikzpicture}
    \ ,\quad 0 \le n < d_L,
\]
are also linearly independent.  Since these are the images under the homomorphism \cref{onion} of $x_1^n e_1 \in W_2(m_L,\Omega_L)$, $0 \le n < d_L$, the latter elements are linearly independent as well.  Therefore, by the ``only if'' direction of \cref{goodman}, $\Omega_L$ is $\ba$-admissible, which is precisely the conclusion of \cref{hemlock,cactus}.

To illustrate the opposite direction, we will prove a stronger result than the ``only if'' direction of \cref{goodman}, describing precisely the polynomial $f_{p,\Omega}$.  In fact, we will show in \cref{jcc} that it is the same as the polynomial $m_{p,\OO_{\Omega}}$ calculated in \cref{UNAM}.  One of the key steps in establishing this relationship is the following result, which shows that the ideal $I(p,\Omega)$ satisfies a similar closure relation to  \cref{cdmx}.  Note that the proof is strongly analogous to \cref{cdmx}, but to carry out this proof in the context of cyclotomic BMW \emph{algebras}, we have to rotate all our diagrams so that the loop which before joined the bottom and top of the diagram now has both ends at the top, and we must add a cap at the bottom.

\begin{lem} \label{shwarma}
    Suppose $e_1 \ne 0$.  If $g \in I(p,\Omega)$, then
    \[
        \hat{g}(u) := \left( -u-\tfrac{1}{2} \right) g(-u) \OO_\Omega(-u) \in I(p,\Omega).
    \]
\end{lem}

\begin{proof}
    A priori, $\hat{g}(u) \in \kk\Laurent{u^{-1}}$.  However, by the argument of \cref{cerveza}, we have
    \[
        e_1 x_1^n g(x_1) e_1
        =
        \begin{tikzpicture}[centerzero]
            \draw[alg] (0.2,0) arc(0:360:0.2);
            \draw[alg] (0.2,0.6) -- (0.2,0.5) arc(360:180:0.2) -- (-0.2,0.6);
            \draw[alg] (0.2,-0.6) -- (0.2,-0.5) arc(0:180:0.2) -- (-0.2,-0.6);
            \multdot{0.2,0}{west}{x^n g(x)};
        \end{tikzpicture}
        =
        [\hat{g}(-u)]_{u^{-n}}\
        \begin{tikzpicture}[centerzero]
            \draw[alg] (0.2,0.4) -- (0.2,0.3) arc(360:180:0.2) -- (-0.2,0.4);
            \draw[alg] (0.2,-0.4) -- (0.2,-0.3) arc(0:180:0.2) -- (-0.2,-0.4);
        \end{tikzpicture}
        = [\hat{g}(-u)]_{u^{-n}} e_1
        = 0
        \qquad \text{for all } n > 0.
    \]
    Since $e_1 \neq 0$, the series $\hat{g}(u)$ must be a polynomial in $\kk[u]$.  Then, applying the argument of \cref{burrito}, we find that
    \[
        0 =\,
        \begin{tikzpicture}[centerzero]
            \draw[alg] (-0.2,0.6) to[out=down,in=up] (0.2,0.1) to[out=down,in=down,looseness=1.5] (-0.2,0.1) to[out=up,in=down] (0.2,0.6);
            \draw[alg] (-0.2,-0.6) -- (-0.2,-0.5) to[out=up,in=up,looseness=1.5] (0.2,-0.5) -- (0.2,-0.6);
            \multdot{0.2,0.1}{west}{x g(x)};
        \end{tikzpicture}
        =\,
        \begin{tikzpicture}[anchorbase]
            \draw[alg] (-0.2,-0.6) -- (-0.2,-0.35) to[out=up,in=up,looseness=1.5] (0.2,-0.35) -- (0.2,-0.6);
            \draw[alg] (-0.2,0.6) -- (-0.2,0.35) to[out=down,in=down,looseness=1.5] (0.2,0.35) -- (0.2,0.6);
            \multdot{0.2,0.35}{west}{\hat{g}(x) - \frac{1}{2} g(x)};
        \end{tikzpicture}
        \,
        = \left( \hat{g}(x_1) - \tfrac{1}{2} g(x_1) \right)e_1.
    \]
    and so $\hat{g}- \frac{1}{2}{g}\in I(p,\Omega)$.  Therefore, $\hat{g} =(\hat{g}- \frac{1}{2}{g})+ \frac{1}{2}{g}\in I(p,\Omega)$.
\end{proof}

We can now give a characterization of $f_{p,\Omega}$; compare with \cref{mexico}.

\begin{theo} \label{noodles}
    We have that $e_1 \neq 0$ in $W_2(p,\Omega)$ if and only if $\OO_\Omega = \OO_g$ for some positive-degree polynomial $g$ dividing $p$.  If this condition is satisfied, then $\OO_\Omega = \OO_{f_{p,\Omega}}$, and $f_{p,\Omega}$ is divisible by any polynomial $g$ that divides $p$ and satisfies $\OO_\Omega = \OO_g$.
\end{theo}

\begin{proof}
    First, suppose $e_1 \ne 0$.  We claim that $\OO_\Omega = \OO_f$ for $f = f_{p,\Omega}$.  Indeed, by \cref{shwarma}, $f$ divides $\hat{f}$.  Then an argument analogous to that in the proof of \cref{hemlock} shows that $\hat{f}(u) = ((-1)^{1+\deg f}u - \frac{1}{2})f(u)$ and that $\OO_\Omega = \OO_f$.

    It remains to show that $f$ is divisible by any polynomial $g$ that divides $p$ and satisfies $\OO_\Omega = \OO_g$.  Consider the cyclotomic Brauer category $\CB(g,\OO_g)$.  We have a natural homomorphism $W_2(p,\Omega) \to \End_{\CB(g,\OO_g)}(\go^{\otimes 2} L(g,\OO_g))$.  Since the image of $f(x_1)e_1$ under this map is zero, we have
    \[
        \begin{tikzpicture}[centerzero]
            \draw (-0.2,-0.5) -- (-0.2,-0.3) arc(180:0:0.2) -- (0.2,-0.5);
            \draw (-0.2,0.5) -- (-0.2,0.3) arc(180:360:0.2) -- (0.2,0.5);
            \multdot{0.2,0.3}{west}{f(x)};
            \draw[module] (1.1,-0.5) \botlabel{L(g,\OO_g)} -- (1.1,0.5);
        \end{tikzpicture}
        = 0
        \quad\overset{\cref{bear2}}{\implies}\quad
        \begin{tikzpicture}[centerzero]
            \draw (0,-0.4) -- (0,0.4);
            \multdot{0,0}{east}{f(x)};
            \draw[module] (0.3,-0.4) \botlabel{L(g,\OO_g)} -- (0.3,0.4);
        \end{tikzpicture}
        = 0.
    \]
    On the other hand, by \cref{RS-thC}, in $\CB(g,\OO_g)$, the morphism $\dotstrand \colon \go L(g,\OO_g)\to \go L(g,\OO_g)$ has minimal polynomial $g$.  Thus, we must have that $f$ is divisible by $g$, as desired.
\end{proof}

We can now give an explicit description of $f_{p,\Omega}$.

\begin{prop} \label{jcc}
    We have
    \[
        f_{p,\Omega}(u)=m_{p,\OO_{\Omega}}(u)
        =
        \begin{cases}
            \frac{\gcd(p(u),\hat{p}(u))}{u} & \text{if $\gcd((u-1/2)p(u),\hat{p}(u))$ has odd degree}, \\
            \gcd(p(u),\hat{p}(u)) & \text{otherwise}.
        \end{cases}
    \]
\end{prop}

\begin{proof}
    \Cref{mexico,noodles} show that $m_{p,\OO_{\Omega}}$ and $f_{p,\Omega}$ are uniquely characterized by the same property, proving the result.
\end{proof}

The polynomial $f_{p,\Omega}$ is also studied by Goodman \cite[\S 5]{Goo12}; in that paper, it is denoted $p_0=(u-u_1)\cdots (u-u_d)$.  The fact that $\OO_\Omega = \OO_{f_{p,\Omega}}$ when $\deg f_{p,\Omega} > 0$ is shown in \cite[Lem.~5.5(2)]{Goo12}.  In \cite{Goo12}, the condition that $\deg f_{p,\Omega}>0$ is called \emph{semi-admissibility} of the data $(p,\Omega)$.  However, to the best of our knowledge, no \emph{explicit} description of $f_{p,\Omega}$ has appeared in the literature.

Another connection between the results of \cref{sec:cyclotomicBrauer} and degenerate cyclotomic BMW algebras can be found by considering the natural map
\[
    \alpha_p \colon W_n(p,\Omega)\to \End_{\CB(p,\OO_\Omega)}(\go^n).
\]

\begin{prop} \label{alpha-iso}
    The map $\alpha_p$ is surjective, with kernel generated, as a two-sided ideal, by $m_{p,\Omega}(x_1)$.  In particular, $\alpha_p$ is an isomorphism if and only if $\OO_p = \OO_\Omega$.
\end{prop}

\begin{proof}
    The map $\alpha_p$ factors through the canonical map $\pi_W \colon W_n(p,\Omega) \to W_n(m,\Omega)$ for $m=m_{p,\Omega}$, since
    \begin{equation*}
    	\alpha(m_{p,\Omega}(x_1))=
        \begin{tikzpicture}[centerzero]
            \draw (0,-0.4) -- (0,0.4);
            \multdot{0,0}{east}{m_{p,\Omega}(x)};
            \draw[module] (0.7,-0.4) \botlabel{L(p,\OO_\Omega)} -- (0.7,0.4);
        \end{tikzpicture}
        = 0,
    \end{equation*}
    by the definition of $m_{p,\Omega}$.  In fact, we have the following commutative diagram:
    \[
        \begin{tikzcd}
        	{ W_n(p,\Omega)} & {\End_{\CB(p,\OO_\Omega)}(\go^n)} \\
        	{ W_n(m,\Omega)} & {\End_{\CB(m,\OO_\Omega)}(\go^n)}
        	\arrow["{\alpha_p}", from=1-1, to=1-2]
        	\arrow["{\pi_W}"', from=1-1, to=2-1]
        	\arrow["{\pi_{\CB}}", from=1-2, to=2-2]
        	\arrow["{\alpha_m}"', from=2-1, to=2-2]
        \end{tikzcd}
    \]
    By \cite[Th.~C]{RS19}, the map $\alpha_m$ is an isomorphism, as is $\pi_{\CB}$, by \cref{tacos}.  Thus, the homomorphisms $\alpha_p$ and $\pi_W$ are intertwined by the induced isomorphism $W_n(m,\Omega) \cong \End_{\CB(p,\OO_\Omega)}(\go^n)$.  Therefore, the result follows from the fact that $\pi_W$ is surjective, with kernel generated by $m_{p,\Omega}(x_1)$.
\end{proof}

We can understand the map $\alpha_p$ better by considering its interaction with the two-sided ideal $E_{p,\Omega}$ of $W_n(p,\Omega)$ generated by $e_1$.  Since $s_is_{i+1}e_is_{i+1}s_i=e_{i+1}$, the ideal $E_{p,\Omega}$ is also generated by $e_i$ for any $i$, or equivalently by the set of all $e_i$.  In terms of diagrams, $E_{p,\Omega}$ is the ideal spanned by all diagrams that factor through $\go^{\otimes r}$ for some $r<n$; this set is manifestly a two-sided ideal.   By \cite[Prop.~7.2]{AMR06}, applied with $f=0$, the quotient $W_n(p,\Omega)/E_{p,\Omega}$ is isomorphic to $H_n^p$, the degenerate cyclotomic Hecke algebra for $p$.  Thus, for the polynomials $p$ and $m=m_{p,\Omega}$, we have short exact sequences compatible with projection:
\[
    \begin{tikzcd}
    	0 & {E_{p,\Omega}} & {W_n(p,\Omega)} & {H_n^p} & 0 \\
    	0 & {E_{m,\Omega}} & { W_n(m,\Omega)} & {H^m_n} & 0
    	\arrow[from=1-1, to=1-2]
    	\arrow[from=1-2, to=1-3]
    	\arrow["{\pi_E}", from=1-2, to=2-2]
    	\arrow[from=1-3, to=1-4]
    	\arrow["{\pi_{W}}", from=1-3, to=2-3]
    	\arrow[from=1-4, to=1-5]
    	\arrow["{\pi_H}", from=1-4, to=2-4]
    	\arrow[from=2-1, to=2-2]
    	\arrow[from=2-2, to=2-3]
    	\arrow[from=2-3, to=2-4]
    	\arrow[from=2-4, to=2-5]
    \end{tikzcd}
\]
By \cite[Prop.~5.11]{Goo12}, the map $\pi_E$ is an isomorphism, so the kernel of $\pi_W$ projects isomorphically to the kernel of
$ \pi_H$.

%=====================================================================
\section{Cyclotomic Kauffman categories\label{sec:cyclotomicKauffman}}
%=====================================================================

In this section, we consider the cyclotomic Kauffman categories introduced in \cite{GRS22}.  These are a categorical analogue of the (nondegenerate) cyclotomic BMW algebras, in the same way that the cyclotomic Brauer categories are a categorical analogue of the degenerate cyclotomic BMW algebras.   Throughout this section, we assume that $\kk$ is a field.  Our discussion in this section is parallel to that of \cref{sec:cyclotomicBrauer}.

Fix a monic polynomial $p(u) \in \kk[u]$ with nonzero constant term, and a power series
\[
    \rOO(u) = \sum_{r=0}^\infty \rOO^{(r)} u^{-r} \in t + u^{-1} \kk \llbracket u^{-1} \rrbracket.
\]
In light of \cref{infgrassk}, define
\begin{equation} \label{romania}
    \lOO(u) = \sum_{r=0}^\infty \lOO^{(r)} u^{-r}
    := \rOO(u)^{-1}
    \in t^{-1} + u^{-1} \kk \llbracket u^{-1} \rrbracket.
\end{equation}
Let $\cJ(p,\rOO)$ be the left tensor ideal of $\AK$ generated by
\[
    \multdotstrand[black]{east}{p(x)},
    \qquad
    \left[ \bubblegenr[black]{u} \right]_{u^{-r}} - \rOO^{(r)} 1_\one,
    \qquad r \in \N.
\]
It then follows from \cref{infgrassk,romania} that
\[
    \left[ \bubblegenl[black]{u} \right]_{u^{-r}} - \lOO^{(r)} 1_\one \in \cJ(p,\rOO),\qquad r \in \N.
\]
We define the corresponding \emph{cyclotomic Kauffman category}
\[
    \CK(p,\rOO) := \AK/\cJ(p,\rOO).
\]

\begin{rem}
    Recall, from \cref{glass}, the identification of $\AK$ with the reverse of the affine Kauffman category of \cite{GRS22}. Our use of the term \emph{cyclotomic Kauffman category} is slightly more general than that of \cite[Def.~1.8]{GRS22}, since we do not impose an analogue of \cite[Def.~1.7(2)]{GRS22} on our parameters.  See \cref{sec:admissibleKauffman} for further discussion of these conditions.
\end{rem}

Let $L(p,\rOO)$ denote the image of $\one$ in the quotient $\CK(p,\rOO)$.  The category $\CK(p,\rOO)$ is a left module category over $\AK$, and it is generated under this action by $L(p,\rOO)$.  The element $L(p,\rOO)$ is a brick if it is nonzero, since, by \cite[Th.~1.6]{GRS22}, all its endomorphisms are polynomials in the bubbles, which have all been specialized to scalars.  Recall, from \cref{sec:KauffmanModules}, that $J_{L(p,\rOO)}$ is the kernel of the algebra homomorphism \cref{pear}, and $m_{L(p,\rOO)}$ is the unique generator of $J_{L(p,\rOO))}$ that lies in $\kk[u]$, is monic, and has nonzero constant term.  Recall also the definition \cref{metrobusKr} of $\rOO_f$.

\begin{prop}[{\cite[Th.~1.12]{GRS22}}] \label{badminton}
    Suppose $f \in \kk[u]$ satisfies the conditions in \cref{sneeze}.  The endomorphism ring $\End(\gok L(f,\rOO_f))$ in the category $\CK(f,\rOO_f)$ is $\C \left[ \dotstrand[black] \right] / ( f ( \dotstrand[black] ) )$.
\end{prop}

In fact, \cite[Th.~1.12]{GRS22} describes a basis for \emph{all} morphism spaces in $\CK(f,\rOO_f)$, but for our purposes, we only need to know the result for endomorphisms of $\gok L(f,\rOO_f)$.

\begin{proof}
    Suppose $f(u) = (u-a_1)(u-a_2) \dotsm (u-a_{d})$ as a product of linear factors in the algebraic closure of $\kk$, and let $\ba = (a_1,a_2,\dotsc,a_{d})$.  The category denoted $\CK(\omega,\ba)$ in \cite{GRS22} is denoted $\CB(f,\rOO_f)$ in our language, under the assumption that $\omega$ is $\ba$-admissible.  As we explain in \cref{ladder} below, $\ba$-admissibility of $\omega$ corresponds to the fact that we have chosen $\rOO_f$ (as opposed to some general $\rOO$).  The condition in \cref{sneeze}\cref{sneeze2} corresponds to \cite[Assumption~1.9]{GRS22}.  Thus, the result follows from \cite[Th.~1.12]{GRS22}.
\end{proof}

\begin{theo} \label{china}
    The cyclotomic Kauffman category $\CK(p,\rOO)$ is not the zero category if and only if $\rOO = \rOO_f$ for some positive-degree monic polynomial $f$ dividing $p$ and satisfying the conditions of \cref{sneeze}.  If such a polynomial exists, then $\rOO = \rOO_{m_{L(p,\rOO)}}$, and $m_{L(p,\rOO)}$ is divisible by any polynomial $f$ that divides $p$, satisfies $\rOO = \rOO_f$, and satisfies the conditions of \cref{sneeze}.
\end{theo}

\begin{proof}
    First suppose that $\rOO = \rOO_f$ for some positive-degree polynomial $f$ dividing $p$ and satisfying the conditions of \cref{sneeze}.  Then we have the quotient functor $\CK(p,\rOO) \to \CK(f,\rOO)$.  \Cref{badminton} implies that $\CK(f,\rOO)$ is not the zero category, and hence $\CK(p,\rOO)$ is not the zero category.  For the other direction, suppose that $\CK(p,\rOO)$ is not the zero category.  Then $m_{L(p,\rOO)}$ has positive degree, and \cref{backhome} implies that $\rOO = \rOO_{m_{L(p,\OO)}}$ and that $m_{L(p,\rOO)}$ satisfies the conditions in \cref{sneeze}.

    Now let $f$ be a monic polynomial dividing $p$, satisfying the conditions in \cref{sneeze}, and satisfying $\rOO = \rOO_f$.  Then we have the quotient functor $\CK(p,\rOO) \to \CK(f,\rOO)$.  It follows that $m_{L(f,\rOO)}$ divides $m_{L(p,\rOO)}$.  On the other hand, \cref{badminton} implies that the endomorphisms $\multdotstrand[black]{east}{n}$, $0 \le n \le \deg f - 1$, are linearly independent in $\CK(f,\rOO)$.  It follows that $f = m_{L(f,\rOO)}$, and so $f$ divides $m_{L(p,\rOO)}$, as claimed.
\end{proof}

\begin{cor} \label{soup}
    If $\CK(p,\rOO)$ is not the zero category, then it is isomorphic to
    \[
        \CK(m) := \CK(m,\rOO_m)
        \qquad \text{for } m = m_{L(p,\rOO)}.
    \]
\end{cor}

If $L$ is a brick in an $\AK$-module category $\cR$, then it follows from \cref{backhome} that the action of $\AB$ on $L$ factors through $\CK(m_L)$.

We assume for the remainder of this section that
\begin{center}
    $\CK(p,\rOO)$ is not the zero category.
\end{center}
It follows from \cref{china} that $m_{L(p,\rOO)}$ is the unique monic polynomial of maximal degree in the set of polynomials $f$ that divide $p$ and satisfy $\rOO = \rOO_f$.  Our next goal is to describe $m_{L(p,\rOO)}$ explicitly.  To simplify notation, set
\[
    L = L(p,\rOO),\qquad
%    J = J_{L(p,\rOO)},\qquad
    m = m_{L(p,\rOO)}.
\]
It follows immediately from the definition of $\CK(p,\rOO)$ that $\rOO = \rOO_L = \rOO_m$.

For $f,g \in \kk[u]$ satisfying the conditions of \cref{sneeze}, define
\begin{equation} \label{croak}
    H_{f,g} := \left( u - q^{\frac{\epsilon_1(g)-\epsilon_1(f)}{2}} \right)^{\frac{|\epsilon_1(g)-\epsilon_1(f)|}{2}}
    \left( u + q^{\frac{\epsilon_2(g) - \epsilon_2(f)}{2}} \right)^{\frac{|\epsilon_2(g)-\epsilon_2(f)|}{2}},
\end{equation}
where $\epsilon_1$ and $\epsilon_2$ are as in \cref{what-is-epsilon}.  Note that
\begin{equation} \label{suitcase}
    \frac{\check{H}_{f,g}}{H_{f,g}}
    = \frac{\left( u-q^{\epsilon_1(f)} \right) \left( u + q^{\epsilon_2(f)} \right)}{\left( u - q^{\epsilon_1(g)} \right) \left( u + q^{\epsilon_2(g)} \right)}.
\end{equation}

\begin{lem} \label{belgium}
    Suppose that $f,g \in \kk[u]$ satisfy the conditions of \cref{sneeze}, and that $g$ is divisible by $f$.  Then
    \[
        \rOO_g = \rOO_f
        \iff
        g = f H_{f,g} \gamma
    \]
    for some $\gamma \in \kk[u]$ of even degree satisfying $\check{\gamma} = \gamma$ and $\gamma(0)=1$.
\end{lem}

\begin{proof}
    We have
    \begin{equation} \label{rabbit}
        \rOO_g = \rOO_f
        \overset{\cref{metrobusKr-1}}{\iff}
        \frac{g(u)}{f(u)} \frac{\left( u-q^{\epsilon_1(f)} \right) \left( u + q^{\epsilon_2(f)} \right)}{\left( u - q^{\epsilon_1(g)} \right) \left( u + q^{\epsilon_2(g)} \right)} = \frac{\check{g}(u)}{\check{f}(u)}
        \overset{\cref{suitcase}}{\iff} \frac{g}{f H_{f,g}} = \frac{\check{g}}{\check{f} \check{H}_{f,g}}.
    \end{equation}
    Since
    \[
        \frac{\check{g}}{\check{f}}(u)
        = \frac{g(u^{-1}) u^{\deg g}}{f(u^{-1}) u^{\deg f}}
        = \frac{g}{f}(u^{-1}) u^{\deg g - \deg f}
        \in \kk[u],
    \]
    it follows from \cref{rabbit,suitcase} that, when $\epsilon_1(f) \ne \epsilon_1(g)$, the polynomial $g/f$ is divisible by $u-q^{\epsilon_1(g)} = u - q^{\frac{\epsilon_1(g)-\epsilon_1(f)}{2}}$.  Similarly, when $\epsilon_2(f) \ne \epsilon_2(g)$, the polynomial $g/f$ is divisible by $u + q^{\epsilon_2(g)} = u + q^{\frac{\epsilon_2(g)-\epsilon_2(f)}{2}}$.  Thus, $\gamma := g/f H_{f,g}$ is a polynomial satisfying $\check{\gamma} = \gamma$ and $g = f H_{f,g} \gamma$.

    It follows from \cref{f-degree,croak} that $\deg \gamma \equiv 0 \pmod 2$.  Finally, we compute
    \begin{align*}
        \gamma(0)
        &= \frac{g(0)}{f(0) H_{f,g}(0)}
        \\
        &\overset{\mathclap{\cref{f-degree}}}{\underset{\mathclap{\cref{croak}}}{=}}\ \epsilon_1(g) \epsilon_2(g) q^{\frac{\epsilon_1(g)-\epsilon_1(f)+\epsilon_2(g)-\epsilon_2(f)}{2}}
        \left(- q^{\frac{\epsilon_1(g)-\epsilon_1(f)}{2}} \right)^{-\frac{|\epsilon_1(g)-\epsilon_1(f)|}{2}}
        \left( q^{\frac{\epsilon_2(g) - \epsilon_2(f)}{2}} \right)^{-\frac{|\epsilon_2(g)-\epsilon_2(f)|}{2}}
        \\
        &= 1. \qedhere
    \end{align*}
\end{proof}

Recall, from \cref{blackstar}, the polynomial
\begin{equation} \label{blackstar2}
    \hat{p}(u)
    = (1 - zu - u^2) p(u^{-1}) u^{\deg p} \rOO(u^{-1})
%    \overset{\cref{beijing}}{=} (1-u^2) p(0) \check{p}(u) \rOO(u^{-1})
    \in \kk[u].
\end{equation}
By \cref{blackcdmx}, we have $ \hat{p}\in J_L$.  Since we are seeking a generator of the ideal $J_L$, it is natural whenever we have a pair of elements of $J_L$ to consider their greatest common denominator, which also lies in $J_L$.  For certain technical reasons, we want to instead consider
\[
    R = \gcd \big( (u^2-zu-1)p(u),\hat{p}(u) \big)\in J_L.
\]
By definition, the minimal polynomial $m$ must divide $R$.  It follows from \cref{metrobusKr-1} that the numerator and denominator of $\rOO = \rOO_m$ vanish to the same order at $1$ and at $-1$.  Therefore
\begin{equation} \label{sandal}
    \text{$p$, $\hat{p}$ and $R$ vanish to the same order at $1$ and vanish to the same order at $-1$}.
\end{equation}

\begin{lem} \label{downpour}
    We have
    \begin{equation} \label{stroop}
        \rOO = t \frac{\check{R}}{R}
    \end{equation}
    and $R(0) = \pm t$.
\end{lem}

\begin{proof}
    Let $d = \deg p$, and let $\sim$ denote the relation on the ring of rational functions of differing by a factor of $\kk^\times$.  We have
    \begin{multline*}
        \frac{\check{\hat{p}}(u)}{(u^2-zu-1)p(u)}
        \overset{\cref{beijing}}{\sim} \frac{\hat{p}(u^{-1}) u^{d+2}}{(u^2-zu-1)p(u)}
        \overset{\cref{blackstar2}}{=} \rOO(u)
        = \rOO(u^{-1})^{-1}
        \\
        \overset{\cref{blackstar2}}{=} \frac{(1-zu-u^2)p(u^{-1})u^d}{\hat{p}(u)}
        \overset{\cref{beijing}}{\sim} \frac{(u^2+zu-1) \check{p}(u)}{\hat{p}(u)}.
    \end{multline*}
    By \cref{hack}, we have
    \[
        \rOO(u)
        \sim \frac{\gcd \left( \check{\hat{p}}(u), (u^2+zu-1) \check{p}(u) \right)}{\gcd \left( \hat{p}(u), (u^2-zu-1) p(u) \right)}.
    \]
    For monic $a,b,c \in \kk[u]$ with nonzero constant terms, we see that if $a$ divides $b$ and $c$, then $\check{a}$ divides $\check{b}$ and $\check{c}$.  It follows that the map $a \mapsto \check{a}$ commutes with taking $\gcd$.  Therefore,
    \[
        \gcd \left( \check{\hat{p}}(u), (u^2+zu-1) \check{p}(u) \right) = \check{R}(u).
    \]
    Thus, $\rOO \sim \check{R}/R$.  Comparing leading terms, we see that \cref{stroop} holds.

    For the final assertion, note that
    \[
        \rOO(0) = \rOO_m(0) \overset{\cref{metrobusKr-1}}{=} t q^{\epsilon_1(m)+\epsilon_2(m)} \frac{\check{m}(0)}{m(0)}
        = \frac{t q^{\epsilon_1(m)+\epsilon_2(m)}}{m(0)^2}
        \overset{\cref{f-degree}}{=} t^{-1}.
    \]
    Thus
    \[
        t^{-1} = t \frac{\check{R}(0)}{R(0)} = \frac{t}{R(0)^2} \implies R(0)^2 = t^2 \implies R(0) = \pm t.
        \qedhere
    \]
\end{proof}

\begin{theo} \label{gooigi}
   If $\CK(p,\rOO)$ is not the zero category, then
   \[
        m_{L(p,\rOO)}
        =
        \begin{dcases}
            \gcd(p,\hat{p}) & \text{if $\deg R$ is odd and $R(0)=t$}, \\
            \frac{\gcd(p,\hat{p})}{u^2-1} & \text{if $\deg R$ is odd and $R(0)=-t$}, \\
            \frac{\gcd(p,\hat{p})}{u+1} & \text{if $\deg R$ is even and $R(0)=t$}, \\
            \frac{\gcd(p,\hat{p})}{u-1} & \text{if $\deg R$ is even and $R(0)=-t$}.
        \end{dcases}
    \]
\end{theo}

\begin{proof}
    Let
    \begin{equation} \label{scone}
        R_1(u)
        =
        \begin{dcases}
            R(u) & \text{if $\deg R$ is odd and $R(0)=t$}, \\
            \frac{R(u)}{u^2-1} & \text{if $\deg R$ is odd and $R(0)=-t$}, \\
            \frac{R(u)}{u+1} & \text{if $\deg R$ is even and $R(0)=t$}, \\
            \frac{R(u)}{u-1} & \text{if $\deg R$ is even and $R(0)=-t$}.
        \end{dcases}
    \end{equation}
    Since $m$ divides $p$ and $\hat{p}$, we have that $m$ divides $R$ and thus $m$ also divides $R_2 := R_1(u)(u^2-1)^2$.  Note that $R_2$ is monic, has odd degree, and $R_2(0) = R_1(0) = t$.  Thus, $R_2$ satisfies the conditions of \cref{sneeze} and, by \cref{what-is-epsilon}, we have $\epsilon_1(R_2)=1$, $\epsilon_2(R_2)=-1$.  Since
    \[
        \rOO_{R_2}
        = t \frac{\check{R}_2}{R_2}
        = t \frac{\check{R}}{R}
        \overset{\cref{stroop}}{=} \rOO
        = \rOO_m,
    \]
    \cref{belgium} implies that $R_2 = m H_{m,R_2} \gamma$ where $\gamma \in \kk[u]$ has even degree, $\check{\gamma} = \gamma$, and $\gamma(0)=1$.

    It follows from \cref{sandal} that $R_2$ vanishes at strictly higher order at $\pm 1$ than $p$ and $\hat{p}$.  Since $m$ divides $p$ and $\hat{p}$, the vanishing order of $m$ at $\pm 1$ is less than or equal to the vanishing order of $p$ and $\hat{p}$.  Since $H_{m,R_2}$ does not vanish at $\pm 1$, it follows that $\gamma$ vanishes at $\pm 1$, and thus is divisible by $(u^2-1)^2$, since $\check{\gamma} = \gamma$.  This shows that $R_1 = \frac{R_2}{(u^2-1)^2} = m H_{m,R_2} \frac{\gamma}{(u^2-1)^2}$ is a polynomial divisible by $m$.

    The polynomial $R_1$ is monic, has odd degree, and $R_1(0) = t$.  Thus, $R_1$ satisfies the conditions of \cref{sneeze} and, by \cref{what-is-epsilon}, we have $\epsilon_1(R)=1$, $\epsilon_2(R)=-1$.  Therefore,
    \[
        \rOO \overset{\cref{stroop}}{=} t \frac{\check{R}}{R} = t \frac{\check{R}_1}{R_1} \overset{\cref{metrobusKr-1}}{=} \rOO_{R_1}.
    \]

    Consider $R_3=\gcd(R_1,p)$, which is divisible by $m$.  Let $a \in \{1,u^2-1,u+1,u-1\}$ be the denominator appearing in \cref{scone}.  Then
    \[
        R_1 = \gcd \left( (u^2-zu-1) \frac{p}{a}, \frac{\hat{p}}{a} \right)
        = \frac{1}{a} \gcd \left( (u^2-zu-1) p, \hat{p} \right).
    \]
    Thus,
    \[
        R_3 = \gcd \left(p, (u^2-zu-1) \frac{p}{a}, \frac{\hat{p}}{a} \right)
        = \frac{1}{a} \gcd \left( p, \hat{p} \right),
    \]
    where the last equality follows from the fact that $u^2-zu-1$ is not divisible by $a$.  Therefore,
    \[
        \frac{R_1}{R_3}
        = \frac{\gcd \left( (u^2-zu-1) p, \check{p} \right)}{\gcd \left( p, \check{p} \right)},
    \]
    which clearly divides $u^2-zu-1$.  Equivalently, $R_1/R_3 \in \{1, u-q, u+q^{-1}, u^2-uz-1\}$.  In fact, one can easily confirm that $R_1/R_3=H_{R_3,R_1}$.  It thus follows immediately from \cref{belgium} that $\rOO_{R_3} = \rOO$.
    \details{
        If $R_1/R_3 = u^2-uz-1 = (u-q)(u+q^{-1})$, then $R_3(0) = -R_1(0) = -t$.  Thus, $R_3$ satisfies the conditions of \cref{sneeze} and, by \cref{what-is-epsilon}, we have $\epsilon_1(R_3)=-1$ and $\epsilon_2(R_3)=1$.  Therefore,
        \[
            \rOO = r \frac{\check{R}_1}{R_1}
            = \frac{(u-q^{-1})(u+q)}{(u-q)(u+q^{-1})} \frac{\check{R}_3}{R_3}
            \overset{\cref{metrobusKr-1}}{=} \rOO_{R_3}.
        \]
        The other cases are analogous.
    }

    Since $R_3$ divides $p$, \cref{china} implies that $m$ is divisible by $R_3$.  As both polynomials are monic, it follows that $m=R_3$, completing the proof.
\end{proof}

\begin{cor}
    The category $\CK(p,\rOO)$ is not the zero category if and only if all three of the following conditions are satisfied:
    \begin{enumerate}
		\item we have $\rOO = \rOO_f$ for some $f$ satisfying the conditions of \cref{sneeze};
		\item the power series $\hat{p}$ is a polynomial;
		\item the rational function $R_1$ defined in \cref{scone} is a polynomial of strictly positive degree.
	\end{enumerate}
\end{cor}

\begin{proof}
    The ``only if'' direction follows from the proof of \cref{gooigi}.  For the ``if'' direction, we note that the given conditions imply that the conditions of \cref{china} are satisfied with $f=R_1$.
\end{proof}

Following the methods of \cref{sec:admissibleBrauer}, one can translate the results of the current section from statements about cyclotomic Kauffman \emph{categories} to statements about cyclotomic BMW \emph{algebras}.  Since the treatment is parallel, and we consider the category point of view to be more natural, we do not give the details of this translation here.

%=============
% Bibliography
%=============

\bibliographystyle{alphaurl}
\bibliography{BKBubbles}

\end{document}